\documentclass[11pt,a4paper]{amsart}
\usepackage[english]{babel}
\usepackage{amsfonts,amsmath,amsthm,amssymb,latexsym}

\usepackage[active]{srcltx}

\usepackage[vcentering,dvips]{geometry}

\newtheorem{teo}{Theorem}[section]
\newtheorem{prop}[teo]{Proposition}
\newtheorem{lema}[teo]{Lemma}
\newtheorem{coro}[teo]{Corollary}

\newtheorem{remark}[teo]{Remark}

\newtheorem{defi}[teo]{Definition}

\numberwithin{equation}{section}

\newcommand{\Z}{\mathbb Z}
\newcommand{\Q}{\mathbb Q}
\newcommand{\K}{\mathbb K}

\newcommand{\C}{\mathbb C}
\newcommand{\F}{\mathbb F}
\newcommand{\A}{\mathbb A}

\newcommand{\bfs}{\boldsymbol}

\newcommand{\fq}{\F_{\hskip-0.7mm q}}

\newcommand{\fp}{\F_{\hskip-0.7mm p}}
\newcommand{\cfp}{\overline{\F}_{\hskip-0.35mm p}}

\def\ifm#1#2{\relax \ifmmode#1\else#2\fi}

\newcommand{\klk}    {\ifm {,\ldots,} {$,\ldots,$}}

\newcommand{\SO}{\mathcal{O}^{\sim}}

\begin{document}
\title{On the bit complexity of polynomial system solving}
\author[N. Gim\'enez]{
Nardo Gim\'enez${}^{1}$}
\author[G. Matera]{
Guillermo Matera${}^{1,2}$}

\address{${}^{1}$Instituto del Desarrollo Humano,
Universidad Nacional de Gene\-ral Sarmiento, J.M. Guti\'errez 1150
(B1613GSX) Los Polvorines, Buenos Aires, Argentina}
\email{\{agimenez,gmatera\}@ungs.edu.ar}

\address{${}^{2}$ National Council of Science and Technology (CONICET),
Ar\-gentina}

\thanks{The authors were partially supported by the grants
PIP CONICET 11220130100598, PIO CONICET-UNGS 14420140100027 and UNGS
30/3084.}

\date{\today}%

\begin{abstract}
We exhibit a probabilistic algorithm which solves a polynomial
system over the rationals defined by a reduced regular sequence. Its
bit complexity is roughly quadratic in the B\'ezout number of the
system and linear in its bit size. Our algorithm solves the input
system modulo a prime number $p$ and applies $p$--adic lifting. For
this purpose, we establish a number of results on the bit length of
a ``lucky'' prime $p$, namely one for which the reduction of the
input system modulo $p$ preserves certain fundamental geometric and
algebraic properties of the original system. These results rely on
the analysis of Chow forms associated to the set of solutions of the
input system and effective arithmetic Nullstellens\"atze.
\end{abstract}

\maketitle

%
%
\section{Introduction}
Solving polynomial systems defined over $\Q$ is a fundamental task
of computational algebraic geometry, which has been the subject of
intensive work for at least 40 years. Symbolic approaches to this
problem include Gr\"obner basis technology, triangular
decomposition, resultants, Macaulay matrices and Kronecker--like
algorithms (see, e.g., \cite{Mora05} and \cite{Mora15} for an
overview of the existing methods). The corresponding {\em arithmetic
complexity}, namely the number of arithmetic operations in $\Q$, has
been analyzed in, e.g., \cite{Lazard81}, \cite{Giusti89a},
\cite{DiFiGiSe91}, \cite{FiGiSm95}, \cite{GiHaHeMoMoPa97},
\cite{GiLeSa01}, \cite{Lecerf03} and \cite{DuLe08}, among others.
The complexity paradigm arising from these works is that polynomial
systems can be solved with a number of arithmetic operations which
is {\em polynomial} in the B\'ezout number of the system. This
conclusion nearly matches the lower bounds of \cite{CaGiHeMaPa03},
\cite{GiHeMaSo11} and \cite{BaHeMaMoPaRo16}, under the assumption
that the corresponding algorithms are ``geometrically robust'',
namely they are universal and allow the solution of certain
``limit'' problems.

On the other hand, less work has been done to analyze the {\em bit
complexity} of these algorithms. Concerning Gr\"obner bases, the
work \cite{HaLa11} by Hashemi and Lazard shows that
zero--dimensional Gr\"obner bases can be computed essentially in
polynomial time in the input size and $D^n$, where $n$ is the number
of unknowns and $D$ is the mean value of the degrees of the defining
polynomials. The bit complexity of Kronecker--like algorithms for
complete intersections is analyzed in, e.g., \cite{GiHaHeMoMoPa97}
and \cite{HaMoPaSo00}, where it is shown that it is polynomial in
the input size and certain invariant called the ``system degree''
(which is upper bounded by the B\'ezout number of the system).
Further, the recent work by Schost and Safey El Din \cite{SaSc16}
considers the bit complexity of multi--homogeneous zero--dimensional
systems and proves that such systems can solved with quadratic
complexity in the multi--homogeneous B\'ezout number and a
corresponding arithmetic analogue of it. Finally,
\cite{GiHaHeMoMoPa97} provides a lower bound on the bit size of the
output when ``standard'' representations are used.

This paper is devoted to analyze the bit complexity of a family of
Kronecker--like algorithms originally due to \cite{GiHeMoMoPa98} and
\cite{GiHaHeMoMoPa97}. We shall consider the improved version of
this algorithm due to \cite{GiLeSa01} (see also \cite{DuLe08}),
which we now discuss. Let $F_1,\ldots,F_r\in\Z[X_1,\ldots,X_n]$ be
polynomials which form a reduced regular sequence, that is,
$F_1,\ldots, F_r$ form a regular sequence and the ideal
$(F_1,\ldots,F_s)$ is radical for $1\le s\le r$. Denote by
$\mathcal{V}_s:=\mathcal{V}(F_1,\ldots,F_s)$ the affine subvariety
of $\C^n$ defined by $F_1,\ldots,F_s$ and by $\delta_s:=\deg
\mathcal{V}_s$ its degree for $1\le s\le r$. Let
$\mathcal{V}:=\mathcal{V}_r$ and $\delta:=\max_{1\le s\le
r}\delta_s$. The algorithm outputs a suitable ``parametrization'' of
a ``lifting fiber'' of $\mathcal{V}$, that is, a (zero--dimensional)
fiber defined over $\Q$ of a general linear projection
$\pi:\mathcal{V}\to\C^{n-r}$ defined over $\Q$. Such a
parametrization is called a ``Kronecker representation''. Several
works show that this constitutes a good representation of
$\mathcal{V}$, namely a ``solution'' of the system
$F_1=0,\ldots,F_r=0$, both from the numeric and the symbolic point
of view (see, e.g., \cite{HeKrPuSaWa00}, \cite{Schost03},
\cite{Lecerf03}, \cite{CaMa06a}, \cite{SoWa05}).

The computation of the Kronecker representation of such a lifting
fiber proceeds in $r$ stages. In the $s$th stage we compute a
Kronecker representation of a lifting fiber of $\mathcal{V}_{s+1}$
from one of $\mathcal{V}_s$. Following a suggestion of
\cite{GiLeSa01}, to keep the bit length of intermediate results
under control, these computations are performed modulo a prime
number $p$, followed by a step of $p$--adic lifting to recover the
integers which define the Kronecker representation of $\mathcal{V}$.
As a consequence, the determination of a prime number $p$ with
``good'' modular reduction is crucial to estimate the bit
complexity of the procedure. 

For our purposes, the modular reduction defined by a prime number
$p$ is ``good'', and the corresponding prime $p$ is called
``lucky'', if basic geometric and algebraic features of the variety
$\mathcal{V}_s$ and its defining ideal $(F_1,\ldots,F_s)$ are
preserved under modular reduction for $1\le s\le r$. Among them, we
may mention dimension, degree and generic smoothness. Further, our
algorithm also requires that the modular reduction of the lifting
fibers under consideration preserves dimension, degree and
non--ramification. Partial results in this direction have been
obtained in \cite{Schost00} (see also \cite{MeSa16}), on modular
reduction of smooth fibers of parametric families of
zero--dimensional varieties, and \cite{DaOsShSo15}, on modular
reduction of zero--dimensional varieties defined over $\Z$.
Unfortunately, these results are not enough for our purposes.

For the analysis of the bit length of lucky primes, we establish
conditions on the coefficients of linear forms defining a projection
$\pi_s:\mathcal{V}_s\to\C^{n-s}$, and the coordinates of a point
$\bfs p\in\C^{n-s}$, which imply that $\pi$ is ``general'' in the
sense above and $\bfs p$ defines a lifting fiber for $1\le s\le r$.
As we need to analyze both conditions for projections and fibers
defined over $\Z$, and their modular reductions, a natural framework
for this analysis is that of an affine variety defined over a
infinite perfect field $\K$. Our main result is the following (see
Proposition \ref{prop: finite_fiber_and_finite_morphism} and Theorem
\ref{th:_lifting_point_rank_and_finite_morphism}).
\begin{teo}\label{th: intro: conditions equidim V}
Let $V\subset\overline{\K}^n$ be an equidimensional variety defined
over $\K$ of dimension $n-s$ and degree $\delta_s$. Let
$\Lambda_{ij}$ $(1\le i\le n-s+1,1\le j\le n)$ and
$Z_1,\ldots,Z_{n-s}$ be indeterminates over $\K[V]$. Denote $\bfs
Z:=(Z_1\klk Z_{n-s})$, $\bfs \Lambda:=(\Lambda_{ij})_{1\le i\le
n-s+1,1\le j\le n}$, $\bfs \Lambda^*:=(\Lambda_{ij})_{1\le i\le
n-s,1\le j\le n}$ and $\bfs \Lambda_i:=
(\Lambda_{i1}\klk\Lambda_{in})$ for $1\le i\le n-s+1$. There exist
polynomials $A_{\scriptscriptstyle V}\in\K[\bfs\Lambda^*]$ and
$\rho_{\scriptscriptstyle V}\in\K[\bfs \Lambda,\bfs Z]$ such that
$\deg_{\bfs\Lambda_i}A_{\scriptscriptstyle V}=\delta_s$ $(1\le i\le
n-s)$, $\deg_{\bfs\Lambda_i}\rho_{\scriptscriptstyle V}\le
\delta_s(2\delta_s-1)$ $(1\le i\le n-s+1)$, $\deg_{\bfs
Z}\rho_{\scriptscriptstyle V}\le \delta_s(2\delta_s-1)$ and the
following properties hold: for any $\bfs\lambda\in\K^{(n-s+1)n}$ and
$\bfs p\in\K^{n-s}$ with $A_{\scriptscriptstyle
V}(\bfs\lambda^*)\rho_{\scriptscriptstyle V}(\bfs\lambda,\bfs
p)\not=0$, if $(Y_1\klk Y_{n-s+1}):=\bfs\lambda \bfs X$, then
\begin{enumerate}
\item the mapping $\pi: V \rightarrow
\mathbb{A}^{n-s}$ defined by $\bfs Y:=(Y_1, \dots, Y_{n-s})$ is a
finite morphism;
\item $Y_{n-s+1}$ induces a primitive element of the ring extension
$\K[\bfs Y]\hookrightarrow \K[V]$;
\item $\mbox{rank}_{\scriptscriptstyle \K[\bfs Y]}\K[V]=\delta_s$;
\item  $\bfs{p}$ is a lifting point of $\pi$ and $Y_{n-s+1}$
induces a primitive element of $\pi^{-1}(\bfs{p})$.
\end{enumerate}
\end{teo}

Our main technical tool is the analysis of the Chow form of $V$. A
similar analysis is obtained in \cite{CaMa06a} under stronger
assumptions, namely that $\K$ is a finite field $\fq$ and $V$ is an
absolutely--irreducible complete intersection.

Then we compare the conditions underlying Theorem \ref{th: intro:
conditions equidim V} for $\K=\Q$ and $\K=\cfp$, where $\fp$ is a
given prime field. This yields an integer multiple $\mathfrak{N}$ of
all primes $p$ which are not lucky in the sense above. We upper
bound the bit length of this integer $\mathfrak{N}$ using estimates
for heights of equidimensional varieties of \cite{DaKrSo13}, and
then obtain a lucky prime $p$ with ``low'' bit length. The following
statement summarizes our results on modular reduction (see Theorems
\ref{th:_simultaneous_noether_normalization} and \ref{th: app:
mathfrak_N_height}).
\begin{teo}\label{th: intro: good modular reduction}
Let $F_1\klk F_r\in\Z[X_1\klk X_n]$ be polynomials of degree at most
$d$ with coefficients of bit length at most $h$. Assume that
$F_1\klk F_r$ form a reduced regular sequence and denote
$\mathcal{V}_s:=V(F_1\klk F_s)\subset\C^n$ and $\delta_s:=\deg
\mathcal{V}_s$ for $1\le s\le r$. Let $\delta:=\max_{1\le s\le
r}\delta_s$. Let $\bfs\lambda\in \Z^{n^2}\setminus \{0\}$ and $\bfs
p:=(p_1\klk p_{n-1})\in \Z^{n-1}$ be randomly chosen elements with
entries of bit length $\mathcal{O}(n^2\delta^3)$. Let $(Y_1\klk
Y_n):=\bfs\lambda\bfs X$ and $\bfs p^s:=(p_1\klk p_{n-s})$ for $1\le
s\le r$.

Let $p$ be a random prime number of bit length
$\mathcal{O}^{\sim}\big(\log (nd^r h)\big)$. Denote by $F_{1,p}\klk
F_{r,p}$, $Y_{1,p}\klk Y_{n,p}$ and $\bfs p_p$ the corresponding
reductions modulo $p$. Then the following conditions are satisfied
for $1\le s\le r$ with probability at least $2/3$:
\begin{enumerate}
  \item the polynomials $F_{1,p}, \dots, F_{s,p}$ generate a radical
  ideal in $\cfp[\bfs X]$ and define an
  equidimensional variety $V_{s,p}\subset
  \cfp^n$ of
  dimension $n-s$ and degree $\delta_s$;
  \item the mapping $\pi_{s,p}:V_{s,p}\rightarrow
  \cfp^{n-s}$
  defined by $Y_{1,p},\dots,Y_{n-s,p}$ is a finite morphism,
  $\bfs p_p^s\in \fp^{n-s}$ is a lifting point of $\pi_{s,p}$,
  and $Y_{n-s+1,p}$ induces a primitive
  element of
  $\pi_{s,p}^{-1}(\bfs p_p^s)$;
  \item any $\bfs q\in \pi_{s,p}\bigl(\pi_{s+1,p}^{-1}(\bfs p^{s+1}_p)\bigr)$
  is a lifting point of $\pi_{s,p}$ and $Y_{n-s+1,p}$ induces a primitive
  element of $\pi_{s,p}^{-1}(\bfs q)$.
\end{enumerate}
\end{teo}

We observe that the analysis of lucky primes becomes much simpler if
only conditions (1) and (2) above are required. An analysis along
these lines can be deduced from \cite{Schost00} (compare with
\cite{MeSa16}). Nevertheless, condition (3), which is critical to
prove the correctness of our algorithm for solving the system
$F_1=0\klk F_r=0$, requires a significant extension of these
techniques.

Finally, we combine the algorithm of \cite{CaMa06a} with $p$--adic
lifting, as in \cite{GiLeSa01}, to obtain an algorithm for solving
the system $F_1=0\klk F_r=0$ with good bit complexity. We prove the
following result (see Theorem \ref{th:
bit_complexity_kronecker_solver} for a precise statement).
\begin{teo}\label{th: intro: bit complexity Kronecker}
Let $F_1\klk F_r$ be polynomials of $\Z[X_1\klk X_n]$ as in the
statement of Theorem \ref{th: intro: good modular reduction}. There
exists a probabilistic algorithm that takes as input an algorithm
evaluating $F_1\klk F_r$ with at most $L$ arithmetic operations, and
outputs a parametrization of a lifting fiber of $\mathcal{V}(F_1,
\dots,F_r)$ with
$
\mathcal{O}^{\sim}\big(n^{\mathcal{O}(1)}L\delta(d\delta+d^rh)\big)$
bit operations.
\end{teo}

The paper is organized as follows. In Section
\ref{notions_and_notations} we recall the notions and results of
algebraic geometry and commutative algebra we shall use, and discuss
the representation of multivariate polynomials by straight--line
programs and algebraic varieties by Kronecker representations. In
Section \ref{section: Noether normalization} we recall the notion of
Chow form of an equidimensional variety, discuss its basic
properties and obtain conditions (1)--(3) of Theorem \ref{th: intro:
conditions equidim V}. In Section \ref{section:
lifting_points_and_lifting_fibers} we discuss the notion of lifting
point and finish the proof of Theorem \ref{th: intro: conditions
equidim V}. In Section \ref{section:
modular_simultaneous_noether_normalization} we prove Theorem
\ref{th: intro: good modular reduction}. For sake of readability,
all estimates on heights of varieties underlying the proof of this
result are postponed to Appendix \ref{section: appendix}. Finally,
in Section \ref{section: computation_of_a_kronecker_representation}
we describe our algorithm for solving the input system $F_1=0\klk
F_r=0$ and analyze its bit complexity, showing thus Theorem \ref{th:
intro: bit complexity Kronecker}.
%
%
\section{Notions and notations}\label{notions_and_notations}
We use standard notions and notations of commutative algebra and
algebraic geometry as can be found in, e.g., \cite{Kunz85},
\cite{Eisenbud95}, \cite{Shafarevich94}.

Let $\K$ be a field and $\overline{\K}$ its algebraic closure. Let
$\K[X_1,\dots,X_n]$ denote the ring of $n$--variate polynomials in
indeterminates $X_1,\dots,X_n$ and coefficients in $\K$. Let
$\mathbb{A}^n:=\mathbb{A}^n(\overline{\K})$ be the affine
$n$--dimensional space over $\overline{\K}$. A subset of
$\mathbb{A}^n$ is called a $\K$--\emph{definable affine subvariety
of} $\mathbb{A}^n$ (a $\K$--\emph{variety} for short) if it is  the
set of common zeros in $\mathbb{A}^n$ of a set of polynomials in
$\K[X_1,\dots,X_n]$. We will use the notations
$\mathcal{V}(F_1,\dots,F_s)$ and $\{F_1=0,\dots,F_s=0\}$ to denote
the $\K$--variety defined by $F_1,\dots,F_s$. Further, if
$\mathcal{I}$ is an ideal of $\K[X_1,\dots,X_n]$, then
$\mathcal{V}(\mathcal{I})$ denotes the $\K$--variety of
$\mathbb{A}^n$ defined by the elements of $\mathcal{I}$. On the
other hand, we shall denote by $\mathcal{I}(V)$ the vanishing ideal
of a $\K$--variety $V\subset\A^n$ in $\K[X_1,\dots,X_n]$ and by
$\K[V]$ its coordinate ring, namely the quotient ring $\K[V]:=
\K[X_1,\dots,X_n]/\mathcal{I}(V)$.

Let $V\subseteq \mathbb{A}^n$ be a $\K$--variety. We denote by $\dim
V$ its dimension with respect to the Zariski topology over $\K$
(which agrees with the Krull dimension of $\K[V]$). More generally,
if $R$ is a ring, then $\dim R$ denotes its Krull dimension. Suppose
further that $V$ is irreducible with respect to the Zariski topology
over $\K$. We define its \emph{degree} as the maximum number of
points lying in the intersection of $V$ with an affine linear
$\overline{\K}$--variety $L$ of $\mathbb{A}^n$ of codimension $\dim
V$ for which $\# (V\cap L)<\infty$. Now, if
$V=\mathcal{C}_1\cup\cdots\cup\mathcal{C}_N$ is the decomposition of
$V$ into irreducible $\K$--components,  we define the degree of $V$
as $\deg V=\sum_{i=1}^N \deg \mathcal{C}_i$ (cf. \cite{Heintz83}).
This definition of degree satisfies the following \emph{B\'ezout
inequality} (\cite{Heintz83}; see also \cite{Fulton98}): if $V$ and
$W$ are $\K$--varieties of $\mathbb{A}^n$, then
 \begin{equation}\label{bezout_inequality}
 \deg(V\cap W)\leq \deg V\deg W.
 \end{equation}
%
%
\subsection{Notions and results of commutative algebra}
A proper ideal $\mathcal{I}$ of $\K[X_1,\dots,X_n]$ is
\emph{unmixed} if the codimensions of its associated primes are all
equal. A classical result asserts that the \emph{unmixedness
theorem} holds for $\K[X_1,\dots,X_n]$, namely an ideal
$\mathcal{I}$ of $\K[X_1,\dots,X_n]$ of codimension $r$ generated by
$r$ elements is unmixed for any $r\geq 0$ (see, e.g., \cite[Theorems
17.6 and 17.7]{Matsumura86}).

Let $\mathcal{I}:=(F_1,\dots,F_r)\subset \K[X_1,\dots,X_n]$ be an
ideal of dimension $n-r$. Then $\mathcal{I}$ is unmixed and defines
an equidimensional $\K$--variety $V\subset \mathbb{A}^n$. Let
$Y_1,\dots, Y_n\in \K[X_1,\dots,X_n]$ be linearly--independent
linear
forms 
such that the mapping $\pi: V \rightarrow \mathbb{A}^{n-r}$ defined
by $Y_1,\dots,Y_{n-r}$ is a finite morphism. The change of variables
$(X_1,\dots,X_n)\to(Y_1,\dots,Y_n)$ is called a \emph{Noether
normalization} of $V$ (or $\mathcal{I}$) and we say that the
variables $Y_1,\dots, Y_n$ are in \emph{Noether position} with
respect to $V$ (or $\mathcal{I}$), the variables $Y_1,\dots,
Y_{n-r}$ being \emph{free}. Let $R:=\K[Y_1, \dots, Y_{n-r}]$ and let
$R'$ denote the field of fractions of $R$. Denote $B:=\K[X_1\klk
X_n]/\mathcal{I}$ and let $B':=R'\otimes_\K B:= R'[Y_{n-r+1}, \dots,
Y_n]/\mathcal{I}^e$, where $\mathcal{I}^e$ is the extension of
$\mathcal{I}$ to $R'[Y_{n-r+1}, \dots, Y_n]$. We consider $B$ as an
$R$--module and $B'$ as an $R'$--vector space respectively. Since
$B$ is a finitely generated, $B'$ is a finite--dimensional
$R'$--vector space, whose dimension we denote by $\dim_{R'}B'$. In
particular, for any $F\in \K[X_1\klk X_n]$ we may consider the
characteristic polynomial $\chi \in R'[T]$ (respectively the minimal
polynomial $\mu \in R'[T]$) of the homothety of multiplication by
$F$ in $B'$. In this situation we have that $\chi$ and $\mu$ belong
to $R[T]$ (see, e.g., \cite[Theorem 1.27]{DuLe08}). We shall call
$\chi$ and $\mu$ respectively the \emph{characteristic} and the
\emph{minimal} polynomials of $F$ modulo $\mathcal{I}$.

Now assume further that $\K$ is an infinite perfect field. Then $B$
is a free $R$--module of finite rank $\textrm{rank}_RB$ (see, e.g.
\cite[Lemma 3.3.1]{GiHeSa93}). Since any basis of $B$ as an
$R$--module induces a basis of $B'$ as an $R'$--vector space, we
have $\mbox{rank}_RB = \dim_{R'}B'$. In this case, we say that
$G\in\K[X_1\klk X_n]$ induces a \emph{primitive element} for
$\mathcal{I}$ if the powers of the image $g$ of $G$ in $B'$ generate
the $R'$--vector space $B'$. We shall also say that $G$ induces a
primitive element of the ring extension $R\hookrightarrow B$.

The following criterion for deciding radicality of an ideal,
probably well--known, is stated and proved here for lack of a
suitable reference.
\begin{lema}\label{lemma: radicality_criterion}
Let $\K$ be a perfect field, $\mathcal{I}:=(F_1, \dots, F_s)\subset
\K[X_1\klk X_n]$ an ideal of dimension $n-s$, and $\mathcal{J}$ the
ideal of $\K[X_1\klk X_n]$ generated by $\mathcal{I}$ and  the
$(s\times s)$--minors of the Jacobian matrix $(\partial F_i/\partial
X_j)_{1\leq i \leq s, 1\leq j \leq n}$. Then the following
conditions are equivalent:
 \begin{itemize}
   \item $\mathcal{I}$ is radical;
   \item $\mathcal{J}$ is not contained in any minimal prime of $\mathcal{I}$.
 \end{itemize}
\end{lema}
\begin{proof}
Let $B:=\K[X_1\klk X_n]/\mathcal{I}$. By \cite[Exercise
11.10]{Eisenbud95}, it suffices to show that the second condition is
equivalent to the following ones:
\begin{enumerate}
  \item the localization of $B$ at each prime of codimension $0$ is regular;\label{R0}
  \item all primes associated to zero in $B$ have codimension $0$. \label{S1}
\end{enumerate}
To prove this equivalence, we observe that the canonical
homomorphism $\K[X_1\klk X_n]\!\rightarrow B$ induces a bijection
between the set of primes associated to $\mathcal{I}$ and the set of
primes  associated to $0$ in $B$. This bijection maps the minimal
primes over $\mathcal{I}$ to the minimal primes over $0$ in $B$,
which are precisely the primes of codimension $0$ in $B$. Now, since
the unmixedness theorem holds in $\K[X_1\klk X_n]$, the ideal
$\mathcal{I}$ is unmixed, and thus the set of primes associated to
$\mathcal{I}$ coincides with the set of minimal primes over
$\mathcal{I}$, which implies that \eqref{S1} is satisfied. Next, the
second condition of the lemma can be rephrased by saying that the
image $\overline{\mathcal{J}}$ of $\mathcal{J}$ in $B$ is not
contained in any prime of $B$ of codimension $0$. By \cite[Corollary
16.20]{Eisenbud95}, this is equivalent to \eqref{R0}, which finishes
the proof.
\end{proof}
%
%
\subsection{Kronecker representations}
Let $V\subset\A^n$ be an equidimensional $\K$--variety of dimension
$n-s$, and let $\mathcal{I}\subset\K[X_1\klk X_n]$ be its vanishing
ideal. For a change of variables $(X_1\klk X_n)\to(Y_1\klk Y_n)$,
denote $R:=\K[Y_1,\dots, Y_{n-s}]$, $B:=\K[V]$ and $R':=\K(Y_1\klk
Y_{n-s})$. Consider $B':=R'[Y_{n-s+1}\klk Y_n]/\mathcal{I}^e$ as an
$R'$-vector space, where $\mathcal{I}^e$ is the extended ideal
$\mathcal{I}R[Y_{n-s+1}\klk Y_n]$, and let $\delta:=\dim_{R'}B'$.
\begin{defi}
A \emph{Kronecker representation} of $\mathcal{I}$ (or $V$) consists
of the following items:
\begin{itemize}
  \item a Noether normalization of $\mathcal{I}$, defined by a linear change
  of variables $(X_1, \dots,\! X_n)$ $\to$ $(Y_1,\dots, Y_n)$ such that
  $Y_{n-s+1}$ induces a \emph{primitive element} for $\mathcal{I}$;
  \item the minimal (monic) polynomial $Q\in R[T]$ of $Y_{n-s+1}$ modulo
  $\mathcal{I}$;
  \item the (unique) polynomials $W_{n-s+2}\klk W_n\in R'[T]$ of degree at most
  $\delta-1$ such that the following identity of ideals holds in
  $R'[Y_{n-s+1}, \dots, Y_n]$:
\begin{equation}\label{ident:kronecker_repres_ideals}
\mathcal{I}^e\!=\!\bigl(Q(Y_{n-s+1}),\!
Q'(Y_{n-s+1})Y_{n-s+2}-W_{n-s+2}(Y_{n-s+1}) \klk
Q'(Y_{n-s+1})Y_n-Q_n(Y_{n-s+1})\bigr),
\end{equation}
where $Q'$ denotes the first derivative of $Q$ with respect to $T$.
\end{itemize}
Considering instead polynomials $V_{n-s+2}\klk V_n$ of degree at
most $\delta-1$ such that
\[
\mathcal{I}^e=\bigl(Q(Y_{n-s+1}), Y_{n-s+2}-V_{n-s+2}(Y_{n-s+1}),
\dots, Y_n-V_n(Y_{n-s+1})\bigr),
\]
we have a \emph{univariate representation} of $\mathcal{I}$ (or
$V$).
\end{defi}

If $Q'\neq 0$, identity \eqref{ident:kronecker_repres_ideals} may be
interpreted in geometric terms as we now explain. Let
$\ell:\mathbb{A}^n\rightarrow \mathbb{A}^n$ be the linear mapping
defined by $Y_1,\dots,Y_n$ and $W:=\ell(V)$. We interpret
$Y_1,\dots, Y_n$ as new indeterminates and consider the mapping
$\Pi:W\rightarrow \mathbb{A}^{n-s+1}$ defined by the projection on
the first $n-s+1$ coordinates. Considering $Q$ as an element of
$\K[Y_1,\dots,Y_{n-s+1}]$, it turns out that $\Pi$ defines a
birational isomorphism between $W$ and the hypersurface $\{Q=0\}$ of
$\mathbb{A}^{n-s+1}$, whose inverse is the rational mapping
$\Phi:\{Q=0\}\rightarrow W$ defined in the following way:
$$\Phi(\bfs y):=\left(\bfs y,
\frac{W_{n-s+2}(\bfs y)}{Q'(\bfs y)}, \dots, \frac{W_n(\bfs
y)}{Q'(\bfs y)}\right).$$
%
%
\subsection{Model of computation}
Besides the Big--Oh notation $\mathcal{O}$, we also use the standard
Soft--Oh notation $\mathcal{O}^\sim$ which does not take into
account logarithmic terms. 
We remark that the cost of certain basic operations (such as
addition, multiplication, division, and gcd) with integers of bit
length $m$ is in $\mathcal{O}^\sim(m)$. In particular, arithmetic
operations in the prime finite field $\fp$ of $p$ elements can be
performed with $\mathcal{O}^\sim(\log p)$ bit operations.

Algorithms in computer algebra usually consider the standard dense
(or sparse) representation model, where multivariate polynomials are
encoded by means of the vector of all (or of all nonzero)
coefficients. However, since a generic $n$--variate polynomial of
degree $d$ has $\binom{n + d}{n}=\mathcal{O}(d^n)$ nonzero
coefficients, its dense or sparse representation requires an
exponential size in  $d$ and $n$, and their manipulation usually
requires an exponential number of arithmetic operations with respect
to $d$ and $n$. To avoid this phenomenon we will use an alternative
representation for multivariate polynomials by means of
straight--line programs (cf. \cite{BuClSh97}). A {\em
(division--free) straight--line program} $\beta$ in $\K[X_1\klk
X_n]$ which {\em represents} or {\em evaluates} polynomials $F_1\klk
F_s \in \K[X_1\klk X_n]$ is a sequence $(Q_1, \dots, Q_r)$ of
elements of $\K[X_1\klk X_n]$ satisfying the following conditions:
\begin{itemize}
  \item $\{F_1, \dots, F_s\} \subseteq \{Q_1, \dots, Q_r\}$;
  \item there exists a finite subset $\mathcal{T}\subset
  \K$, called the set of \emph{parameters} of $\beta$,
  such that for every $1 \le \rho \le r$, the polynomial $Q_{\rho}$ either is an
element of $\mathcal{T} \cup \{X_1, \dots, X_n \}$, or there exist
$1 \le \rho_1, \rho_2 < \rho$ such that
$Q_{\rho}=Q_{\rho_1}\circ_{\rho}Q_{\rho_2}$, where $\circ_{\rho}$ is
one of the arithmetic operations $+, -, \times$.
\end{itemize}
The {\em length} of $\beta$ is defined as the total number of
arithmetic operations performed during the evaluation process
defined by $\beta$.

Our algorithm is probabilistic, of {\em Monte Carlo} type (see,
e.g., \cite{GaGe99}). One of the probabilistic aspects is related to
random choices of points outside certain Zariski open sets. A basic
tool for estimating the corresponding probability of success is the
following well--known result (see, e.g., \cite[Lemma 6.44]{GaGe99}).
\begin{lema}
\label{lemma: zippel_schwartz} Let $R$ be an integral domain, $U_1,
\dots, U_k$ indeterminates over $R$, $S\subseteq R$ a finite set
with $s:=\# S$ elements, and $F\in R[U_1, \ldots, U_k]$ a nonzero
polynomial of degree at most $d$. Then $F$ has at most $ds^{k-1}$
zeros in $S^k$.
\end{lema}
We shall interpret Lemma \ref{lemma: zippel_schwartz} in terms of
probabilities: for an element $\bfs u$ chosen uniformly at random in
$S^k$, the probability that $F(\bfs u)\neq 0$ is greater than
$1-d/s$.

The second probabilistic aspect concerns the choice of a ``lucky''
prime number $p$. In connection with this matter, we have the
following result (see, e.g., \cite[Section 18.4]{GaGe99}).
\begin{lema}\label{lemma: finding_primes}
Let $B$, $m$ be positive integers and $M$ a nonzero integer such
that $\log |M|\leq \frac{B}{m}$. There is a probabilistic algorithm
which, from the integer $B$ and any positive integer $k$, returns a
prime $p$ between $B+1$ and $2B$ not dividing $M$. It performs
$\mathcal{O}^\sim(k\,\log^2 B)$ bit operations and returns the right
result with probability at least
\[
\Bigl(1-\frac{\log B}{2^{k-1}}\Bigr)\Bigl(1-\frac{2}{m}\Bigr).
\]
\end{lema}
\begin{proof}
According to, e.g.,  \cite[Theorem 18.8]{GaGe99}, there is a
probabilistic algorithm which computes a random prime $p$ such that
$B<p\le 2B$ with $\mathcal{O}^\sim(k\,\log^2 B)$ bit operations and
probability of success at least $1-{\log B}/{2^{k-1}}$. On the other
hand, if $p$ is a random prime with $B< p \leq 2B$, then $p$ does
not divide $M$ with probability at least $1-{2}/{m}$. Combining both
assertions the lemma follows.
\end{proof}
%
%
\section{On Noether normalizations}
\label{section: Noether normalization}
Let $\K$ be a perfect field and  $V\subset \mathbb{A}^n$ an
equidimensional $\K$--variety of dimension $n-s\ge 0$ and degree
$\delta$. In this section we obtain a condition on the coefficients
of linear forms $Y_1\klk Y_{n-s+1}\in\K[X_1\klk X_n]$ which implies
that $Y_1\klk Y_{n-s}$ define a Noether normalization of $V$ and
$Y_{n-s+1}$ is a primitive element of the ring extension $\K[Y_1\klk
Y_{n-s}]\hookrightarrow\K[V]$ (Proposition \ref{prop:
finite_fiber_and_finite_morphism}). As these conditions rely heavily
on properties of the Chow form of $V$, we also recall the notion of
Chow form of an equidimensional variety and some of its basic
properties.
%
%
\subsection{The Chow form of an equidimensional variety}
\label{subsec: Chow form equidimensional variety}
%
Let $\bfs\Lambda^h\!:=\!(\!\Lambda_{ij})_{1\leq i \leq n-s+1, 0\leq
j \leq n}$ be a matrix of indeterminates over $\K[V]$, let
$\bfs\Lambda^h_i:=(\Lambda_{i0}, \dots, \Lambda_{in})$ and
$\bfs\Lambda_i:=(\Lambda_{i1}, \dots, \Lambda_{in})$ for $1 \leq i
\leq n-s+1$. A Chow form of $V$ is a square--free polynomial
$F_{\scriptscriptstyle V}$ of $\K[\bfs\Lambda^h]$ such that
$F_{\scriptscriptstyle V}(\bfs\lambda^h)=0$ if and only if
$\overline{V}\cap\{\lambda_{i0}+\sum_{j=1}^n\lambda_{ij}X_j=0\ (1\le
i\le n-s+1)\}$ is nonempty, where $\overline{V}\subset\mathbb{P}^n$
is the projective closure of $V$ with respect to the canonical
inclusion $\mathbb{A}^n\hookrightarrow \mathbb{P}^{n}$ (see
\cite[Chapter X, Section 6]{HoPe68b}). We observe that
$F_{\scriptscriptstyle V}$ is multihomogeneous of degree $\delta$ in
each group of variables $\bfs\Lambda^h_i$ for $1 \leq i \leq n-s+1$,
and is uniquely determined up to nonzero multiples in $\K$. Let
$\bfs\Lambda:=(\Lambda_{ij})_{1\leq i \leq n-s+1,1\leq j \leq n}$
and let $Z_1,\dots,Z_{n-s+1}$ be new indeterminates. Let
$P_{\scriptscriptstyle V}\in \K[\bfs\Lambda,Z_1,\dots,Z_{n-s+1}]$ be
the unique polynomial such that
\[
P_{\scriptscriptstyle
V}(\bfs\Lambda,\Lambda_{10},\dots,\Lambda_{n-s+1,
0})=F_{\scriptscriptstyle
V}(\bfs\Lambda^h_1,\dots,\bfs\Lambda^h_{n-s+1}).
\]
By abuse of language we also call $P_{\scriptscriptstyle V}$ a Chow
form of $V$.

Let $\xi_1,\dots,\xi_n$ be the coordinate functions of $V$ induced
by $X_1,\dots,X_n$. Set $\bfs\xi:=(\xi_1,\dots,\xi_n)$ and let
$\bfs\Lambda_i \cdot\bfs\xi \in \K[V][\bfs\Lambda]$ be defined by
\[
\bfs\Lambda_i \cdot\bfs\xi:=\sum_{j=1}^n\Lambda_{ij}\xi_j \quad (1
\leq i \leq n-s+1).
\]
A fundamental property of the Chow form is that
$P_{\scriptscriptstyle V}$ is uniquely determined, up to
multiplication by nonzero elements of $\K$, by the following two
conditions:
\begin{itemize}
\item if
$\bfs\Lambda\bfs\xi:=(\bfs\Lambda_1\cdot\bfs\xi,
\dots,\bfs\Lambda_{n-s+1}\cdot\bfs\xi)$, then the following identity
holds in $\K[V][\bfs\Lambda]$:
\begin{equation}\label{ident:chow form_fundamental}
P_{\scriptscriptstyle V}(\bfs\Lambda,\bfs\Lambda\bfs\xi)=0.
\end{equation}
Equivalently, let $\bfs\Lambda_i\cdot \bfs
X:=\sum_{j=1}^n\Lambda_{ij}X_j$ for $1\leq i\leq n-s+1$ and
$\bfs\Lambda\bfs X:=(\bfs\Lambda_1\cdot \bfs X, \dots,
\bfs\Lambda_{n-s+1}\cdot \bfs X)$. Then the polynomial
$P_{\scriptscriptstyle V}(\bfs\Lambda,\bfs\Lambda\bfs X)\in
\K[\bfs\Lambda,\bfs X]$ vanishes on the variety
$\mathbb{A}^{(n-s+1)n}\times V$.
\item  If $G\in \K[\bfs\Lambda,Z_1,\dots, Z_{n-s+1}]$ is
any polynomial such that $G(\bfs\Lambda,\bfs\Lambda\bfs\xi)=0$, then
$P_{\scriptscriptstyle V}$ divides $G$ in $\K[\bfs\Lambda,
Z_1,\dots, Z_{n-s+1}]$.
\end{itemize}

Furthermore, $F_{\scriptscriptstyle V}$ has the following features
(see \cite[Chapter X, Sections 7 and 9]{HoPe68b}):
\begin{enumerate}
\item $F_{\scriptscriptstyle V}$ is homogeneous of degree $\delta$
in the $(n-s+1)\times (n-s+1)$--minors of $\bfs\Lambda^h$;
\item $\deg_{(\Lambda_{10}, \dots, \Lambda_{n-s+1,0})} F_{\scriptscriptstyle V}
=\deg_{\Lambda_{n-s+1,0}} F_{\scriptscriptstyle V}=\delta$;
\label{features_chow_form_item_2}
\item if $V$ is an irreducible $\K$--variety, then $F_{\scriptscriptstyle V}$
is an irreducible polynomial of $\K[\bfs\Lambda^h]$. More generally,
if $V=\mathcal{C}_1\cup \cdots \cup \mathcal{C}_N$ is the
decomposition of $V$ into irreducible $\K$--components, and
$F_{\scriptscriptstyle \mathcal{C}_i}$ is a Chow form of
$\mathcal{C}_i$ for $1\leq i \leq N$, then $\prod_{1\leq i \leq s}
F_{\scriptscriptstyle \mathcal{C}_i}$ is a Chow form of $V$.
\label{features_chow_form_item_3}
\end{enumerate}

\begin{remark}
Let $A_{\scriptscriptstyle V}\in
\K[\bfs\Lambda^h_1,\dots,\bfs\Lambda^h_{n-s}]$ be the (nonzero)
polynomial which arises as the coefficient of the monomial
$\Lambda_{n-s+1,0}^{\delta}$ in $F_{\scriptscriptstyle V}$,
considering $F_{\scriptscriptstyle V}$ as an element of
$\K[\bfs\Lambda][\Lambda_{10},\dots,\Lambda_{n-s+1,0}]$. Then
(\ref{features_chow_form_item_2}) implies that
$A_{\scriptscriptstyle V}$ is independent of
$\Lambda_{10},\dots,\Lambda_{n-s\,0}$, that is,
$A_{\scriptscriptstyle V}\in \K[\bfs\Lambda_1,\dots,
\bfs\Lambda_{n-s}]$. In particular, $A_{\scriptscriptstyle V}$ is
homogeneous of degree $\delta$ in the $(n-s)\times (n-s)$--minors of
the $(n-s) \times n$--matrix $\bfs\Lambda^*=(\Lambda_{ij})_{1\leq i
\leq n-s, 1 \leq j \leq n}$.
\end{remark}

Let $\rho_{\scriptscriptstyle V}\in \K[\bfs\Lambda, Z_1,\dots,
Z_{n-s}]$ be the discriminant of $P_{\scriptscriptstyle V}$ with
respect to $Z_{n-s+1}$, namely
  \[
  \rho_{\scriptscriptstyle V}:=\mathrm{Res}_{Z_{n-s+1}}
  \left(P_{\scriptscriptstyle V}, \frac{\partial P_{\scriptscriptstyle V}}{\partial Z_{n-s+1}}\right).
  \]

\begin{lema}\label{lemma: nonzero_discriminant}
$\rho_{\scriptscriptstyle V}$ and $\partial P_{\scriptscriptstyle
V}/\partial Z_{n-s+1}$ are both nonzero.
\end{lema}
\begin{proof}
We have that $A:=\K[\bfs\Lambda,
Z_1,\dots,Z_{n-s+1}]/(P_{\scriptscriptstyle V})$ is a reduced
$\K$--algebra. Since $\K$ is perfect, by \cite[Corollary, page
194]{Matsumura80} it follows that $A$ is a separable $\K$--algebra.
Let $\K'$ denote the algebraic closure of
$\K(\bfs\Lambda,Z_1,\dots,Z_{n-s})$. By \cite[27.G]{Matsumura80}, we
deduce that the $\K'$--algebra
$A\otimes_{\K}\K'=\K'[Z_{n-s+1}]/(P_{\scriptscriptstyle V})$ is
reduced.  Since $\K'$ is a perfect field, this implies that
$\partial P_{\scriptscriptstyle V}/\partial Z_{n-s+1}\neq 0$. Now,
by \eqref{features_chow_form_item_2} and
\eqref{features_chow_form_item_3} above, each irreducible factor of
$P_{\scriptscriptstyle V}$ is a Chow form of an irreducible
component $\mathcal{C}_i$ of $V$, of positive degree $\deg
\mathcal{C}_i$ in $Z_{n-s+1}$. Then the previous argument shows that
the partial derivative with respect to $Z_{n-s+1}$ of each
irreducible factor of $P_{\scriptscriptstyle V}$ does not vanish,
which in turn implies that $P_{\scriptscriptstyle V}$ and $\partial
P_{\scriptscriptstyle V}/\partial Z_{n-s+1}$ are relatively prime
polynomials of $\K[\bfs\Lambda, Z_1,\dots,Z_{n-s+1}]$. Since
$\K[\bfs\Lambda, Z_1,\dots,Z_{n-s}]$ is a factorial ring, this
implies that the resultant $\rho_{\scriptscriptstyle V}$ of these
polynomials does not vanish.
\end{proof}

Further, $\rho_{\scriptscriptstyle V}$ satisfies the following
degree estimates:
$$
\deg_{(Z_1, \dots, Z_{n-s})} \rho_{\scriptscriptstyle V}  \leq
(2\delta-1)\delta, \quad  \deg_{\bfs\Lambda_{i}}
\rho_{\scriptscriptstyle V} \leq (2\delta-1)\delta \quad (1 \leq i
\leq n-s+1).
$$
In particular, for its total degree we have $\deg
\rho_{\scriptscriptstyle V}\leq (n-s+2)(2\delta^2-\delta)$.

Let $\bfs Z:=(Z_1,\dots,Z_{n-s+1})$. Further, for any
$\bfs\lambda:=(\lambda_{ij})_{1\leq i \leq n-s+1, 1 \leq j \leq
n}\in \mathbb{A}^{(n-s+1)n}$, we write
$\bfs\lambda_i:=(\lambda_{i1}, \dots, \lambda_{in})$ and
$\bfs\lambda_i \cdot\bfs\xi:= \sum_{j=1}^n \lambda_{ij}\xi_j$ for $1
\leq i \leq n-s+1$. We consider $\K[V][\bfs\Lambda]$ as a
$\K[\bfs\Lambda, \bfs Z]$--algebra through the ring homomorphism
$\K[\bfs\Lambda, \bfs Z]\rightarrow \K[V][\bfs\Lambda]$ which maps
any $F \in \K[\bfs\Lambda, \bfs Z]$ to $F(\bfs\Lambda,
\bfs\Lambda\bfs\xi)$. In these terms, we have the following result.
\begin{lema}\label{lemma: non_zero_divisor}
$\partial P_{\scriptscriptstyle V}/\partial Z_{n-s+1}$ is not a zero
divisor of the $\K[\bfs\Lambda, \bfs Z]$--algebra
$\K[V][\bfs\Lambda]$.
\end{lema}
\begin{proof}
Let $G\in \K[\bfs\Lambda, \bfs X]$ be any polynomial such that
\begin{equation}\label{eq:non_zero_divisor_lemma_1}
\frac{\partial P_{\scriptscriptstyle V}}{\partial
Z_{n-s+1}}(\bfs\Lambda,\bfs\Lambda\bfs\xi)\cdot
G(\bfs\Lambda,\bfs\xi)=0
\end{equation}
in $\K[V][\bfs\Lambda]$. We have $\rho_{\scriptscriptstyle V} \in
(P_{\scriptscriptstyle V},\partial P_{\scriptscriptstyle V}/\partial
Z_{n-s+1})\K[\bfs\Lambda, \bfs Z]$. Since $P_{\scriptscriptstyle
V}(\bfs\Lambda,\bfs\Lambda\bfs\xi)=0$, we deduce that
$\rho_{\scriptscriptstyle V}(\bfs\Lambda,\bfs\Lambda_1 \cdot\bfs\xi,
\dots, \bfs\Lambda_{n-s}\cdot\bfs\xi)$ is a multiple of $\partial
P_{\scriptscriptstyle V}/\partial Z_{n-s+1}(\bfs\Lambda,\bfs\Lambda
\bfs\xi)$ in the ring $\K[V][\bfs\Lambda]$. Combining this with
\eqref{eq:non_zero_divisor_lemma_1}, we deduce that
\[
\rho_{\scriptscriptstyle V}(\bfs\Lambda,\bfs\Lambda_1 \cdot\bfs\xi,
\ldots,\bfs\Lambda_{n-s}\cdot\bfs\xi)\cdot G(\bfs\Lambda,\bfs\xi)=0
\]
in $\K[V][\bfs\Lambda]$. Suppose that there exists an irreducible
$\K$--component $\mathcal{C}$ of $V$ such that
$G(\bfs\Lambda,\bfs\xi)\neq 0$ in $\K[\mathcal{C}][\bfs\Lambda]$.
Then
\[
\rho_{\scriptscriptstyle V}(\bfs\Lambda,\bfs\Lambda_1 \cdot\bfs\xi,
\ldots,\bfs\Lambda_{n-s}\cdot\bfs\xi)\cdot G(\bfs\Lambda,\bfs\xi)=0
\]
in $\K[\mathcal{C}][\bfs\Lambda]$. Since
$\K[\mathcal{C}][\bfs\Lambda]$ is an integral domain, we conclude
that $\rho_{\scriptscriptstyle V}(\bfs\Lambda,\bfs\Lambda_1 \cdot
\bfs\xi, \ldots, \bfs\Lambda_{n-s}\cdot\bfs\xi)=0$ in
$\K[\mathcal{C}][\bfs\Lambda]$. This implies that
\begin{equation}\label{eq:non_zero_divisor_lemma_2}
\rho_{\scriptscriptstyle V}(\bfs\Lambda,\bfs\Lambda_1 \cdot\bfs\xi,
\ldots, \bfs\Lambda_{n-s}\cdot\bfs\xi)=0
\end{equation}
in $\overline{\K}[\mathcal{C}][\bfs\Lambda]$, where $\overline{\K}$
is the algebraic closure of $\K$. On the other hand, by Lemma
\ref{lemma: nonzero_discriminant} the polynomial
$\rho_{\scriptscriptstyle V}$ is  nonzero. Then, for a generic
choice of $\bfs\lambda \in \mathbb{A}^{(n-s+1)n}$, the ring
extension $\overline{\K}[\bfs\lambda_1
\cdot\bfs\xi,\ldots,\bfs\lambda_{n-s}\cdot \bfs\xi]\hookrightarrow
\overline{\K}[V]$ is integral and $\rho_{\scriptscriptstyle
V}(\bfs\lambda, Z_1, \ldots, Z_{n-s})$ is a nonzero polynomial in
$\overline{\K}[Z_1,\ldots, Z_{n-s}]$. By
(\ref{eq:non_zero_divisor_lemma_2}) we deduce that
$\rho_{\scriptscriptstyle V}(\bfs\lambda, \bfs\lambda_1 \cdot
\bfs\xi, \ldots, \bfs\lambda_{n-s}\cdot \bfs\xi)=0$
in $\overline{\K}[\mathcal{C}]$, which shows that $\bfs\lambda_1
\cdot\bfs\xi, \ldots, \bfs\lambda_{n-s} \cdot\bfs\xi$ are
algebraically dependent over $\overline{\K}$. Since
$\overline{\K}[\bfs\lambda_1 \cdot\bfs\xi, \dots, \bfs\lambda_{n-s}
\cdot\bfs\xi]\hookrightarrow \overline{\K}[\mathcal{C}]$ is also
integral, it follows that $\dim \mathcal{C}\leq n-s-1$, which is a
contradiction. Therefore, $G(\bfs\Lambda,\bfs\xi)= 0$ in
$\K[\mathcal{C}][\bfs\Lambda]$ for every irreducible component
$\mathcal{C}$ of $V$. We conclude that $G(\bfs\Lambda,\bfs\xi)= 0$
in $\K[V][\bfs\Lambda]$, which finishes the proof.
\end{proof}
%
%
\subsection{A generic condition for a Noether normalization}
In the sequel, 
for $\bfs\lambda:=(\lambda_{ij})_{1\leq i \leq n-s+1, 1\leq j \leq
n}\in \K^{(n-s+1)n}$ we write $\bfs\lambda^*:=(\lambda_{ij})_{1\leq
i
\leq n-s, 1\leq j \leq n}$. 
%
\begin{prop}\label{prop: finite_fiber_and_finite_morphism}
With hypotheses and notations as before, let $\bfs\lambda\in
\K^{(n-s+1)n}$ be such that $A_{\scriptscriptstyle
V}(\bfs\lambda^*)\neq 0$. Let $Y_i:=\bfs\lambda_i \cdot \bfs X$ for
$1\leq i \leq n-s+1$, $R:=\K[Y_1, \dots, Y_{n-s}]$, $B:=\K[V]$,
$R':=\K(Y_1, \dots, Y_{n-s})$ and $B':=R'\otimes_\K B$. Then the
mapping $\pi: V \rightarrow \mathbb{A}^r$ defined by $Y_1, \dots,
Y_{n-s}$ is a finite morphism. Further, if $\rho_{\scriptscriptstyle
V}(\bfs\lambda, Z_1,\ldots,Z_{n-s})\neq 0$, then $Y_{n-s+1}$ induces
a primitive element of the ring extension $R\hookrightarrow \K[V]$
and $\dim_{R'}B'\leq \delta$.
\end{prop}
\begin{proof}
Let $\bfs\Lambda^*=(\Lambda_{ij})_{1\leq i \leq n-s, 1\leq j \leq
n}$. Recall that $A_{\scriptscriptstyle V}$ is homogeneous of degree
$\delta$ in the $(n-s)\times (n-s)$--minors of $\bfs\Lambda^*$.
Since $A_{\scriptscriptstyle V}(\bfs\lambda^*)\neq 0$, at least one
of the minors of the $(n-s)\times n$ matrix $\bfs\lambda^*$ is
nonzero. We deduce that the linear forms $Y_1, \dots, Y_{n-s}$ are
linearly independent. Thus there exist linear forms ${Y}_{n-s+1},
\dots, {Y}_n \in \K[\bfs X]$ such that $Y_1, \dots, Y_{n-s},
{Y}_{n-s+1},\dots,{Y}_n$ are linearly independent. Let $\bfs
w_k:=(w_{k1}, \dots, w_{kn})\in \K^n$ be such that ${Y}_{n-s+k}=\bfs
w_k\cdot \bfs X$ for $1\leq k \leq s$. Let $Q_k\in
\K[Z_1,\dots,Z_{n-s+1}]$ be the polynomial obtained by replacing in
$P_{\scriptscriptstyle V}$ the matrix $\bfs\Lambda$ for
$(\bfs\lambda^*,\bfs w_k)$. From \eqref{ident:chow form_fundamental}
we deduce that
\begin{equation}\label{eq:finite_morphism_proposition}
Q_k(Y_1,\dots,Y_{n-s},\bfs w_k \cdot\bfs\xi)=0
\end{equation}
in the $R$--algebra $B$ for $1\leq k \leq s$, where
$\bfs\xi:=(\xi_1, \dots, \xi_n)$ denotes the $n$--tuple of
coordinate functions in $B$ induced by $X_1, \dots, X_n$. Observe
that $\deg_{Z_{n-s+1}}Q_k\le\delta$ and that $A_{\scriptscriptstyle
V}(\bfs\lambda^*)$ is the coefficient of $Z_{n-s+1}^{\delta}$ in
$Q_k$. Since  $A_{\scriptscriptstyle V}(\bfs\lambda^*)\neq 0$, we
have that $\deg_{Z_{n-s+1}}Q_k=\delta$ and
\eqref{eq:finite_morphism_proposition} may be interpreted as a
relation of integral dependence for the image $\bfs w_k
\cdot\bfs\xi$ of ${Y}_{n-s+k}$ in $B$ over $R$ for $1 \leq k \leq
s$. Moreover, $\K[Y_1,\ldots,Y_n]= \K[\bfs X]$ because the linear
forms $Y_1, \dots, Y_n$ are linearly independent. This implies that
$R\rightarrow B$ is an integral ring extension.

To prove that $\pi$ is finite, let $\mathcal{C}$ be any irreducible
$\K$--component of $V$ and let $\pi_{\mathcal{C}}$ be the
restriction of $\pi$ to $\mathcal{C}$. It suffices to prove that
$\pi_{\mathcal{C}}$ is dominant or, equivalently, that its dual ring
homomorphism $\pi^*_{\mathcal{C}}: \K[\mathbb{A}^{n-s}]\rightarrow
\K[\mathcal{C}]$ is injective. Let $t_i$ denote the $i$--th
coordinate function of $\mathbb{A}^{n-s}$ for $1\leq i \leq n-s$.
With a slight abuse of notation denote also by $\bfs\xi$ the
$n$--tuple of coordinate functions of $\K[\mathcal{C}]$ induced by
$X_1, \dots, X_n$. Then $\pi^*_{\mathcal{C}}(t_i)=\bfs\lambda_i
\cdot\bfs\xi$ for $1\leq i \leq n-s$. Since $\K[\mathcal{C}]$ is
integral over $\K[\bfs\lambda_1 \cdot\bfs\xi,
\dots,\bfs\lambda_{n-s} \cdot\bfs\xi]$ and $\dim \mathcal{C}=r$, we
deduce that $\bfs\lambda_1 \cdot\bfs\xi, \dots, \bfs\lambda_{n-s}
\cdot \bfs\xi$ are algebraically independent over $\K$. This implies
the injectivity of $\pi^*_{\mathcal{C}}$, which concludes the proof
of the first assertion of the proposition.

Next, 
taking partial derivatives with respect to the variable
$\Lambda_{n-s+1,k}$ at both sides of \eqref{ident:chow
form_fundamental}, we obtain the following identity in
$\K[V][\bfs\Lambda]$ for $1\leq k \leq n$:
\begin{equation}\label{ident:partial_derivatives_chow_form}
\frac{\partial P_{\scriptscriptstyle V }}{\partial
Z_{n-s+1}}(\bfs\Lambda,\bfs\Lambda\bfs\xi)\,\xi_k + \frac{\partial
P_{\scriptscriptstyle V} }{\partial
\Lambda_{n-s+1,k}}(\bfs\Lambda,\bfs\Lambda\bfs\xi)=0.
\end{equation}
From \eqref{ident:chow form_fundamental} and
\eqref{ident:partial_derivatives_chow_form} we deduce that there
exists $R_k\in \K[\bfs\Lambda, \bfs Z]$ such that
\begin{equation}\label{ident:discriminant_chow_form}
\rho_{\scriptscriptstyle V}(\bfs\Lambda,\bfs\Lambda_1\cdot\bfs\xi,
\dots,\bfs\Lambda_{n-s}\cdot\bfs\xi)\,\xi_k=
R_k(\bfs\Lambda,\bfs\Lambda\bfs\xi)
\end{equation}
in $\K[V][\bfs\Lambda]$ for $1\le k\le n$. By substituting
$\bfs\lambda$ for $\bfs\Lambda$  in
\eqref{ident:discriminant_chow_form} we deduce that
\[
\rho_{\scriptscriptstyle
V}(\bfs\lambda,Y_1,\dots,Y_{n-s})\xi_k=R_k(\bfs\lambda, Y_1,\dots,
Y_{n-s}, \bfs\lambda_{n-s+1}\cdot\bfs\xi)
\]
in $\K[V]$ for $1\leq k \leq n$. By the choice of $\bfs\lambda$, the
polynomial $\rho_{\scriptscriptstyle V}(\bfs\lambda,
Z_1,\dots,Z_{n-s})$ is nonzero. Since $\bfs\lambda_1\cdot \bfs\xi,
\dots,\bfs\lambda_{n-s}\cdot\bfs\xi$ are algebraically independent
over $\K$, we deduce that $\rho_{\scriptscriptstyle V}(\bfs\lambda,
Y_1, \dots, Y_{n-s})$ is a nonzero element of $R$. Then the previous
identities show that the powers of $\bfs\lambda_{n-s+1} \cdot
\bfs\xi$ generate the $R'$--vector space $B'$. In other words,
$Y_{n-s+1}$ induces a primitive element of the ring extension
$R\hookrightarrow \K[V]$.

Now, let $Q\in R[Z_{n-s+1}]$ be the polynomial obtained by
substituting $\bfs\lambda$ for $\bfs\Lambda$ and $Y_1, \dots,
Y_{n-s}$ for $Z_1, \dots, Z_{n-s}$ in $P_{\scriptscriptstyle V}$.
From \eqref{ident:chow form_fundamental} we deduce that
$Q(\bfs\lambda_{n-s+1} \cdot\bfs \xi)=0$ in $B'$. Taking into
account that $\deg_{Z_{n-s+1}} Q =\delta$ we conclude that
$\dim_{R'}B'\leq \delta$.
\end{proof}
%
%
\section{Lifting points and lifting fibers}\label{section: lifting_points_and_lifting_fibers}
%
%
Assume as in Section \ref{section: Noether normalization} that $\K$
is perfect field and $V\subset \mathbb{A}^n$ is an equidimensional
$\K$--variety of dimension $n-s$ and degree $\delta$. Let
$F_1,\dots,F_s \in \K[\bfs X]$ be polynomials that generate the
vanishing ideal $\mathcal{I}$ of $V$. Assume further that we are
given linear forms $Y_1,\dots, Y_{n-s}\in \K[\bfs X]$ defining a
finite morphism $\pi: V \rightarrow \mathbb{A}^{n-s}$, and let $J\in
\K[\bfs X]$ be the Jacobian determinant of $Y_1,\dots, Y_{n-s},F_1,
\dots, F_s$ with respect to the variables $X_1, \dots, X_n$. A point
$\bfs p\in \K^{n-s}$ is called a \emph{lifting point} of $\pi$ with
respect to the system $F_1=0,\dots, F_s=0$ if $J(\bfs x)\neq 0$ for
every $\bfs x\in \pi^{-1}(\bfs p)$.  We call the zero--dimensional
variety $\pi^{-1}(\bfs p)$ the \emph{lifting fiber} of $\bfs p$.
According to Proposition \ref{prop: lifting_point_and_rank} below,
the notions of lifting point and lifting fiber are independent of
the choice of the polynomials $F_1, \dots, F_s$ generating
$\mathcal{I}(V)$. Consequently, in the sequel we shall simply say
that $\bfs p$ is a lifting point of $\pi$ and $\pi^{-1}(\bfs p)$ is
a lifting fiber without reference to $F_1,\dots,F_s$.

The notion of lifting fiber in this framework was first introduced
in \cite{GiHaHeMoMoPa97}. The concept was isolated in
\cite{HeKrPuSaWa00}, where it was shown how one can use a Kronecker
representation of a lifting fiber of a given equidimensional variety
to tackle certain fundamental algorithmic problems associated to it
(see also \cite{GiLeSa01}, \cite{Schost03}, \cite{BoMaWaWa04},
\cite{PaSa04} and \cite{JeMaSoWa09} for extensions, refinements and
algorithmic aspects related to lifting fibers). The notion is also
important in numerical algebraic geometry, where it is known under
the name of {\em witness set} (see, e.g., \cite{SoWa05}; see
\cite{SoVeWa08} for a dictionary between lifting fibers and witness
sets).

As expressed in the introduction, the output of the main algorithm
of this paper will be a lifting fiber of the variety defined by the
input system. For this reason, we devote Section \ref{subsec:
properties_of_lifting_points} to discuss a number of properties of
lifting points and lifting fibers which are important for the
algorithm. Then in Section \ref{subsec:
generic_condition_for_a_lifting_point} we obtain a condition on the
coordinates of a point $\bfs p\in\K^{n-s}$ which implies that $\bfs
p$ is a lifting point of $\pi$ (Theorem
\ref{th:_lifting_point_rank_and_finite_morphism}). Finally, in
Section \ref{subsec:
kronecker_representations_from_specializations_of_the_Chow_form} we
show that, taking partial derivatives and specializing a Chow form
of $V$ at the coordinates of linear forms $Y_1\klk Y_{n-s+1}$ as
above and a lifting point $\bfs p$ of $\pi$, we obtain a Kronecker
representation of the lifting fiber $\pi^{-1}(\bfs p)$ and a related
object, called a lifting curve (Propositions \ref{prop:
kronecker_representation_of_J} and \ref{prop:
kronecker_representation_of_K}).
%
%
\subsection{Properties of lifting points}
\label{subsec: properties_of_lifting_points}
Let $\bfs p:=(p_1\klk p_{n-s})\in\K^{n-s}$ be a point as above. Then
$\pi^{-1}(\bfs p)=V\cap \{Y_1-p_1=0,\dots,Y_{n-s}-p_{n-s}=0\}$. We
shall prove that $\bfs p$ is a lifting point of $\pi$ if and only if
the ideal
\begin{equation}\label{def:ideal_J}
\mathcal{J}:=(F_1,\dots,F_s,
Y_1-p_1,\dots,Y_{n-s}-p_{n-s})\subset\K[\bfs X]
\end{equation}
is radical. To this aim, we start with a technical result.

\begin{lema} \label{lemma: rank_fiber}
With hypotheses and notations as above, assume further that $\K[V]$
is a free $R$--module of finite rank $D$, where
$R:=\K[Y_1,\dots,Y_{n-s}]$. Fix $j$ with $0\leq j \leq n-s$ and let
$\mathcal{J}_j\subseteq \K[\bfs X]$ be the ideal
$\mathcal{J}_j:=(F_1,\dots,F_s)+(Y_1-p_1,\dots,Y_j-p_j)$. Then
$\K[\bfs X]/\mathcal{J}_j$ is a free $\K[Y_{j+1}, \dots,
Y_{n-s}]$--module of rank equal to $D$. Moreover, if the coordinate
functions of $V$ defined by $G_1\klk G_D\in\K[\bfs X]$ form a basis
of $\K[V]$ as $R$--module, then they also induce a basis of $\K[\bfs
X]/\mathcal{J}_j$ as $\K[Y_{j+1}, \dots, Y_{n-s}]$--module.
\end{lema}
\begin{proof} It suffices to prove the last assertion.
Let $F\in \K[\bfs X]$. There exist $A_1, \dots,
A_{\scriptscriptstyle D}\in R$ such that
$F=A_1G_1+\cdots+A_{\scriptscriptstyle D}G_{\scriptscriptstyle D}$
in $\K[V]$. Note that
\[
A_i\equiv A_i(p_1,\dots,p_j, Y_{j+1}, \dots, Y_{n-s}) \quad  \mod
(Y_1-p_1, \dots, Y_j-p_j)
 \]
for $1 \leq i \leq D$. Hence, if $B_i:=A_i(p_1,\dots,p_j, Y_{j+1},
\dots, Y_{n-s})$ for $1 \leq i \leq D$, then $F=B_1G_1+ \cdots +
B_{\scriptscriptstyle D}G_{\scriptscriptstyle D}$ in $\K[\bfs
X]/\mathcal{J}_j$. This shows that $G_1, \dots,
G_{\scriptscriptstyle D}$ generate $\K[\bfs X]/\mathcal{J}_j$ as a
$\K[Y_{j+1}, \dots, Y_{n-s}]$--module. Next, suppose that $B_1G_1+
\cdots + B_{\scriptscriptstyle D}G_{\scriptscriptstyle D}=0$ in
$\K[\bfs X]/\mathcal{J}_j$ for certain $B_1, \dots,
B_{\scriptscriptstyle D}\in \K[Y_{j+1}, \dots, Y_{n-s}]$. It follows
that there exist $H_1, \dots, H_j\in \K[\bfs X]$ such that $ B_1G_1+
\cdots + B_{\scriptscriptstyle D}G_{\scriptscriptstyle
D}=H_1(Y_1-p_1) + \cdots +H_j(Y_j-p_j) $ in $\K[V]$. We can write
$H_i=\sum_{k=1}^{D}C_{ik}G_k$ in $\K[V]$ with $C_{ik}\in R$ for
$1\leq i \leq j$ and $1\leq k \leq D$. As a consequence, we obtain
the following identity in $\K[V]$:
\[
\biggl(B_1 - \sum_{k=1}^jC_{k1}(Y_k-p_k)\biggr)G_1 + \cdots +
\biggl(B_{\scriptscriptstyle D} - \sum_{k=1}^jC_{k
{\scriptscriptstyle D}}(Y_k-p_k)\biggr)G_{\scriptscriptstyle D}=0.
\]
Since $G_1, \dots, G_{\scriptscriptstyle D}$ induce a basis of
$\K[V]$ as $R$--module, we see that $B_i=
\sum_{k=1}^jC_{ki}(Y_k-p_k)$ for $1\leq i \leq D$. By substituting
$p_k$ for $Y_k$ in these identities for $1 \leq k \leq j$, we
conclude that $B_i=0$ for $1\leq i \leq D$.  This shows that
$G_1,\dots, G_{\scriptscriptstyle D}$ define $\K[Y_{j+1},\dots,
Y_{n-s}]$--linearly independent elements of $\K[\bfs
X]/\mathcal{J}_j$, which  finishes the proof of the lemma.
\end{proof}

Now we are able to prove that the ideal $\mathcal{J}$ above is
radical if $\bfs p$ is a lifting point of $\pi$.
\begin{lema}\label{lemma: positive_dimension_lifting_fiber}
With  hypotheses and notations as in Lemma \ref{lemma: rank_fiber},
assume further that $\bfs p$ is a lifting point of $\pi$. Then
$\mathcal{J}_j$ is a radical, equidimensional ideal of dimension
$n-s-j$. Further, if $W_j\subseteq \mathbb{A}^n$ is the
$\K$--variety defined by $\mathcal{J}_j$, then the mapping $\pi_j:
W_j \rightarrow \mathbb{A}^{n-s-j}$ defined by $Y_{j+1}, \dots,
Y_{n-s}$ is a finite morphism.
\end{lema}
\begin{proof}
Since $\K[\bfs X]/\mathcal{J}_j$ is a finite
$\K[Y_{j+1},\dots,Y_{n-s}]$--module by Lemma \ref{lemma:
rank_fiber}, we deduce that $\K[\bfs X]/\mathcal{J}_j$ is integral
over $\K[Y_{j+1},\dots,Y_{n-s}]$. This implies that $\dim W_j \leq
n-s-j$. On the other hand, the Principal Ideal theorem (see, e.g.,
\cite[Theorem 10.2]{Eisenbud95}) shows that $\dim W_j\geq n-s-j$,
from which we conclude that $\dim W_j= n-s-j$. By the unmixedness
theorem it follows that $\mathcal{J}_j$ is unmixed. Next, let
$\mathcal{C}$ be an irreducible $\K$--component of $W_j$. We claim
that the restriction $\pi_\mathcal{C}: \mathcal{C} \rightarrow
\mathbb{A}^{n-s-j}$ of $\pi_j$ to $\mathcal{C}$ is a finite
morphism. Indeed, since $\K[Y_{j+1}, \dots, Y_{n-s}]\hookrightarrow
\K[W]$ is an integral extension of rings, so is the extension
$\K[Y_{j+1}, \dots, Y_{n-s}]\rightarrow \K[\mathcal{C}]$ induced by
the dual ring homomorphism ${\pi}_{\mathcal{C}}^*$. Moreover, since
$\dim \mathcal{C} =n-s-j$, it follows that $\pi_{\mathcal{C}}^*$ is
injective and therefore $\pi_{\mathcal{C}}$ is dominant. This proves
the claim, which implies that $\pi_j$ is a finite morphism.

It remains to prove that $\mathcal{J}_j$ is radical. Let
$\mathcal{C}$ be any irreducible $\K$--component of $W_j$. Since the
restriction $\pi_{\mathcal{C}}$ of $\pi_j$ to $\mathcal{C}$ is a
finite morphism, it is surjective, and there exists $\bfs x\in
{\pi}_{\mathcal{C}}^{-1}(p_{j+1}, \dots, p_{n-s})=\mathcal{C}\cap
\{Y_{j+1}-p_{j+1}=0, \dots, Y_{n-s}-p_{n-s}=0\}$. Let $J$ be the
Jacobian determinant of $F_1, \dots, F_s, Y_1-p_1, \dots,
Y_{n-s}-p_{n-s}$ with respect to $X_1, \dots, X_n$ and $M_j$ the
Jacobian matrix of $F_1, \dots,F_s, Y_1-p_1,\dots, Y_j-p_j$ with
respect to $X_1, \dots, X_n$. Since $\bfs p$ is a lifting point of
$\pi$ and $\bfs x\in \pi^{-1}(\bfs p)$, we have $J(\bfs x)\neq 0$,
which implies that $M_j(\bfs x)$ has rank $s+j$. As a consequence,
there exists an $(s+j)\times (s+j)$--minor $m$ of $M_j$ such that
$m(\bfs x)\neq 0$. It follows that the ideal generated by
$\mathcal{J}_j$ and all the $(s+j)\times (s+j)$--minors of $M_j$ is
not contained in $\mathcal{I}(\mathcal{C})$. Then Lemma \ref{lemma:
radicality_criterion} shows that $\mathcal{J}_j$ is radical.
\end{proof}

Let $\bfs p:=(p_1,\dots,p_{n-s})\in \K^{n-s}$ be a lifting point of
$\pi$. In the sequel we shall interpret $Y_1\klk Y_{n-s}$ either as
linear forms in $X_1, \dots, X_n$ or as indeterminates over $\K$,
each interpretation being clear from the context. By Lemma
\ref{lemma: positive_dimension_lifting_fiber}, the zero--dimensional
ideal $\mathcal{J}:=(F_1, \dots, F_s, Y_1-p_1, \dots,
Y_{n-s}-p_{n-s})\subset \K[\bfs X]$ is radical and therefore it is
the vanishing ideal of the lifting fiber $V_{\bfs p}:=\pi^{-1}(\bfs
p)$. Now, for the main algorithm of this paper we shall consider
certain curve associated to $\bfs p$ and $V$, which we now
introduce. Let $\bfs p^*:=(p_1,\dots,p_{n-s-1})$ and let $W_{\bfs
p^*}\subset \mathbb{A}^n$ be the $\K$--variety defined by the ideal
$$\mathcal{K}:=(F_1,\dots, F_s, Y_1-p_1, \dots,
Y_{n-s-1}-p_{n-s-1})\subseteq \K[\bfs X].$$
According to Lemma \ref{lemma: positive_dimension_lifting_fiber},
$\mathcal{K}$ is a radical, equidimensional ideal of dimension $1$
and the mapping ${\pi}_1: W_{\bfs p^*}\rightarrow \mathbb{A}^1$
defined by $Y_{n-s}$ is a finite morphism. We call $W_{\bfs p^*}$
the \emph{lifting curve} defined by $\bfs p^*$.

We shall identify $V_{\bfs p}$ with a zero--dimensional subvariety
of $\mathbb{A}^s$ and $W_{\bfs p^*}$ with a curve of
$\mathbb{A}^{s+1}$ as follows. For simplicity of notations, we shall
denote by $F_i(Y_1,\dots,Y_n)$ or $F_i(\bfs Y)$ the element of
$\K[Y_1,\dots,\! Y_n]$ obtained by rewriting $F_i(X_1,\dots, X_n)$
in the variables $Y_1,\dots,Y_n$.
\begin{lema}\label{lemma: lifting_curve_isomorphism}
With hypotheses as in Lemma \ref{lemma:
positive_dimension_lifting_fiber}, the following assertions hold:
\begin{itemize}
  \item the  polynomials $F_1(\bfs p, Y_{n-s+1},\dots, Y_n),\dots,
F_s(\bfs p, Y_{n-s+1},\dots, Y_n)$ generate a radical,
zero--dimensional ideal $\overline{\mathcal{J}}$ of
$\K[Y_{n-s+1},\dots,Y_n]$, and the $\K$--variety
$\mathcal{V}(\overline{\mathcal{J}})\subset \mathbb{A}^s$ is
isomorphic to $V_{\bfs p}$. Further, $\K[Y_{n-s+1}, \dots,
Y_n]/\overline{\mathcal{J}}$ is a $\K$--vector space of dimension
$\mathrm{rank}_R\K[V]$;
  \item the polynomials $F_1(\bfs p^*, Y_{n-s},\dots, Y_n), \dots,
F_s(\bfs p^*, Y_{n-s},\dots, Y_n)$ generate a radical,
equidimensional ideal $\overline{\mathcal{K}}$ of
$\K[Y_{n-s},\dots,Y_n]$ of dimension $1$, and the $\K$--variety
$\mathcal{V}(\overline{\mathcal{K}})\subset \mathbb{A}^{s+1}$ is
isomorphic to $W_{\bfs p^*}$. Further, $Y_{n-s},\dots, Y_n$ are in
Noether position with respect to $\overline{\mathcal{K}}$ and
$\K[Y_{n-s}, \dots, Y_n]/\overline{\mathcal{K}}$  is a free
$\K[Y_{n-s}]$--module of rank equal to $\mathrm{rank}_R\K[V]$.
\end{itemize}
\end{lema}
\begin{proof}
Clearly, we have an isomorphism of  $\K$--algebras
$$
\K[Y_1,\dots,Y_n]/\mathcal{J}\cong \K[Y_{n-s+1},\dots
,Y_n]/\overline{\mathcal{J}},
$$
which maps $F(Y_1,\dots,Y_n) \mod \mathcal{J}$ to $F(\bfs p,
Y_{n-s+1}, \dots, Y_n) \mod \overline{\mathcal{J}}$. It follows that
$\overline{\mathcal{J}}$ is radical and zero--dimensional, since so
is $\mathcal{J}$. Therefore, this is an isomorphism between the
coordinate rings of $V_{\bfs p}$ and
$\mathcal{V}(\overline{\mathcal{J}})$, and proves that $V_{\bfs p}$
and $\mathcal{V}(\overline{\mathcal{J}})$ are isomorphic. Similarly,
we have an isomorphism of $\K$--algebras
$$
\K[Y_1,\dots,Y_n]/\mathcal{K}\cong \K[Y_{n-s},\dots
,Y_n]/\overline{\mathcal{K}},
$$
which maps $F(Y_1,\dots,Y_n) \mod \mathcal{K}$ to $F(\bfs p^*,
Y_{n-s}, \dots, Y_n) \mod \overline{\mathcal{K}}$. Arguing as before
we conclude that $\overline{\mathcal{K}}$ is radical and $W_{\bfs
p^*}$ and $\mathcal{V}(\overline{\mathcal{K}})$ are isomorphic.
Further, we have that $Y_j$ is integral over $\K[Y_{n-s}]$ modulo
$\overline{\mathcal{K}}$ for $n-s+1\leq j \leq n$, which proves that
$Y_{n-s},\dots, Y_n$ are in Noether position with respect to
$\overline{\mathcal{K}}$. Finally, the assertions concerning
freeness and ranks follow by Lemma \ref{lemma: rank_fiber}, which
completes the proof of the lemma.
\end{proof}

A critical step in our main algorithm is to obtain a Kronecker
representation of a lifting curve $W_{\bfs p^*}$ from one of a
lifting fiber $V_{\bfs p}$. This will be achieved by considering a
symbolic version of the Newton method, which requires that the
polynomials $F_1(\bfs p, Y_{n-s+1},\dots, Y_n), \dots, F_s(\bfs p,
Y_{n-s+1},\dots, Y_n)$ define all points of $V_{\bfs p}$ by
transversal cuts. Further, in Section \ref{subsec:
lifting_the_integers} we shall lift a Kronecker representation of
the output lifting fiber modulo a prime number $p$, which also
requires such a transversality condition. As the next result shows,
this is guaranteed if $\bfs p$ is a lifting point $\pi$.
\begin{lema}\label{lemma: jacobian_nonzero}
With hypotheses as in Lemma \ref{lemma:
positive_dimension_lifting_fiber}, the Jacobian determinant
$\overline{J}$ of the polynomials $F_1(\bfs p, Y_{n-s+1},\dots,
Y_n), \dots, F_s(\bfs p, Y_{n-s+1},\dots, Y_n)$ with respect to
$Y_{n-s+1},\dots, Y_n$ is invertible in $\K[Y_{n-s+1},
\dots,Y_n]/\overline{\mathcal{J}}$.
\end{lema}
\begin{proof}
Let $\mathcal{P}_1, \dots, \mathcal{P}_N$ be the minimal prime
ideals of $\overline{\mathcal{J}}$. Since $\overline{\mathcal{J}}$
is radical, by Lemma \ref{lemma: radicality_criterion} we deduce
that $\overline{J}\notin \mathcal{P}_i$ for $1\leq i \leq N$. As
$\overline{\mathcal{J}}$ is of dimension zero, each $\mathcal{P}_i$
is a maximal ideal of $\K[Y_{n-s+1}, \dots, Y_n]$, which implies
that $\overline{J}$ is a unit in $\K[Y_{n-s+1}, \dots,
Y_n]/\mathcal{P}_i$ for $1\leq i \leq N$. By the Chinese remainder
theorem we conclude that $\overline{J}$ is a unit in $\K[Y_{n-s+1},
\dots, Y_n]/\overline{\mathcal{J}}$, which finishes the proof of the
lemma.
\end{proof}

Finally, assuming that $F_1\klk F_s$ form a regular sequence, we
shall need to see that this is preserved when specializing
$(Y_1,\dots, Y_{n-s})$ at a lifting point $\bfs p$. We have the
following result.
\begin{lema}\label{lema: reduced_regular_sequence}
Assume that $F_1, \dots, F_s$ form a regular sequence of $\K[\bfs
X]$ and $Y_1,\dots,Y_n$ are linear forms of $\K[\bfs X]$ in Noether
position with respect to $V_i:=\{F_1=0,\dots,F_i=0\}$ for $1\leq i
\leq s$. Then $F_1(\bfs p,Y_{n-s+1},\dots, Y_n), \dots, F_s(\bfs p,
Y_{n-s+1},\dots,Y_n)$ form a regular sequence of
$\K[Y_{n-s+1},\dots,Y_n]$ for any $\bfs p\in\K^{n-s}$.
\end{lema}
\begin{proof}
It suffices to show that $F_1(\bfs p,Y_{n-s+1},\dots, Y_n), \dots,
F_i(\bfs p, Y_{n-s+1},\dots,Y_n)$ define a subvariety  of
$\mathbb{A}^s$ of dimension $s-i$ for $1\leq i \leq s$. Let
$L_s:=\{Y_1=p_1,\dots,Y_{n-s}=p_{n-s}\}\subset \mathbb{A}^n$ and
$\pi_i:V_i\rightarrow \mathbb{A}^{n-i}$ the mapping defined by
$Y_1,\dots,Y_{n-i}$. Then $V_i\cap
L_s=\pi_i^{-1}\bigl(\{Y_1=p_1,\dots,Y_{n-s}=p_{n-s}\}\bigr)$. Since
$\pi_i$ is a finite morphism,  we have that $\dim V_i\cap
L_s=\dim_{\mathbb{A}^{n-i}} \{Y_1=p_1,\dots,t_{n-s}=Y_{n-s}\}=s-i$,
and the conclusion of the lemma follows by noting that
$\bigl\{F_1(\bfs p,Y_{n-s+1},\dots, Y_n)=0, \dots, F_i(\bfs p,
Y_{n-s+1},\dots,Y_n)=0\bigr\}$ and $V_i\cap L_s$ are isomorphic
varieties.
 \end{proof}
%
%
\subsection{A condition for lifting points}
\label{subsec: generic_condition_for_a_lifting_point}
In this section we obtain a condition for the coordinates of a point
$\bfs p\in \K^{n-s}$ which implies that it is a lifting point of
$\pi$. A first step in this direction is provided by the following
characterization of the notion of lifting point, which also proves
that the concept is independent of the polynomials $F_1\klk F_s$
generating the vanishing ideal of the variety $V$.
\begin{prop} \label{prop: lifting_point_and_rank}
Assume that $\K[V]$ is a free $R$--module of finite rank
$D:=\mathrm{rank}_R\K[V]$. Then $\#\pi^{-1}(\bfs p)\le D$ for any
$\bfs p\in \K^{n-s}$, with equality if and only if $\bfs p$ is a
lifting point.
\end{prop}
\begin{proof}
Let $\bfs p:=(p_1, \dots, p_{n-s})$ and let
$\mathcal{J}\subset\K[\bfs X]$ be the zero--dimensional ideal of
\eqref{def:ideal_J}. By Lemma \ref{lemma: rank_fiber} we have
$\dim_{\K} \K[\bfs X]/\mathcal{J}=D$.
Since $\#\pi^{-1}(\bfs p)= \dim_{\K} \K[\bfs X]/\sqrt{\mathcal{J}}$,
the inequality of the statement follows.

Now we prove the characterization of lifting points. Let
$\chi_{\scriptscriptstyle J}\in R[T]$ be the characteristic
polynomial of $J$ modulo $\mathcal{I}$. Since by Lemma \ref{lemma:
rank_fiber} a basis of $\K[V]$ as $R$--module induces a basis of
$\K[\bfs X]/\mathcal{J}$ as $\K$--vector space, it is easy to see
that $\chi_{\scriptscriptstyle J}(\bfs p,T)$ is the characteristic
polynomial of $J$ modulo $\mathcal{J}$. Let
$\mu:=\chi_{\scriptscriptstyle J}(0)$ be the constant term of
$\chi_{\scriptscriptstyle J}$, so that $\mu(\bfs p)$ is the constant
term of $\chi_{\scriptscriptstyle J}(\bfs p,T)$. By the
Nullstellensatz, $\bfs p$ is a lifting point of $\pi$ if and only if
the equality of ideals $\mathcal{J} + (J) = (1)$ holds in $\K[\bfs
X]$. Note that, by the unmixedness theorem, $\mathcal{J}$ is
unmixed. Then $\mathcal{J} + (J) = (1)$ if and only if $\mu(\bfs
p)\neq 0$ (see, e.g., \cite[Lemma 2.1(c)]{DuLe08}). Further,
$\mu(\bfs p)\neq 0$ if and only if $J$ is not a zero divisor in
$\K[\bfs X]/\mathcal{J}$ (see, e.g., \cite[Lemma 2.1(b)]{DuLe08}),
which in turn holds if and only if $J$ is not contained in any
associated prime of $\mathcal{J}$ (see, e.g., \cite[Theorem
6.1(ii)]{Matsumura86}). Finally, by Lemma \ref{lemma:
radicality_criterion}, the latter is equivalent to the radicality of
$\mathcal{J}$. Summarizing, we have that $\bfs p$ is a lifting point
of $\pi$ if and only if $\mathcal{J}$ is a radical ideal. On the
other hand, $\mathcal{J}$ is radical if and only if $\dim_\K\K[\bfs
X]/\mathcal{J}= \#\pi^{-1}(\bfs p)$. Since $\dim_{\K} \K[\bfs
X]/\mathcal{J}=D$, the proposition follows.
\end{proof}

Let $\bfs\Lambda:=(\Lambda_{ij})_{1\leq i \leq n-s+1, 1\leq j \leq
n}$, $\bfs Z:=(Z_1,\dots, Z_{n-s+1})$  and let
$P_{\scriptscriptstyle V}\in \K[\bfs\Lambda, \bfs Z]$ be a Chow form
of $V$. Denote as before by $A_{\scriptscriptstyle V}\in
\K[\bfs\Lambda_1, \dots, \bfs\Lambda_{n-s}]$ the (nonzero)
coefficient of the monomial $Z_{n-s+1}^{\delta}$ in
$P_{\scriptscriptstyle V}$, and by $\rho_{\scriptscriptstyle V} \in
\K[\bfs\Lambda, Z_1,\dots, Z_{n-s}]$ the discriminant of
$P_{\scriptscriptstyle V}$ with respect to $Z_{n-s+1}$. Consider the
quotient ring $\K[\bfs\Lambda, \bfs Z]/(P_{\scriptscriptstyle V})$
as a $\K[\bfs\Lambda, \bfs Z]$--algebra by means of the canonical
ring homomorphism $\K[\bfs\Lambda, \bfs Z]\rightarrow
\K[\bfs\Lambda, \bfs Z]/(P_{\scriptscriptstyle V})$. Further,
consider as before $\K[V][\bfs\Lambda]$ as a $\K[\bfs\Lambda, \bfs
Z]$--algebra by means of the ring homomorphism $\K[\bfs\Lambda, \bfs
Z]\rightarrow \K[V][\bfs\Lambda]$ which maps any $F \in
\K[\bfs\Lambda, \bfs Z]$ to $F(\bfs\Lambda, \bfs\Lambda\bfs\xi)$. By
Lemma \ref{lemma: nonzero_discriminant}, the polynomial $\partial
P_{\scriptscriptstyle V}/\partial Z_{n-s+1}$ is nonzero and hence
\[
S:=\bigl\{(\partial P_{\scriptscriptstyle V}/\partial
Z_{n-s+1})^{\eta}: \eta \geq 0 \bigr\}
\]
is a multiplicatively closed subset of $\K[\bfs\Lambda, \bfs Z]$. We
consider the localizations
\begin{align*}
\K[\bfs\Lambda, \bfs Z]_{\partial P_{\scriptscriptstyle V}/\partial Z_{n-s+1}} &
:=S^{-1}\K[\bfs\Lambda, \bfs Z], \\
\bigl(\K[\bfs\Lambda, \bfs Z]/(P_{\scriptscriptstyle
V})\bigr)_{\partial P_{\scriptscriptstyle V}/\partial Z_{n-s+1}}
&:=  S^{-1}\K[\bfs\Lambda, \bfs Z]/(P_{\scriptscriptstyle V}), \\
\K[V][\bfs\Lambda]_{\partial P_{\scriptscriptstyle V}/\partial
Z_{n-s+1}} &:=  S^{-1}\K[V][\bfs\Lambda].
\end{align*}
Let
$ \K[\bfs\Lambda,\!\bfs Z]/(P_{\scriptscriptstyle V})\!\rightarrow\!
\K[V][\mathbf{\Lambda]} $
be the $\K[\bfs\Lambda, \bfs Z]$--algebra homomorphism that maps
$[Z_i]_{\textrm{mod} P_{\scriptscriptstyle V}}$ to
$\bfs\Lambda_i\cdot\bfs\xi$ for  $1 \leq i \leq n-s+1$ and consider
the $\K[\bfs\Lambda,\bfs Z]_{\partial P_{\scriptscriptstyle
V}/\partial Z_{n-s+1}}$--algebra homomorphism
\begin{equation}\label{def:Phi_homomorphism}
\Phi: \bigl(\K[\bfs\Lambda,\bfs Z]/(P_{\scriptscriptstyle
V})\bigr)_{\partial P_{\scriptscriptstyle V}/\partial Z_{n-s+1}}
\rightarrow \K[V][\bfs\Lambda]_{\partial P_{\scriptscriptstyle
V}/\partial Z_{n-s+1}}.
\end{equation}
that extends this map. The next result asserts that $\Phi$ is an
isomorphism.
\begin{lema}
$\Phi$ is an isomorphism of  $\K[\bfs\Lambda,\bfs Z]_{\partial
P_{\scriptscriptstyle V}/\partial Z_{n-s+1}}$--algebras.
\end{lema}
\begin{proof}
By the minimality of $P_{\scriptscriptstyle V}$ the homomorphism
$\K[\bfs\Lambda,\bfs Z]/(P_{\scriptscriptstyle V})\rightarrow
\K[V][\mathbf{\Lambda]}$ above is injective, and thus so is $\Phi$.
To prove surjectivity, by
\eqref{ident:partial_derivatives_chow_form} we have
$ \xi_k\!=-\frac{\partial P_{\scriptscriptstyle V}\!/\partial
\Lambda_{n-s+1,k}(\bfs\Lambda,\bfs\Lambda\bfs\xi)}{ \partial
P_{\scriptscriptstyle V}/\partial Z_{n-s+1}} $ in
$\K[V][\bfs\Lambda]_{\partial P_{\scriptscriptstyle V}/\partial
Z_{n-s+1}}$ for  $1\leq k \leq n$. It follows that
\begin{equation}\label{eq:Phi_surjective}
\xi_k=\Phi \left( -\frac{[\partial P_{\scriptscriptstyle V}
/\partial \Lambda_{n-s+1,k}]_{\textrm{mod} P_{\scriptscriptstyle
V}}}{\partial P_{\scriptscriptstyle V}/\partial Z_{n-s+1}}  \right)
\end{equation}
for  $1\leq k \leq n$. Since  $\xi_1, \dots, \xi_n$ generate
$\K[V][\bfs\Lambda]_{\partial P_{\scriptscriptstyle V}/\partial
Z_{n-s+1}}$ as a $\K[\bfs\Lambda,\bfs Z]_{\partial
P_{\scriptscriptstyle V}/\partial Z_{n-s+1}}$--algebra, the lemma
follows.
\end{proof}

We shall also need the following technical result.
\begin{lema}\label{lemma: cardinality_of_the_fiber} For any
$F\in \K[\bfs X]$, let $F_\Lambda\in \K[\bfs\Lambda, \bfs Z]$ be any
polynomial such that
\begin{equation}\label{eq:cardinality_fiber_lemma_1}
F\left(  -\frac{\partial P_{\scriptscriptstyle V} /\partial
\Lambda_{n-s+1,1}}{\partial P_{\scriptscriptstyle V}/\partial
Z_{n-s+1}},\dots, -\frac{\partial P_{\scriptscriptstyle V} /\partial
\Lambda_{n-s+1,n}}{\partial P_{\scriptscriptstyle V}/\partial
Z_{n-s+1}} \right)=\frac{F_{\scriptscriptstyle \Lambda}}{(\partial
P_{\scriptscriptstyle V}/\partial Z_{n-s+1})^\eta}
\end{equation}
for some $\eta\in \mathbb{N}$. If $F$ vanishes on $V$, then
$F_{\scriptscriptstyle \Lambda}$ is a multiple of
$P_{\scriptscriptstyle V}$. Further, for $1\leq i \leq n-s+1$, the
following polynomial $H_i\in \mathbb{Z}[\bfs\Lambda, \bfs Z]$ is a
multiple of $P_{\scriptscriptstyle V}$:
\begin{equation}\label{eq:cardinality_fiber_lemma_2}
H_i:=\frac{\partial P_{\scriptscriptstyle V}}{\partial
Z_{n-s+1}}\,Z_i + \sum_{j=1}^{n}\Lambda_{ij}\frac{\partial
P_{\scriptscriptstyle V} }{\partial \Lambda_{n-s+1,j}}.
\end{equation}
\end{lema}
\begin{proof}
Considering \eqref{eq:cardinality_fiber_lemma_1} modulo
$P_{\scriptscriptstyle V}$ and applying $\Phi$ to both sides, by
\eqref{eq:Phi_surjective} we see that
\[
F(\bfs\xi)=\frac{F_{\scriptstyle\Lambda}(\bfs\Lambda,\bfs\Lambda\bfs\xi)}{(\partial
P_{\scriptscriptstyle V}/\partial Z_{n-s+1})^\eta}.
\]
Since $F(\bfs\xi)=0$ and $\partial P_{\scriptscriptstyle V}/\partial
Z_{n-s+1}$ is not a zero divisor of $\K[V][\bfs\Lambda]$ (Lemma
\ref{lemma: non_zero_divisor}), we conclude that
$F_{\scriptstyle\Lambda}(\bfs\Lambda,\bfs\Lambda\bfs\xi)=0$. By the
minimality of $P_{\scriptscriptstyle V}$ the first assertion
follows.

To prove the second assertion, we observe that
\begin{align}\label{inverse_phi_equation}
[Z_i]_{\textrm{mod} P_{\scriptscriptstyle V}} =
\Phi^{-1}(\bfs\Lambda_i \cdot \bfs\xi)  =  \sum_{j=1}^n\Lambda_{ij}
\,\Phi^{-1}(\xi_j)
\end{align}
for $1 \leq i \leq n-s+1$. By this and \eqref{eq:Phi_surjective} it
follows that
$$
[Z_i]_{\textrm{mod} P_{\scriptscriptstyle
V}}=-\sum_{j=1}^n\Lambda_{ij} \frac{[\partial P_{\scriptscriptstyle
V} /\partial \Lambda_{n-s+1,j}]_{\textrm{mod} P_{\scriptscriptstyle
V}}}{\partial P_{\scriptscriptstyle V}/\partial Z_{n-s+1}}
$$
for  $1 \leq i \leq n-s+1$, which readily implies the second
assertion of the lemma.
\end{proof}

The next result, combined with Proposition \ref{prop:
lifting_point_and_rank}, will yield the condition characterizing
lifting points we are looking for.
\begin{prop}\label{prop: cardinality_of_the_fiber}
Let $\bfs\lambda\in \K^{(n-s+1)n}$ and $\bfs p\in  \K^{n-s}$ be such
that $A_{\scriptscriptstyle
V}(\bfs\lambda^*)\rho_{\scriptscriptstyle V}(\bfs\lambda, \bfs
p)\neq 0$, let $Y_i:=\bfs\lambda_i \cdot \bfs X$ for $1\leq i \leq
n-s$ and $\pi: V\rightarrow \mathbb{A}^{n-s}$ the mapping defined by
$Y_1,\dots, Y_{n-s}$. Then $\# \pi^{-1}(\bfs p)=\delta$.
\end{prop}
\begin{proof}
By the choice of $\bfs\lambda$, the polynomial
$P_{\scriptscriptstyle V}(\bfs\lambda, \bfs p, Z_{n-s+1})$ has
degree $\delta$. Since
\[
\rho_{\scriptscriptstyle V}(\bfs\lambda, \bfs
p)=\mbox{Res}_{Z_{n-s+1}}\biggl(P_{\scriptscriptstyle
V}(\bfs\lambda,\bfs p, Z_{n-s+1}),\frac{\partial
P_{\scriptscriptstyle V}}{\partial Z_{n-s+1}}(\bfs\lambda, \bfs p,
Z_{n-s+1})\biggr)
\]
and $\rho_{\scriptscriptstyle V}(\bfs\lambda, \bfs p)\neq 0$, the
polynomial $P_{\scriptscriptstyle V}(\bfs\lambda, \bfs p,
Z_{n-s+1})$ is separable. Let $z_1, \dots, z_{\delta} \in
\overline{\K}$ be the $\delta$ different roots of
$P_{\scriptscriptstyle V}(\bfs\lambda, \bfs p, Z_{n-s+1})$ and set
$\bfs y^k:=(\bfs p,z_k)$ for $1 \leq k \leq \delta$. We have that
$\partial P_{\scriptscriptstyle V}/\partial
Z_{n-s+1}(\bfs\lambda,\bfs y^k)\neq 0$ for $1 \leq k \leq \delta$,
and thus the point
\[
\bfs x^k:=\biggl(-\frac{\partial P_{\scriptscriptstyle V}/\partial
\Lambda_{n-s+1,1}(\bfs\lambda,\bfs y^k)}{\partial
P_{\scriptscriptstyle V}/\partial Z_{n-s+1}(\bfs\lambda,\bfs
y^k)},\dots, -\frac{\partial P_{\scriptscriptstyle V}/\partial
\Lambda_{n-s+1,n}(\bfs\lambda,\bfs y^k)}{\partial
P_{\scriptscriptstyle V}/\partial Z_{n-s+1}(\bfs\lambda,\bfs y^k)}
\biggr) \in \mathbb{A}^n
\]
is well defined for $1\leq k \leq \delta$.

We claim that $\bfs x^1, \dots,\bfs x^\delta$ are pairwise distinct
and $\pi^{-1}(\bfs p)=\{\bfs x^1, \dots,\bfs x^\delta\}$. Indeed,
let $F\in \K[\bfs X]$ be any polynomial vanishing on $V$ and
$F_{\scriptstyle\Lambda}\in \K[\bfs\Lambda,\bfs Z]$ a corresponding
polynomial according to \eqref{eq:cardinality_fiber_lemma_1}. By
Lemma \ref{lemma: cardinality_of_the_fiber} we have
$F_{\scriptstyle\Lambda}(\bfs\lambda,\bfs y^k)=0$, and thus $F(\bfs
x^k)=0$, for $1 \leq k \leq \delta$. This proves that $\bfs x^1,
\dots,\bfs x^\delta$ belong to $V$. Further, Lemma \ref{lemma:
cardinality_of_the_fiber} also shows that
\[
H_i(\bfs\lambda,\bfs y^k)=\frac{\partial P_{\scriptscriptstyle
V}}{\partial Z_{n-s+1}}(\bfs\lambda,\bfs y^k)
y^k_i+\sum_{j=1}^n\lambda_{ij}\frac{\partial P_{\scriptscriptstyle
V}}{\partial \Lambda_{n-s+1,j}}(\bfs\lambda,\bfs y^k)=0
\]
for $1\leq i \leq n-s+1$ and $1 \leq k \leq \delta$. By the
definition of $\bfs x^k$ it follows that
\begin{equation}\label{eq:cardinality_fiber_prop}
y^k_i=\bfs\lambda_i\cdot \bfs x^k \quad (1\leq i\leq n-s+1).
\end{equation}
Since $y^k_i=p_i$ for $1\leq i \leq n-s$,
\eqref{eq:cardinality_fiber_prop} implies that $\pi(\bfs x^k)=\bfs
p$ and $z_k=\bfs\lambda_{n-s+1}\cdot \bfs x^k$ for $1\leq k \leq
\delta$. Since the $z_k$ are pairwise distinct, we deduce that so
are the $\bfs x^k$. This proves that $\#\pi^{-1}(\bfs p)\geq
\delta$. On the other hand, since $\pi$ is a finite morphism
(Proposition \ref{prop: finite_fiber_and_finite_morphism}), the
fiber $\pi^{-1}(\bfs p)$ is finite, and by (\ref{bezout_inequality})
we have
\[
\# \pi^{-1}(\bfs p)=\deg \big(V \cap
\{Y_1-p_1=0,\dots,Y_{n-s}-p_{n-s}=0\}\bigl)\leq \deg V=\delta,
\]
which concludes the proof of the claim and the proposition.
\end{proof}

Now we are able to state the main result of this section.
\begin{teo}\label{th:_lifting_point_rank_and_finite_morphism}
Let $\bfs\lambda\in \K^{(n-s+1)n}$ and $\bfs p\in \K^{n-s}$ be such
that $A_{\scriptscriptstyle
V}(\bfs\lambda^*)\rho_{\scriptscriptstyle V}(\bfs\lambda, \bfs
p)\neq 0$. Let $Y_i:=\bfs\lambda_i \cdot \bfs X$ for $1\leq i \leq
n-s+1$ and $R:=\K[Y_1, \dots, Y_{n-s}]$. Then:
\begin{itemize}
  \item the mapping $\pi: V \rightarrow \mathbb{A}^{n-s}$ defined by
  $Y_1, \dots, Y_{n-s}$ is a finite morphism and $Y_{n-s+1}$ induces
  a primitive element of the ring extension $R\hookrightarrow
  \K[V]$;
  \item if $\K[V]$ is a free $R$--module, then $\mathrm{rank}_{\scriptscriptstyle R}\K[V]=\delta$;
  \item  $\bfs p$ is a lifting point of $\pi$ and $Y_{n-s+1}$
  induces a primitive element of $\pi^{-1}(\bfs p)$.
\end{itemize}
\end{teo}
\begin{proof}
Proposition \ref{prop: finite_fiber_and_finite_morphism} proves the
first assertion. Combining Propositions \ref{prop:
finite_fiber_and_finite_morphism}, \ref{prop:
lifting_point_and_rank} and \ref{prop: cardinality_of_the_fiber} we
deduce that $\delta=\#\pi^{-1}(\bfs p)\le
\mathrm{rank}_{\scriptscriptstyle R}\K[V]\le \delta$. It follows
that $\#\pi^{-1}(\bfs p)=\delta$ and $\bfs p$ is a lifting point of
$\pi$. Next, let $\bfs p:=(p_1,\dots, p_{n-s})$. By substituting
$\bfs\lambda$ for $\bfs\Lambda$ and $p_1,\dots,p_{n-s}$ for
$\bfs\lambda_1 \cdot \bfs\xi, \dots, \bfs\lambda_{n-s} \cdot
\bfs\xi$  in \eqref{ident:discriminant_chow_form}, we deduce that
\[
\rho_{\scriptscriptstyle V}(\bfs\lambda,\bfs
p)\xi_k=R_k(\bfs\lambda,\bfs p, \bfs\lambda_{n-s+1} \cdot \bfs\xi)
\]
in $\pi^{-1}(\bfs p)$ for $1\leq k \leq n$. Since
$\rho_{\scriptscriptstyle V}(\bfs\lambda,\bfs p)\neq 0$, we conclude
that $\K\bigl[\pi^{-1}(\bfs p)\bigr]=\K[\bfs\lambda_{n-s+1} \cdot
\bfs\xi]$, which proves that $Y_{n-s+1}$ induces a primitive element
of $\pi^{-1}(\bfs p)$.
\end{proof}
%
%
\subsection{Kronecker representations from specializations of the Chow form}
 \label{subsec: kronecker_representations_from_specializations_of_the_Chow_form}
Let be given $\bfs\lambda:=(\lambda_{ij})_{1\leq i \leq n-s+1, 1\leq
j \leq n}\in \K^{(n-s+1)n}$ and $\bfs p:=(p_1,\dots,p_{n-s})\in
\K^{n-s}$ satisfying the hypotheses of Proposition \ref{prop:
lifting_point_and_rank} and Theorem
\ref{th:_lifting_point_rank_and_finite_morphism}. Define
$Y_i:=\bfs\lambda_i\cdot \bfs X$ for $1\leq i \leq n-s+1$, and let
$R:=\K[Y_1, \dots, Y_{n-s}]$ and $B:=\K[V]$. Assume that we are also
given linear forms $Y_{n-s+2},\dots, Y_n\in \K[\bfs X]$ such that
$Y_1\klk Y_n$ are linearly independent. Then
   \begin{itemize}
     \item $Y_1,\dots, Y_n$ are in Noether position with respect to
     $\mathcal{I}$;
     \item $\bfs p$ is a lifting point of the finite morphism
     $\pi:V\rightarrow\mathbb{A}^{n-s}$ defined by
     $Y_1,\dots,Y_{n-s}$;
     \item $B$ is a free $R$--module of finite rank equal to $\delta$.
   \end{itemize}
We shall show that Kronecker representations of the definining
ideals of $V$, the lifting fiber $V_{\bfs p}$ and the lifting curve
$W_{\bfs p^*}$ can be obtained by specializing any Chow form of $V$.
This will provide a criterion to check that the modular reductions
considered during our main algorithm behave properly.

Let $P_{\scriptscriptstyle V}\in \K[\bfs\Lambda, \bfs Z]$ be a Chow
form of $V$, and let $A_{\scriptscriptstyle V}\in
\K[\bfs\Lambda_1,\dots,\bfs\Lambda_{n-s}]$ and $\rho_{V}\in
\K[\bfs\Lambda,Z_1,\dots,Z_{n-s}]$ be defined as in Section
\ref{subsec: generic_condition_for_a_lifting_point}. By
\eqref{ident:chow form_fundamental} and
\eqref{ident:partial_derivatives_chow_form}, we have
\begin{equation}\label{ident:chow_form__fundamental_1}
P_{\scriptscriptstyle V}(\bfs\Lambda,\bfs\Lambda\bfs\xi)  =0,\quad
\frac{\partial P_{\scriptscriptstyle V}}{\partial
Z_{n-s+1}}(\bfs\Lambda,\bfs\Lambda\bfs\xi)\xi_k + \frac{\partial
P_{\scriptscriptstyle V}}{\partial
\Lambda_{n-s+1,k}}(\bfs\Lambda,\bfs\Lambda\bfs\xi)=0 \quad (1\leq k
\leq n),
\end{equation}
in $\K[V][\bfs\Lambda]$. Let $T$ be a new indeterminate and define
$Q,W_{n-s+2}\klk W_n\in R[T]$ by
$$
Q:=\frac{P_{\scriptscriptstyle V}(\bfs\lambda, Y_1, \dots,
Y_{n-s},T)}{A_{\scriptscriptstyle V}(\bfs\lambda^*)},\quad
W_j:=-\sum_{k=1}^{n}\frac{\lambda_{jk}}{A_{\scriptscriptstyle
V}(\bfs\lambda^*)}\frac{\partial P_{\scriptscriptstyle V}}{\partial
\Lambda_{n-s+1,k}}(\bfs\lambda, Y_1, \dots, Y_{n-s},T)
$$
for $n-s+2\leq j \leq n$. Substituting $\bfs\lambda$ for
$\bfs\Lambda$ in (\ref{ident:chow_form__fundamental_1}) we deduce
that
\begin{equation}\label{eq:kronecker_representation_I}
Q(Y_{n-s+1}) \in \mathcal{I},\quad
Q'(Y_{n-s+1})Y_j-W_j(Y_{n-s+1})\in \mathcal{I} \quad (n-s+2\leq j
\leq n),
\end{equation}
where $Q'$ denotes the first derivative of $Q$ with respect to $T$.

Note that $Q$ is a monic polynomial of degree $\delta$ and $\deg
W_j< \delta$ for $n-s+2\leq j \leq n$. On the other hand, by the
choice of $\bfs\lambda$ we have that the discriminant of $Q$, which
is equal to ${\rho_{\scriptscriptstyle V}(\bfs\lambda, Y_1,\dots,
Y_{n-s})}/{A_{\scriptscriptstyle V}(\bfs\lambda^*)^{2\delta-1}}$, is
a nonzero element of $R$. Thus $Q$ is square--free and $Q'$ is
invertible modulo $Q$. In particular, $Q'(Y_{n-s+1})$ is invertible
in $B':=R'[Y_{n-s+1}, \dots, Y_n]/\mathcal{I}^e$, and
\eqref{eq:kronecker_representation_I} shows that the homomorphism of
$R'$--algebras
$R'[T]/(Q) \rightarrow B'$,
which maps $T \mod Q$ to $Y_{n-s+1} \mod \mathcal{I}^e$, is
surjective. This means that $Y_{n-s+1}$ is a primitive element for
$\mathcal{I}$. On the other hand, since $\dim_{R'}B'=\delta$, the
above homomorphism is an isomorphism. We conclude that $Q$ is the
minimal polynomial of $Y_{n-s+1}$ over $R'$ modulo $\mathcal{I}^e$,
and we have the following identity of ideals in $R'[Y_{n-s+1},
\dots, Y_n]$:
$$
\mathcal{I}^e\!=\!\bigl(Q(Y_{n-s+1}),
\!Q'(Y_{n-s+1})Y_{n-s+2}-W_{n-s+2}(Y_{n-s+1}),
\dots,\!Q'(Y_{n-s+1})Y_n-W_n(Y_{n-s+1})\bigr).
$$
Further, by construction $\deg_TW_j\le\delta-1$ for $n-s+2\le j\le
n$. As a consequence, we obtain the following result.
\begin{prop}\label{prop: kronecker_representation_of_I}
The polynomials $Q, W_{n-s+2}, \dots, W_n$ form the Kronecker
representation of $\mathcal{I}$ with primitive element $Y_{n-s+1}$.
\end{prop}

\begin{remark}\label{total_degrees_remark}
Since $\deg_{\scriptscriptstyle (Z_1,\dots,Z_{n-s+1})}
P_{\scriptscriptstyle V}=\deg_{\scriptscriptstyle
Z_{n-s+1}}P_{\scriptscriptstyle V}=\delta$ (see Section \ref{subsec:
Chow form equidimensional variety}), we have
$\deg_{(Y_1,\dots,Y_{n-s},T)} Q=\delta$ and
$\deg_{(Y_1,\dots,Y_{n-s},T)} W_j\le \delta$ for ${n-s+2}\le j\le
n$.
\end{remark}

Now, let $\mathcal{J}:=\mathcal{I}+
(Y_1-p_1,\dots,Y_{n-s}-p_{n-s})$. Denote as in Lemma \ref{lemma:
lifting_curve_isomorphism} by $\overline{\mathcal{J}}$ the image of
$\mathcal{J}$ in $\K[Y_{n-s+1},\dots,Y_n]$. Substituting
$p_1,\dots,p_{n-s}$ for $Y_1,\dots, Y_{n-s}$ in
\eqref{eq:kronecker_representation_I} we obtain
\begin{align}\label{eq:kronecker_representation_J}
Q(\bfs p,Y_{n-s+1})&\in \overline{\mathcal{J}}, \ \, Q'(\bfs
p,Y_{n-s+1})Y_j-W_j(\bfs p,Y_{n-s+1})\in \overline{\mathcal{J}} \ \,
(n-s+2\leq j \leq n).
\end{align}
The polynomial $Q(\bfs p,T)$ is monic of degree $\delta$ and $\deg
W_j(\bfs p,T)< \delta$ for $n-s+2\leq j \leq n$. The discriminant of
$Q(\bfs p,T)$ is ${\rho_{\scriptscriptstyle V}(\bfs\lambda, \bfs
p)}/{A_{\scriptscriptstyle V}(\bfs\lambda^*)}^{2\delta-1}$, and thus
nonzero due to the choice of $\bfs\lambda$ and $\bfs p$. It follows
that $Q(\bfs p,T)$ is square--free and $Q'(\bfs p,T)$ is invertible
modulo $Q(\bfs p,T)$. This implies that $Q'(\bfs p,Y_{n-s+1})$ is
invertible in $\K[Y_{n-s+1},\dots,Y_n]/\overline{\mathcal{J}}$, and
\eqref{eq:kronecker_representation_J} shows that the homomorphism of
$\K$--algebras
\[
\K[T]/\bigl(Q(\bfs p,T)\bigr) \rightarrow
K[Y_{n-s+1},\dots,Y_n]/\overline{\mathcal{J}},\quad T\!\!\!\! \mod
Q(\bfs p,T)\mapsto Y_{n-s+1}\!\!\!\! \mod \overline{\mathcal{J}},
\]
is surjective. This means that $Y_{n-s+1}$ induces a primitive
element for $\overline{\mathcal{J}}$. Further, since $\K[V_{\bfs
p}]\cong \K[Y_{n-s+1},\dots,Y_n]/\overline{\mathcal{J}}$ is a
$\K$--vector space of dimension equal to $\mathrm{rank}_R\K[V]$, and
$\mathrm{rank}_R\K[V]=\deg Q(\bfs p,T)=\delta$, it follows that the
above homomorphism is an isomorphism. We conclude that $Q(\bfs p,T)$
is the minimal polynomial of $Y_{n-s+1}$ over $\K$ modulo
$\overline{\mathcal{J}}$, and that the following equality of ideals
holds  in $\K[Y_{n-s+1},\dots, Y_n]$:
$$
\overline{\mathcal{J}} = \bigl(Q(\bfs p,Y_{n-s+1}), Q'(\bfs
p,Y_{n-s+1})Y_j-W_j(\bfs p,Y_{n-s+1}):n-s+2\le j\le n\bigr).
$$
%
%
Identifying $\mathcal{J}$ with its image in
$\K[Y_{n-s+1},\dots,Y_n]$, we obtain the following result.
\begin{prop}\label{prop: kronecker_representation_of_J}
The polynomials $Q(\bfs p,T), W_{n-s+2}(\bfs p,T), \dots, W_n(\bfs
p, T)$ form the Kronecker representation  of $\mathcal{J}$ with
primitive element $Y_{n-s+1}$.
\end{prop}

Finally, we discuss a Kronecker representation of
$\mathcal{K}:=\mathcal{I}+(Y_1-p_1,\dots,Y_{n-s-1}-p_{n-s-1})$. Let
$\bfs p^*:=(p_1,\dots,p_{n-s-1})$ and let $\overline{\mathcal{K}}$
be the image of $\mathcal{K}$ in $\K[Y_{n-s},\dots,Y_n]$ as in Lemma
\ref{lemma: lifting_curve_isomorphism}. Then $Y_{n-s},\dots, Y_n$
are in Noether position with respect to $\overline{\mathcal{K}}$ and
$\K[W_{\bfs p^*}]\cong\K[Y_{n-s},\dots,Y_n]/\overline{\mathcal{K}}$
is a free $\K[Y_{n-s}]$--module of rank equal to
$\mathrm{rank}_R\K[V]$. Substituting $p_1,\dots,p_{n-s-1}$ for
$Y_1,\dots, Y_{n-s-1}$ in \eqref{eq:kronecker_representation_I}, we
deduce that
\begin{align}\label{eq:kronecker_representation_of_K}
Q(\bfs p^*, Y_{n-s}, Y_{n-s+1})&\in \overline{\mathcal{K}},\\
\nonumber Q'(\bfs p^*,Y_{n-s, }Y_{n-s+1})Y_j-W_j(\bfs p^*, Y_{n-s,
}Y_{n-s+1}) &\in \overline{\mathcal{K}}  \quad (n-s+2\leq j \leq n).
\end{align}
Observe that $Q(\bfs p^*, Y_{n-s}, T)$ is monic of degree $\delta$
and $\deg W_j(\bfs p^*, Y_{n-s}, T)< \delta$ for $n-s+2\leq j \leq
n$. By the choice of $\bfs\lambda$, the discriminant
${\rho_{\scriptscriptstyle V}(\lambda, \bfs p^*,
Y_{n-s})}/{A_{\scriptscriptstyle V}(\bfs\lambda^*)}^{2\delta-1}$ of
$Q(\bfs p^*, Y_{n-s}, T)$ is a nonzero element of $\K[Y_{n-s}]$.
Therefore, $Q(\bfs p^*, Y_{n-s}, T)$ is square--free, $Q'(\bfs p^*,
Y_{n-s}, T)$ is invertible modulo $Q(\bfs p^*, Y_{n-s}, T)$, and
thus $Q'(\bfs p^*, Y_{n-s}, Y_{n-s+1})$ is invertible in
$\K(Y_{n-s})[Y_{n-s+1},\dots,Y_n]/\overline{\mathcal{K}}^e$, where
$\overline{\mathcal{K}}^e$ is the extension of
$\overline{\mathcal{K}}$ to the ring
$\K(Y_{n-s})[Y_{n-s+1},\dots,Y_n]$. By
\eqref{eq:kronecker_representation_of_K} the homomorphism of
$\K(Y_{n-s})$--algebras
 \[
 \K(Y_{n-s})[T]/\bigl(Q(\bfs p^*, Y_{n-s}, T)\bigr)
 \rightarrow \K(Y_{n-s})[Y_{n-s+1},\dots,Y_n]/\overline{\mathcal{K}}^e
 \]
which maps $T \mod Q(\bfs p^*, Y_{n-s}, T)$ to $Y_{n-s+1} \mod
\overline{\mathcal{K}}^e$ is surjective. In particular,
$Y_{n-s+\!1}$ induces a primitive element for
$\overline{\mathcal{K}}$. Since
$\K(Y_{n-s})[Y_{n-s+\!1},\dots,\!Y_n]/\overline{\!\mathcal{K}}^e$ is
a $\K(Y_{n-s})$--vector space of dimension equal to
$\mathrm{rank}_R\K[V]=\deg Q(\bfs p^*, Y_{n-s}, T)=\delta$, this
homomorphism is an isomorphism. We conclude that $Q(\bfs p^*,
Y_{n-s}, T)$ is the minimal polynomial of $Y_{n-s+1}$ modulo
$\overline{\mathcal{K}}^e$, and the following equality of ideals
holds in $\K(Y_{n-s})[Y_{n-s+1},\dots,Y_n]$:
\begin{align*}
\overline{\mathcal{K}}^e = \bigl(Q(\bfs p^*, Y_{n-s}, Y_{n-s+1}),
Q'(\bfs p^*, Y_{n-s}, Y_{n-s+1})Y_{n-s+2}-W_{n-s+2}(\bfs p^*, Y_{n-s}, Y_{n-s+1}), \\
\dots,Q'(\bfs p^*, Y_{n-s}, Y_{n-s+1})Y_n-W_n(\bfs p^*,Y_{n-s},
Y_{n-s+1})\bigr).
\end{align*}
Identifying $\mathcal{K}$ with its image in $\K[Y_{n-s},\dots,Y_n]$,
we obtain the following result.
\begin{prop} \label{prop: kronecker_representation_of_K}
$Q(\bfs p^*, Y_{n-s},T),W_{n-s+2}(\bfs p^*,Y_{n-s},T),\dots,
W_n(\bfs p^*,Y_{n-s},T)$ form the Kronecker representation of
$\mathcal{K}$  with primitive element $Y_{n-s+1}$.
\end{prop}
%
%
\section{On the conditions for a good modular reduction}
\label{section: modular_simultaneous_noether_normalization}
From now on we consider polynomials $F_1 \klk F_r \in \Z[\bfs X]$ of
degree at most $d$ that form a reduced regular sequence, and denote
$\mathcal{V}_s:=V(F_1\klk F_s)$ and $\delta_s:=\deg \mathcal{V}_s$
for $1\le s\le r$. As explained in the introduction, our aim is to
describe an algorithm for solving the system $F_1=0\klk F_r=0$ and
analyze its bit complexity. This algorithm outputs a Kronecker
representation of a lifting fiber of $\mathcal{V}_r$ and relies on
modular methods. For this reason, a crucial point is the choice of a
``lucky'' prime number, namely one which provides a good modular
reduction, of ``low'' bit length. In this section we exhibit a
nonzero integer multiple $\mathfrak{N}$ of all the unlucky prime
numbers. More precisely, we show that, for a suitable choice of
$\bfs \lambda \in \Z^{n^2}$ and $\bfs p \in \Z^{n-1}$, there is a
nonzero integer $\mathfrak{N}$ with the following property: if $p$
is a prime number not dividing $\mathfrak{N}$, then all conditions
in Theorem \ref{th: intro: good modular reduction} modulo $p$ are
satisfied. Further, our description of $\mathfrak{N}$ is explicit
enough as to allow us to estimate its bit length (Theorem \ref{th:
app: mathfrak_N_height}). By this estimate and well--known methods
for finding small primes not dividing a given integer we shall be
able to compute in Section \ref{section:
computation_of_a_kronecker_representation} a lucky prime of low bit
length with hight probability of success.

The determination of the integer $\mathfrak{N}$ proceeds in several
stages. In Section \ref{subsec: dimension_and_noether_normalization}
we deal with conditions $(1)$--$(2)$ of Theorem \ref{th: intro: good
modular reduction}, and the corresponding results are summarized in
Theorem \ref{th: conditions_mod_p_summary}. Then in Section
\ref{subsec: lifting_fibers_not_meeting_a_discriminant} we discuss
the fulfillment of the more involved condition $(3)$ of Theorem
\ref{th: intro: good modular reduction}.

In the sequel, if $p$ is a prime number and $G$ any polynomial with
integer coefficients, we denote by $G_p$ its reduction modulo $p$.
Further, if $G_1\klk G_t\in\Z[\bfs X]$ define a variety
$\mathcal{W}:=V(G_1\klk G_t)\subset\A^m:=\A^m(\overline{Q})$, we
denote by $\mathcal{W}_{p}:=V(G_{1,p} \klk G_{t,p})\subseteq
\mathbb{A}_{\scriptscriptstyle \cfp}^m:=\A^m(\cfp)$ the
corresponding reduction modulo $p$.
%
%
\subsection{First conditions for a good modular reduction}
\label{subsec: dimension_and_noether_normalization} Fix $s$ with
$1\leq s \leq r$ and $\bfs \lambda \in \Z^{(n-s+1)n}$ such that the
hypotheses of Proposition \ref{prop:
finite_fiber_and_finite_morphism} are satisfied. In this section we
establish a condition on a prime number $p$ which implies that the
variety $\mathcal{V}_{s,p}$ is equidimensional and reduced of
dimension $n-s$ and degree $\delta_s$, and the linear forms
$(Y_{1,p} \klk Y_{n-s,p}):= \bfs \lambda_p \bfs X$ are the free
variables of a Noether normalization of $\mathcal{V}_{s,p}$.

Throughout this section and the next one,
$\bfs\Lambda:=(\Lambda_{ij})_{1\leq i \leq n-s+1, 1\leq j\leq n}$
and $\bfs Z:=(Z_1,\dots,Z_{n-s+1})$ denote a matrix and a vector of
indeterminates over $\Q[\mathcal{V}_s]$. We set $\bfs
\Lambda_i:=(\Lambda_{i1},\dots,\Lambda_{in})$ and $\bfs\Lambda_i
\cdot \bfs X:=\sum_{j=1}^n\Lambda_{ij}X_j$ for $1 \leq i \leq
n-s+1$. Further, we denote $\bfs\Lambda \bfs X:=(\bfs\Lambda_1 \cdot
\bfs X, \dots, \bfs\Lambda_{n-s+1} \cdot \bfs X)$,
$\bfs\Lambda^*:=(\Lambda_{ij})_{1\leq i \leq n-s, 1\leq j\leq n}$
and $\bfs\Lambda^* \bfs X:=(\bfs\Lambda_1 \cdot \bfs X, \dots,
\bfs\Lambda_{n-s} \cdot \bfs X)$. Finally, given $\bfs \lambda
:=(\lambda_{ij})_{1\leq i \leq n-s+1, 1\leq j\leq n}\in
\Z^{(n-s+1)n}$, we adopt the notations $\bfs \lambda_i\cdot  \bfs X$
$(1\leq i \leq n-s+1)$, $\bfs\lambda \bfs X$, $\bfs \lambda^*$ and
$\bfs\lambda^* \bfs X$ accordingly. Denote by $P_s\in
\Q[\bfs\Lambda,\bfs Z]$ a Chow form of $\mathcal{V}_s$. Since $P_s$
is uniquely determined up to nonzero multiples in $\Q$, we may
assume that $P_s$ is a primitive polynomial of $\Z[\bfs\Lambda, \bfs
Z]$. Let as before $A_s\in \Z[\bfs\Lambda_1 \klk \bfs\Lambda_{n-s}]$
be the coefficient of the monomial $Z_{n-s+1}^{\delta_s}$ in $P_s$
and $\rho_s\in \Z[\bfs\Lambda, Z_1 \klk Z_{n-s}]$ the discriminant
of $P_s$ with respect to $Z_{n-s+1}$, that is,
  $$
  \rho_s:=\mathrm{Res}_{Z_{n-s+1}}\left(P_s, \frac{\partial P_s}{\partial Z_{n-s+1}}\right).
 $$
According to Lemma \ref{lemma: nonzero_discriminant}, the
polynomials $\partial P_s/\partial Z_{n-s+1}$ and $\rho_s$ are both
nonzero.

As a first step, we give a condition for consistency of the system
$F_{1,p}=0 \klk F_{s,p}=0$.
\begin{lema}\label{lemma:_condition_modular_ reduction_consistent}
Let $p$ be a prime number such that
$A_{s,p}(\bfs\lambda_p^*)\rho_{s,p}(\!\bfs\lambda_p,\!Z_1 \klk
\!Z_{n-s})$ is nonzero. Let $Y_{i,p}:=\bfs\lambda_{i,p}\cdot \bfs X$
for $1\leq i \leq n-s$. If $\pi_{s,p}: \mathcal{V}_{s,p}\rightarrow
\mathbb{A}_{\scriptscriptstyle \cfp}^{n-s}$ is the mapping defined
by $Y_{1,p} \klk Y_{n-s,p}$, then any $\bfs q \in
\mathbb{A}_{\scriptscriptstyle \cfp}^{n-s}$ with
$\rho_{s,p}(\bfs\lambda_p, \bfs q)\neq 0$ satisfies $\#
\pi_{s,p}^{-1}(\bfs q)\geq \delta_s$.
\end{lema}
\begin{proof}
Note that $P_{s,p}(\bfs\lambda_p,\bfs q, Z_{n-s+1})$ has degree
$\delta_s$, because $A_{s,p}(\bfs\lambda_p^*)\neq 0$. It follows
that
\[
\rho_{s,p}(\bfs\lambda_p,\bfs
q)=\mathrm{Res}_{Z_{n-s+1}}\biggl(P_{s,p}(\bfs\lambda_p, \bfs q,
Z_{n-s+1}), \frac{\partial P_{s,p}}{\partial
Z_{n-s+1}}(\bfs\lambda_p,\bfs q, Z_{n-s+1})\biggr),
\]
and thus the polynomial $P_{s,p}(\bfs\lambda_p,\bfs q,Z_{n-s+1})$ is
separable. Let $z_1,\dots,z_{\delta_s}\in \cfp$ be the roots of
$P_{s,p}(\bfs\lambda_p,\bfs q,Z_{n-s+1})$ and $\bfs y^k:=(\bfs q,
z_k)$ for $1\leq k \leq \delta_s$. As $\partial P_{s,p}/\partial
Z_{n-s+1}(\bfs\lambda_p,\bfs y^k)\neq 0$ for $1\leq k \leq
\delta_s$, the point
 \[
\bfs x^k:=\biggl(-\frac{\partial P_{s,p}/\partial
\Lambda_{n-s+1,1}(\bfs\lambda_p,\bfs y^k)}{\partial P_{s,p}/\partial
Z_{n-s+1}(\bfs\lambda_p,\bfs y^k)},\dots, -\frac{\partial
P_{s,p}/\partial \Lambda_{n-s+1,n}(\bfs\lambda_p,\bfs y^k)}{\partial
P_{s,p}/\partial Z_{n-s+1}(\bfs\lambda_p,\bfs y^k)} \biggr) \in
\A_{\scriptscriptstyle \cfp}^n
\]
is well defined for $1\leq k \leq \delta_s$.

We claim that $\bfs x^{1} \klk \bfs x^{\delta_s}$ are pairwise
distinct and $\{\bfs x^{1} \klk \bfs
x^{\delta_s}\}\subseteq\pi_{s,p}^{-1}(\bfs q)$. Indeed, let
$F_{\Lambda,j}\in \Z[\bfs\Lambda, \bfs Z]$ and $\eta_j\in
\mathbb{N}$ be such that
\begin{equation}\label{eq:lema_modular_consistency}
F_j\left(  -\frac{\partial P_s /\partial \Lambda_{n-s+1,1}}{\partial
P_s/\partial Z_{n-s+1}},\dots, -\frac{\partial P_s /\partial
\Lambda_{n-s+1,n}}{\partial P_s/\partial Z_{n-s+1}}
\right)=\frac{F_{\Lambda,j}}{(\partial P_s/\partial
Z_{n-s+1})^{\eta_j}}
\end{equation}
for $1\leq j \leq s$. Also let
$$
H_i:=\frac{\partial P_s}{\partial Z_{n-s+1}}\,Z_i +
\sum_{j=1}^{n}\Lambda_{ij}\frac{\partial P_s }{\partial
\Lambda_{n-s+1,j}}.
$$
for $1\leq i \leq n-s+1$. Lemma \ref{lemma:
cardinality_of_the_fiber} shows that $F_{\Lambda,j}$ $(1\le j\le s)$
and $H_i$ $(1\le i\le n-s+1)$ are multiples of $P_s$ in
$\Q[\bfs\Lambda, \bfs Z]$. Further, since $P_s$ is a primitive
polynomial, we conclude that they are multiples of $P_s$ in $\Z[\bfs
\Lambda,\bfs Z]$, and thus that $F_{\Lambda,j,p}$  $(1\le j\le s)$
and $H_{i,p}$  $(1\le i\le n-s+1)$ are multiples of $P_{s,p}$. As
$P_{s,p}(\bfs\lambda_p,\bfs y^k)=0$  by construction, we see that
$F_{\Lambda,j,p}(\bfs\lambda_p,\bfs y^k)=0$ and
$H_{i,p}(\bfs\lambda_p,\bfs y^k)=0$ for $1\leq k \leq \delta_s$, and
reducing \eqref{eq:lema_modular_consistency} modulo $p$ we deduce
that $F_{j,p}(\bfs x^k)=0$ for $1\leq k \leq \delta_s$. Then
following the proof of Proposition \ref{prop:
cardinality_of_the_fiber} {\em mutatis mutandis} we conclude that
$\bfs x^1,\dots,\bfs x^{\delta_s}$ are pairwise distinct points of
$\pi_{s,p}^{-1}(\bfs q)$.
\end{proof}
By definition, $P_s(\bfs\Lambda, \bfs\Lambda \bfs X) \in
\Z[\bfs\Lambda, \bfs X]$ vanishes on the set of common zeros
$\mathbb{A}^{(n-s+1)n}\times \mathcal{V}_s$ of $F_1,\dots,F_s$ in
$\mathbb{A}^{(n-s+1)n}\times \mathbb{A}^n$. By the Nullstellensatz,
there exist $\alpha_s\in \Z\setminus\{0\}$ and $\mu_s \in
\mathbb{N}$ such that
\begin{equation}\label{alpha_s_definition}
\alpha_s P_s(\bfs\Lambda,\bfs\Lambda \bfs X)^{\mu_s}\in
(F_1,\dots,F_s)\Z[\bfs\Lambda,\bfs X].
\end{equation}
Our next result provides a condition which implies that the modular
reduction preserves dimension and a Noether normalization.
\begin{prop}\label{prop: finite_morphism}
Let $p$ be a prime number such that
 $\alpha_{s,p}A_{s,p}(\bfs \lambda_p^*)\rho_{s,p}(\bfs\lambda_p, Z_1,\dots,Z_{n-s})$ is nonzero.  Let
$Y_i:=\bfs\lambda_i\cdot \bfs X$ for $1\leq i \leq n-s$. Then:
  \begin{enumerate}
    \item $F_{1,p}, \dots, F_{s,p}$ generate an unmixed ideal in
    $\cfp[\bfs X]$ of dimension $n-s$;
    \item the mapping $\pi_{s,p}:  \mathcal{V}_{s,p}\rightarrow
    \A_{\scriptscriptstyle \cfp}^{n-s}$
    defined by $Y_{1,p},\dots,Y_{n-s,p}$ is a finite morphism.
  \end{enumerate}
\end{prop}
\begin{proof}
Recall that $A_s$ is homogeneous of degree $\delta_s$ in the
$(n-s)\times (n-s)$--minors of $\bfs\Lambda^*$. Since $p \nmid
A_s(\bfs\lambda^*)$, at least one of the $(n-s)\times (n-s)$--minors
of $\bfs\lambda^*$ is nonzero modulo $p$. We deduce that the linear
forms $Y_{1,p}, \dots, Y_{n-s,p}$ are linearly independent, and
there exist linear forms ${Y}_{n-s+1}, \dots, {Y}_n \in \Z[\bfs X]$
such that $Y_{1,p}, \dots, {Y}_{n,p}$ are linearly independent in
$\mathbb{F}_p[\bfs X]$. Let $\bfs w_k\in \Z^n$ be such that
${Y}_{n-s+k}=\bfs w_k\cdot \bfs X$ for $1\leq k \leq s$ and
$$Q_k:=P_s(\bfs\lambda^*,\bfs w_k,Y_1\klk Y_{n-s},Y_{n-s+k})\in
\Z[Z_1,\dots,Z_{n-s+1}].$$
From \eqref{alpha_s_definition} we see that
$\alpha_s Q_k(Y_1,\dots,Y_{n-s},{Y}_{n-s+k})^{\mu_s}\in
(F_1,\dots,F_s)\Z[\bfs X]$
for $1\leq k \leq s$, and reducing modulo $p$ we obtain
\begin{equation}\label{eq:finite_morphism_mod_p_prop_1}
\alpha_{s,p}Q_{k,p}(Y_{1,p},\dots,Y_{n-s,p}, {Y}_{n-s+k,p})^{\mu_s}
\in (F_{1,p}\dots,F_{s,p}) \fp[\bfs X]
\end{equation}
for $1\leq k \leq s$. Observe that $\deg_{Z_{n-s+1}}Q_k=\delta_s$
and $A_s(\bfs\lambda^*)$ is the coefficient of
$Z_{n-s+1}^{\delta_s}$ in $Q_k$. Since  $p \nmid
\alpha_sA_s(\bfs\lambda^*)$, identity
\eqref{eq:finite_morphism_mod_p_prop_1} may be interpreted as an
integral dependence relation for ${Y}_{n-s +k,p}$ over
$\overline{\mathbb{F}}_p[Y_{1,p},\dots,Y_{n-s,p}]$ modulo
$(F_{1,p}\dots,F_{s,p})$. Further, since
$\overline{\mathbb{F}}_p[Y_{1,p}, \dots,
{Y}_{n,p}]=\overline{\mathbb{F}}_p[\bfs X]$, we conclude that
$\overline{\mathbb{F}}_p[Y_{1,p},\dots,Y_{n-s,p}]\rightarrow
\overline{\mathbb{F}}_p[\mathcal{V}_{s,p}]$ is an integral ring
extension. In particular, we have $\dim \mathcal{V}_{s,p}\leq n-s$.
Moreover, since $A_{s,p}(\bfs\lambda_p^*)\rho_{s,p}(\bfs\lambda_p,
Z_1 \klk Z_{n-s})\not=0$, by Lemma \ref{lemma:_condition_modular_
reduction_consistent} the variety
$\mathcal{V}_{s,p}=V(F_{1,p},\dots,F_{s,p})$ is nonempty. Therefore,
$(F_{1,p}, \dots, F_{s,p})$ is a proper ideal of
$\overline{\mathbb{F}}_p[\bfs X]$ of dimension at most $n-s$, while
the Principal Ideal theorem (see, e.g., \cite[Theorem
10.2]{Eisenbud95}) implies $\dim(F_{1,p}, \dots, F_{s,p})\geq n-s$.
We conclude that $\dim (F_{1,p}, \dots, F_{s,p})=n-s$, and the
unmixedness theorem proves that $(F_{1,p}, \dots, F_{s,p})$ is
unmixed. This shows the first assertion. Since the ring extension
$\overline{\mathbb{F}}_p[Y_{1,p},\dots,Y_{n-s,p}]\rightarrow
\overline{\mathbb{F}}_p[\mathcal{V}_{s,p}]$ is integral and $\dim
\mathcal{V}_{s,p}=n-s$, it follows that $\pi_{s,p}:
\mathcal{V}_{s,p}\rightarrow \mathbb{A}_{\scriptscriptstyle
\overline{\mathbb{F}}_p}^{n-s}$ is a finite morphism, which finishes
the proof.
\end{proof}

Next we show that the hypotheses of Proposition \ref{prop:
finite_morphism} also guarantee that the degree is preserved under
modular reduction, and the modular Chow form is obtained reducing
modulo $p$ that of $\mathcal{V}_s$.
\begin{coro}\label{coro: equality_of_the_degree}
With notations and hypotheses as in Proposition \ref{prop:
finite_morphism}, $\deg \mathcal{V}_{s,p}=\delta_s$ and $P_{s,p}$
is a Chow form of $\mathcal{V}_{s,p}$.
\end{coro}
\begin{proof}
Since $p \nmid \alpha_s$, from (\ref{alpha_s_definition}) we deduce
that
$
P_{s,p}(\bfs\Lambda,\bfs\Lambda \bfs X)^{\mu_s}\in
(F_{1,p},\dots,F_{s,p})\fp[\bfs\Lambda, \bfs X].
$
It follows that $P_{s,p}(\bfs\Lambda, \bfs\Lambda \bfs X)$ vanishes
on $\mathbb{A}_{\scriptscriptstyle
\overline{\mathbb{F}}_p}^{(n-s+1)n}\times \mathcal{V}_{s,p}$. As a
consequence, if $Q_s\in \fp[\bfs\Lambda, \bfs Z]$ is a Chow form of
$\mathcal{V}_{s,p}$, then $Q_s$ divides $P_{s,p}$ in
$\fp[\bfs\Lambda, \bfs Z]$. Since  $P_{s,p}$ is nonzero, because
$P_s$ is primitive,  we conclude that
\[
\deg \mathcal{V}_{s,p}=\deg_{Z_{n-s+1}}Q_s\leq
\deg_{Z_{n-s+1}}P_{s,p}\leq \delta_s.
\]
On the other hand, Proposition \ref{prop: finite_morphism} shows
that $\pi_{s,p}$ is a finite morphism, and the (finite) fiber
$\pi_{s,p}^{-1}(\bfs p_p)$ satisfies $\#\pi_{s,p}^{-1}(\bfs p_p)\geq
\delta_s$ by Lemma \ref{lemma:_condition_modular_
reduction_consistent}. The B\'ezout inequality
\eqref{bezout_inequality} implies
\[
\#\pi_{s,p}^{-1}(\bfs p_p)=\deg\bigl(\mathcal{V}_{s,p}\cap
\{Y_{1,p}-p_{1,p}=0,\dots, Y_{n-s,p}-p_{n-s}=0\}\bigr)\leq \deg
\mathcal{V}_{s,p}.
\]
This proves that $\deg \mathcal{V}_{s,p}= \delta_s$. Since $Q_s$ is
homogeneous of degree $\delta_s$ and $P_{s,p}$ has degree at most
$\delta_s$ in each set of variables $(Z_i,
\Lambda_{i1},\dots,\Lambda_{in})$ for $1\leq i \leq n-s+1$, we
deduce that $P_{s,p}=\epsilon Q_s$ for some $\epsilon \in
\fp\setminus \{0\}$, showing thus that $P_{s,p}$ is a Chow form of
$\mathcal{V}_{s,p}$.
\end{proof}

Finally, we obtain a condition which implies that the modular
reduction preserves generic smoothness. Let $\bfs p:=(p_1, \dots,
p_{n-s})\in \Z^{n-s}$ be such that
$A_s(\bfs\lambda^*)\rho_s(\bfs\lambda, \bfs p)\neq 0$. From Theorem
\ref{th:_lifting_point_rank_and_finite_morphism} it follows that
$\bfs p$ is a lifting point of the mapping $\pi_s: \mathcal{V}_s
\rightarrow \mathbb{A}^{n-s}$ defined by $Y_1, \dots, Y_{n-s}$. Then
$F_1, \dots, F_s, Y_1-p_1, \dots, Y_{n-s}-p_{n-s}$ and the Jacobian
determinant $J_s$ of $F_1, \dots, F_s, Y_1-p_1, \dots,
Y_{n-s}-p_{n-s}$ with respect to $X_1,\dots, X_n$ do not have common
zeros in $\mathbb{A}^{n}$. By the Nullstellensatz, there exist
$\gamma_s\in \Z\setminus \{0\}$ and $G_1, \dots, G_{n+1}\in \Z[\bfs
X]$ such that
\begin{equation}\label{gamma_s_definition}
\gamma_s=G_1F_1+ \cdots + G_sF_s+ G_{s+1}(Y_1-p_1) + \cdots +
G_n(Y_{n-s}-p_{n-s}) + G_{n+1}J_s.
\end{equation}
The nonvanishing of $\gamma_s$ modulo $p$ provides the additional
condition we are looking for.
\begin{lema} \label{lema: radicality_modulo_p}
With the previous hypotheses and notations, let $p$ be a prime
number such that $p\nmid \alpha_s\gamma_s A_s(\bfs
\lambda^*)\rho_s(\bfs \lambda, \bfs p)$. Then $F_{1,p}, \dots,
F_{s,p}$ generate a radical ideal in $\cfp[\bfs X]$.
\end{lema}
\begin{proof}
Since by hypothesis $\alpha_{s,p}A_{s,p}(\bfs
\lambda_p^*)\rho_{s,p}(\bfs \lambda_p, Z_1 \klk Z_{n-s})$ is
nonzero, from Proposition \ref{prop: finite_morphism} it follows
that $\mathcal{V}_{s,p}$ is equidimensional of dimension $n-s$ and
the mapping $\pi_{s,p}: \mathcal{V}_{s,p}\rightarrow
\A_{\scriptscriptstyle \cfp}^{n-s}$ defined by
$Y_{1,p},\dots,Y_{n-s,p}$ is a finite morphism. On the other hand,
reducing \eqref{gamma_s_definition} modulo $p$ we see that
\[
\gamma_{s,p}=G_{1,p}F_{1,p}+ \cdots + G_{s,p}F_{s,p}+
G_{s+1,p}(Y_{1,p}-p_{1,p}) + \cdots + G_{n,p}(Y_{n-s,p}-p_{n-s,p}) +
G_{n+1,p}J_{s,p}
\]
holds in $\fp[\bfs X]$.  We deduce that $J_{s,p}(\bfs x)\neq 0$ for
any $\bfs x\in \pi_{s,p}^{-1}(\bfs p)$. Let $\mathcal{C}_1, \dots,
\mathcal{C}_h$ be the irreducible components of $\mathcal{V}_{s,p}$
and let $\pi_{\mathcal{C}_i}$ denote the restriction of $\pi_{s,p}$
to $\mathcal{C}_i$ for $1 \leq i \leq h$. Since $\mathcal{V}_{s,p}$
is equidimensional,  $\pi_{\mathcal{C}_i}$ is a finite morphism. In
particular, $\pi_{\mathcal{C}_i}$ is surjective and
$\mathcal{C}_i\cap \pi_{s,p}^{-1}(\bfs p_p)\neq \emptyset$ for $1
\leq i \leq h$. It follows that $J_{s,p}$ does not vanish
identically on $\mathcal{C}_i$, which implies that there exists an
$(s\times s)$--minor $M_i\in \fp[\bfs X]$ of the Jacobian matrix
$(\partial F_{i,p}/\partial X_j)_{1 \leq i \leq s, 1 \leq j \leq n}$
not vanishing identically on $\mathcal{C}_i$ for  $1 \leq i \leq h$.
Let $\mathcal{J} \subseteq \overline{\mathbb{F}}_p[\bfs X]$ be  the
ideal generated by $F_{1,p}, \dots, F_{s,p}$ and the $(s \times
s)$--minors of the Jacobian matrix $(\partial F_{i,p}/\partial
X_j)_{1 \leq i \leq s, 1 \leq j \leq n}$. If $\mathcal{P}_i\subseteq
\overline{\mathbb{F}}_p[\bfs X]$ is the vanishing ideal of
$\mathcal{C}_i$ for $1\leq i \leq h$, then $\mathcal{P}_1, \dots,
\mathcal{P}_h$ are the minimal prime ideals of $(F_{1,p}, \dots,
F_{s,p})$. Since $M_i\not \in \mathcal{P}_i$, we have
$\mathcal{J}\nsubseteq \mathcal{P}_i$ for $1\leq i \leq h$, and
Lemma \ref{lemma: radicality_criterion} proves that the ideal
$(F_{1,p}, \dots, F_{s,p})$ is radical.
\end{proof}
We summarize all the previous results in the following theorem.
\begin{teo}\label{th: conditions_mod_p_summary} Let
$\bfs \lambda \in \Z^{(n-s+1)n}$ and $\bfs p \in \Z^{n-s}$ be such
that $A_s(\bfs \lambda^*)\rho_s(\bfs \lambda, \bfs p)\neq 0$ and $p$
a prime number such that $p\nmid \alpha_s\gamma_s A_s(\bfs
\lambda^*)\rho_s(\bfs \lambda, \bfs p)$, where $\alpha_s$ and
$\gamma_s$ are the integers of \eqref{alpha_s_definition} and
\eqref{gamma_s_definition} respectively. Let
$Y_{i,p}:=\bfs\lambda_{i,p}\cdot \bfs X$ for $1\le i\le n-s+1$ and
$R_{s,p}:=\cfp[Y_{1,p},\dots,Y_{n-s,p}]$. Then the following
conditions hold:
 \begin{itemize}
   \item $F_{1,p}, \dots, F_{s,p}$ generate a radical ideal in $\cfp[\bfs X]$
   and define an equidimensional variety $\mathcal{V}_{s,p}\subset \A_{\scriptscriptstyle \cfp}^{n-s}$
   of dimension $n-s$ and degree $\delta_s$;
   \item the mapping $\pi_{s,p}: \mathcal{V}_{s,p}\rightarrow \A_{\scriptscriptstyle \cfp}^{n-s}$
   defined by $Y_{1,p},\dots,Y_{n-s,p}$ is a finite morphism and $Y_{n-s+1,p}$ induces
   a primitive element of the ring extension $R_{s,p}\hookrightarrow
      \cfp[\mathcal{V}_{s,p}]$;
   \item $\mathrm{rank}_{R_{s,p}}\cfp[\mathcal{V}_{s,p}]=\delta_s$;
   \item any $\bfs q\in \mathbb{A}_{\scriptscriptstyle \overline{\mathbb{F}}_p}^{n-s}$
         with $\rho_{s,p}(\bfs\lambda_p,\bfs q)\neq 0$  is a lifting point of $\pi_{s,p}$ and $Y_{n-s+1,p}$ induces a primitive element of $\pi_{s,p}^{-1}(\bfs q)$.
 \end{itemize}
\end{teo}
\begin{proof}
The first assertion follows by Proposition \ref{prop:
finite_morphism}, Corollary \ref{coro: equality_of_the_degree} and
Lemma \ref{lema: radicality_modulo_p}. Since $P_{s,p}$ is a Chow
form of $\mathcal{V}_{s,p}$ by Corollary \ref{coro:
equality_of_the_degree}, the last three assertions are a consequence
of Theorem \ref{th:_lifting_point_rank_and_finite_morphism} applied
to $\K=\overline{\mathbb{F}}_p$.
\end{proof}
%
%
\subsection{Lifting fibers not meeting a discriminant}
\label{subsec: lifting_fibers_not_meeting_a_discriminant}
Throughout this section we assume that $s\le r-1$. Our main
algorithm is recursive, and in its $s$th step computes a geometric
solution of the fiber $\pi_{s+1}^{-1}(\bfs p^*)$ from one of the
lifting curve $W_{\bfs p^*}$. As the geometric solution of $W_{\bfs
p^*}$ constitutes a ``good'' representation of $W_{\bfs p^*}$
outside the discriminant locus $\{\rho_s(\bfs\lambda,Y_1\klk
Y_{n-s})=0\}$, it is critical that $\pi_{s+1}^{-1}(\bfs p^*)$ does
not intersect this hypersurface. In this section we show that for a
generic choice of the coordinates of $\bfs\lambda$ and $\bfs p$ this
condition is satisfied and discuss when this is preserved under
modular reduction.

For this purpose, we use the following terminology: for two
subvarieties $\mathcal{V}$ and $\mathcal{W}$ of $\mathbb{A}^n$, we
say that $\mathcal{W}$ cuts $\mathcal{V}$ \emph{properly} if
$\mathcal{W}$ does not contain any irreducible
$\overline{\Q}$--component of $\mathcal{V}$. We have the following
result.
\begin{lema}\label{lemma: pure_dimesion_r_subvariety}
There exists a polynomial ${\sf R}_s\in
\overline{\mathbb{Q}}[\bfs\Lambda]\setminus \{0\}$ of degree at most
$2(n-s+2)\delta_s^2\delta_{s+1}$ with the following property: for
every $\bfs\lambda \in \mathbb{A}^{(n-s+1)n}$ with ${\sf
R}_s(\bfs\lambda)\neq 0$, the hypersurface
$\{\rho_s(\bfs\lambda,\bfs\lambda^* \bfs X)=0\}\subset\mathbb{A}^n$
cuts $\mathcal{V}_{s+1}$ properly.
\end{lema}
\begin{proof}
Let $\mathcal{C}_1, \dots, \mathcal{C}_h$ be the irreducible
$\overline{\mathbb{Q}}$--components of $\mathcal{V}_{s+1}$, and let
$\bfs z_i\in \mathcal{C}_i$ be a nonsingular point of
$\mathcal{V}_{s+1}$ for $1\leq i \leq h$. Define
\[
{\sf R}_s:=\prod_{i=1}^h\rho_s(\bfs\Lambda, \bfs\Lambda^* \bfs z_i).
\]
We claim that ${\sf R}_s$ satisfies the conditions of the lemma.
Indeed, fix $1 \leq i \leq h$. Since $\bfs z_i$ is a nonsingular
point of $\mathcal{V}_{s+1}$ and
$\mathcal{I}(\mathcal{V}_{s+1})=\mathcal{I}(\mathcal{V}_s)+(F_{s+1})$,
then $\bfs z_i$ is also a nonsingular point of $\mathcal{V}_s$.
Hence, for a generic choice of $\bfs\lambda\in
\mathbb{A}^{(n-s+1)n}$, denoting by $\pi_s: \mathcal{V}_s\rightarrow
\mathbb{A}^{n-s}$ the mapping $\pi_s(\bfs x):=\bfs\lambda^* \bfs x$
and $\bfs p:=\pi_s(\bfs z_i)$, the following conditions are
satisfied:
\begin{itemize}
  \item $\#\pi_s^{-1}(\bfs p)=\delta_s$;
  \item the linear form $\bfs\lambda_{n-s+1} \cdot \bfs X$ separates
   the points of $\pi_{s}^{-1}(\bfs p)$;
  \item the discriminant of the polynomial $P_s(\bfs\lambda,
  \bfs p, Z_{n-s+1})$ is $\rho_s(\bfs\lambda, \bfs p)$.
\end{itemize}
Indeed, since $\bfs z_i$ is a nonsingular point of $\mathcal{V}_s$,
then $\mathcal{V}_s$ has multiplicity $1$ at $\bfs z_i$ (see, e.g.,
\cite[\S 5A, Corollary 5.15]{Mumford95}). This means that a generic
linear space of dimension $s$ passing through $\bfs z_i$ meets
$\mathcal{V}_s$ in exactly $\delta_s-1$ points different from $\bfs
z_i$, which shows the first condition. The remaining conditions are
clearly satisfied.

Let $\bfs x^1,\dots,\bfs x^{\delta_s}$ be the $\delta_s$ points of
$\pi_{s}^{-1}(\bfs p)$. Since $\bfs\lambda_{n-s+1} \cdot \bfs X$
separates these points, the polynomial $P_s(\bfs\lambda, \bfs p,
Z_{n-s+1})$ has $\delta_s$ different roots, namely
$\bfs\lambda_{n-s+1} \cdot\bfs x^i$ for $1\leq i \leq \delta_s$. We
conclude that $\rho_s(\bfs\lambda, \bfs p)\neq 0$. It follows that
$\rho_s(\bfs\Lambda,\bfs\Lambda^*\bfs z_i)$ is a nonzero polynomial
in $\overline{\mathbb{Q}}[\bfs\Lambda]$ for $1\leq i \leq h$ and
therefore ${\sf R}_s\in \overline{\mathbb{Q}}[\bfs\Lambda]\setminus
\{0\}$. Since $\deg\rho_s(\bfs\Lambda, \bfs\Lambda^* \bfs z_i)\leq
(n-s+2)(2\delta_s-1)\delta_s$ and $h\leq \delta_{s+1}$, the estimate
for the degree ${\sf R}_s$ follows. Finally, let $\bfs\lambda \in
\mathbb{A}^{(n-s+1)n}$ be such that ${\sf R}_s(\bfs\lambda)\neq 0$.
Then $\rho_s(\bfs\lambda,\bfs\lambda^*\bfs z_i)\neq 0$ for $1 \leq i
\leq h$, which shows that $\mathcal{C}_i$ is not contained in the
hypersurface $\{\rho_s(\bfs\lambda,\bfs\lambda^* \bfs X)=0\}$ of
$\mathbb{A}^n$ for $1\leq i \leq h$.
\end{proof}

Let $\bfs\lambda\in \Z^{(n-s+1)n}\setminus \{0\}$ be such that ${\sf
R}_s(\bfs\lambda)\neq 0$ and let $\mathcal{W}_{\bfs\lambda^s}\subset
\mathbb{A}^n$ be the variety
\begin{equation}\label{w_s_definition_equation}
\mathcal{W}_{\bfs\lambda^s}:=\mathcal{V}_{s+1}\cap
\{\rho_s(\bfs\lambda,\bfs\lambda^* \bfs X)=0\}.
\end{equation}
By Lemma \ref{lemma: pure_dimesion_r_subvariety},
$\mathcal{W}_{\bfs\lambda^s}$ is either empty or equidimensional of
dimension $n-s-2$.

Assume that $\mathcal{W}_{\bfs\lambda^s}=\emptyset$ and let
$\rho_{\bfs\lambda^s}:=\rho_s(\bfs\lambda, \bfs\lambda^* \bfs X) \in
\Z[\bfs X]$. By the Nullstellensatz there exists
$\mu_{\bfs\lambda^s}\in \Z\setminus \{0\}$ satisfying
\begin{equation}\label{empty_case_equation}
  \mu_{\bfs\lambda^s}\in (F_1, \dots, F_{s+1},
  \rho_{\bfs\lambda^s})\Z[\bfs X].
\end{equation}
On the other hand, assume that $\mathcal{W}_{\bfs\lambda^s}\neq
\emptyset$ and let $Y_j:=\bfs\lambda_j \cdot \bfs X$ for $1\leq j
\leq n-s-1$. By \cite[Theorem 3.3]{Jelonek05} (see also
\cite[Theorem 3.1]{DaKrSo13}) there exists a nonzero polynomial
$B_{\bfs\lambda^s}\in \Z[Z_1,\dots,Z_{n-s-1}]$ with $\deg
B_{\bfs\lambda^s} \leq \deg  \mathcal{W}_{\bfs\lambda^s}$ such that
\begin{equation}\label{nonempty_case_implicit_equation}
B_{\bfs\lambda^s}\bigl(Y_1(\bfs x),\dots,Y_{n-s-1}(\bfs x)\bigr)=0
\end{equation}
for every $\bfs x\in \mathcal{W}_{\bfs\lambda^s}$. Since $\deg
\mathcal{W}_{\bfs\lambda^s}\leq \deg \mathcal{V}_{s+1} \deg
\rho_{\bfs\lambda^s}$, we have
\begin{equation}\label{nonempty_case_implicit_equation_degree_estimate}
\deg B_{\bfs\lambda^s}\leq 2(n-s+2)\delta_s^2\delta_{s+1}.
\end{equation}

As $B_{\bfs\lambda^s}(Y_1\klk Y_{n-s-1})$ vanishes on the variety
$\mathcal{W}_{\bfs\lambda^s}\subset\A^n$ defined by
$F_1,\dots,F_{s+1}$ and $\rho_{\bfs\lambda^s}$, by the
Nullstellensatz there exist $\beta_{\bfs\lambda^s}\in
\Z\setminus\{0\}$ and $\ell_{\bfs\lambda^s} \in \mathbb{N}$ such
that
\begin{equation}\label{nullstellensatz_implicit_equation}
\beta_{\bfs\lambda^s}B_{\bfs\lambda^s}(Y_1\klk
Y_{n-s-1})^{\ell_{\bfs \lambda^s}}\in
(F_1,\dots,F_{s+1},\rho_{\bfs\lambda^s})\Z[\bfs X].
\end{equation}

Now we are able to establish our condition for a good modular
reduction at the $s$th step. Let ${\sf M}_s\in \Z[\bfs \Lambda, Z_1
\klk Z_{n-s}]\setminus \{0\}$ be the polynomial defined by
\begin{equation}\label{def:polynomial_M_s}
{\sf M}_s:=\alpha_s\gamma_sA_s(\bfs \Lambda^*)\rho_s(\bfs \Lambda,
Z_1\klk Z_{n-s}),
\end{equation}
where $\alpha_s$ and $\gamma_s$ are the integers of
\eqref{alpha_s_definition} and \eqref{gamma_s_definition}
respectively. Observe that
\begin{equation}\label{eq: deg_M_s}
    \deg {\sf M}_s \leq 2(n-s+2)\delta_s^2.
\end{equation}
Further, let ${\sf C}_s\in \Z[\bfs\Lambda]$ be a nonzero coefficient
of ${\sf M}_s{\sf M}_{s+1}\in\Z[\bfs\Lambda][Z_1, \dots, Z_{n-s}]$.
For $\bfs\lambda\in \Z^{(n-s+1)n}\setminus \{0\}$ with ${\sf
C}_s(\bfs\lambda){\sf R}_s(\bfs \lambda)\neq 0$, define ${\sf
L}_{\bfs\lambda^s}\in \Z[Z_1, \dots, Z_{n-s}]\setminus\{0\}$ as
\begin{equation}\label{def:L_lambda_s}
{\sf L}_{\bfs\lambda^s}:=\left\{
\begin{array}{cl}
\mu_{\bfs\lambda^s}&\textrm{ if }
\mathcal{W}_{\bfs\lambda^s}=\emptyset,\\
\beta_{\bfs\lambda^s}B_{\bfs\lambda^s}&\textrm{ if }
\mathcal{W}_{\bfs\lambda^s}\neq \emptyset,\end{array} \right.
\end{equation}
where $\mu_{\bfs\lambda^s}$, $B_{\bfs\lambda^s}$ and
$\beta_{\bfs\lambda^s}$ are defined as in
\eqref{empty_case_equation},
\eqref{nullstellensatz_implicit_equation} and
\eqref{nonempty_case_implicit_equation}. Finally, define
$$
{\sf N}_{\bfs\lambda^s}:={\sf M}_s(\bfs\lambda, Z_1, \dots,
Z_{n-s}){\sf M}_{s+1}(\bfs\lambda^*, Z_1, \dots, Z_{n-s-1}){\sf
L}_{\bfs\lambda^s}(Z_1, \dots, Z_{n-s-1}).
$$
\begin{teo}\label{th: lifting_fibers_modulo_p}
Let $1 \leq s \leq r-1$. Let $\bfs\lambda\in \Z^{(n-s+1)n}$ and
$\bfs p:=(p_1\klk p_{n-s})\in \Z^{n-s}$ be such that ${\sf
C}_s(\bfs\lambda){\sf R}_s(\bfs\lambda)\neq 0$ and ${\sf
N}_{\bfs\lambda^s}(\bfs p)\neq 0$, and let $p$ be a prime number
with $p \nmid {\sf N}_{\bfs\lambda^s}(\bfs p)$. If
$Y_i:=\bfs\lambda_i \cdot \bfs X$ for $1 \leq i \leq n-s+1$, then
the following conditions are satisfied:
\begin{enumerate}
  \item $F_{1,p}, \dots, F_{s,p}$ generate a
  radical ideal in $\cfp[\bfs X]$ and define
  an equidimensional variety $\mathcal{V}_{s,p}\subset
  \mathbb{A}_{\scriptscriptstyle \cfp}^n$
  of dimension $n-s$ and degree $\delta_s$.  The same holds for
  $F_{1,p}, \dots, F_{s+1,p}$ and $\mathcal{V}_{s+1,p}$;
  \item the mapping $\pi_{s,p}:\mathcal{V}_{s,p}\rightarrow
  \mathbb{A}_{\scriptscriptstyle \cfp}^{n-s}$
  defined by $Y_{1,p},\dots,Y_{n-s,p}$ is a finite morphism,
  $\bfs p_p\in \fp^{n-s}$ is a lifting point of $\pi_{s,p}$,
  and $Y_{n-s+1,p}$ induces a primitive element of $\pi_{s,p}^{-1}(\bfs
  p_p)$;
  \item the mapping  $\pi_{s+1,p}:\mathcal{V}_{s+1,p}\rightarrow
  \mathbb{A}_{\scriptscriptstyle \cfp}^{n-s-1}$
  defined by $Y_{1,p},\dots,Y_{n-s-1,p}$ is a finite morphism. Furthermore,
  if $\bfs p^*:=(p_1\klk p_{n-s-1})$, then  $\bfs p^*_p$ is a lifting point of $\pi_{s+1,p}$ and
  $Y_{n-s,p}$ induces a primitive element of $\pi_{s+1,p}^{-1}(\bfs
  p^*_p)$;
  \item any $\bfs q\in \pi_{s,p}\bigl(\pi_{s+1,p}^{-1}(\bfs p^*_p)\bigr)$
  satisfies $\rho_{s,p}(\bfs\lambda_p, \bfs q)\neq 0$. In particular, any such $\bfs q$ is a
  lifting point of $\pi_{s,p}$ and $Y_{n-s+1,p}$ induces a primitive element
  of $\pi_{s,p}^{-1}(\bfs q)$.
\end{enumerate}
\end{teo}
\begin{proof}
Since $p\nmid {\sf M}_s(\bfs\lambda, \bfs p){\sf
M}_{s+1}(\bfs\lambda^*, \bfs p^*)$, the first three assertions
follow by Theorem \ref{th: conditions_mod_p_summary}.

To prove the last assertion, let $\bfs q\in
\pi_{s,p}\bigl(\pi_{s+1,p}^{-1}(\bfs p^*_p)\bigr)$. Then there
exists $\bfs x\in \pi_{s+1,p}^{-1}(\bfs p_p^*)$ such that $\bfs
q=\bigl(\bfs p_p^*,Y_{n-s,p}(\bfs x)\bigr)$. Suppose that the
variety $\mathcal{W}_{\bfs\lambda^s}$ of
\eqref{w_s_definition_equation} is empty. Considering
\eqref{empty_case_equation} modulo $p$, and taking into account that
$p\nmid \mu_{\bfs\lambda^s}$, we deduce that $F_{1,p}, \dots,
F_{s+1,p}$ and $\rho_{\bfs\lambda^s, p}$ generate the unit ideal of
$\cfp[\bfs X]$. As $\bfs x\in \mathcal{V}_{s+1,p}$, it follows that
$\rho_{s,p}(\bfs\lambda_p, \bfs q)=\rho_{\bfs\lambda^s, p}(\bfs
x)\neq 0$. Since $p \nmid {\sf M}_s(\bfs\lambda, \bfs p)$, by
Theorem \ref{th: conditions_mod_p_summary} we conclude  that $\bfs
q$ is a lifting point of $\pi_{s,p}$ and  $Y_{n-s+1,p}$ induces a
primitive element of $\pi_{s,p}^{-1}(\bfs q)$. On the other hand, if
$\mathcal{W}_{\bfs\lambda^s}\neq \emptyset$, then considering
\eqref{nullstellensatz_implicit_equation} modulo $p$ and taking into
account that $p\nmid \beta_{\bfs\lambda^s}$ we see that
$$
B_{\bfs\lambda^s,
p}(Y_{1,p},\ldots,Y_{n-s-1,p})^{\ell_{\bfs\lambda^s}}\in
(F_{1,p},\dots,F_{s+1,p},\rho_{\bfs\lambda^s,p})\cfp[\bfs X].
$$
This implies that $B_{\bfs\lambda^s, p}$ vanishes on
$\mathcal{V}_{s+1,p}\cap \{\rho_{\bfs\lambda^s, p}=0\}$. Further,
the fact that $p\nmid B_{\bfs\lambda^s}(\bfs p^*)$ implies
$B_{\bfs\lambda^s, p}(\bfs x)=B_{\bfs\lambda^s,p}(\bfs p_p^*)\neq
0$, and then $\rho_{s,p}(\bfs\lambda_p, \bfs q)=\rho_{\bfs\lambda^s,
p}(\bfs x)\neq 0$.  Arguing as before we deduce that $\bfs q$ is a
lifting point of $\pi_{s,p}$ and $Y_{n-s+1,p}$ induces a primitive
element of $\pi_{s,p}^{-1}(\bfs q)$.
\end{proof}

\begin{remark}\label{fiber_of_the_fiber_remark}
With hypotheses as in Theorem \ref{th: lifting_fibers_modulo_p}, let
$\pi_{s+1,p}^{-1}(\bfs p_p^*)=\{\bfs x^1 \klk \bfs
x^{\delta_{s+1}}\}$. Since $Y_{n-s,p}$ induces a primitive element
of $\pi_{s+1,p}^{-1}(\bfs p_p^*)$, it separates $\bfs x^1\klk \bfs
x^{\delta_{s+1}}$. Therefore, if $q\in \cfp[T]$ is the minimal
polynomial of $Y_{n-s,p}$ over $\pi_{s+1,p}^{-1}(\bfs p_p^*)$, then
its roots in $\cfp$ are $Y_{n-s,p}(\bfs x^1), \dots, Y_{n-s,p}(\bfs
x^{\delta_{s+1}})$. Since
$$\pi_{s,p}\bigl(\pi_{s+1,p}^{-1}(\bfs p_p^*)\bigr)=\Bigl\{\bigl(\bfs
p_p^*, Y_{n-s,p}(\bfs x^1)\bigr), \dots, \bigl(\bfs p_p^*,
Y_{n-s,p}(\bfs x^{\delta_{s+1}})\bigr)\Bigr\},$$
we can rephrase item $(4)$ of Theorem \ref{th:
lifting_fibers_modulo_p} in the following way:
$\rho_{s,p}\bigl(\bfs\lambda_p,(\bfs p_p^*,a)\bigr)\neq 0$ for every
root $a\in \cfp$ of $q$. Thus, $(\bfs p_p^*,a)$ is a lifting point
of $\pi_{s,p}$ and $Y_{n-s+1,p}$ induces a primitive element of
$\pi_{s,p}^{-1}(\bfs p_p^*, a)$.
\end{remark}
%
%
\subsection{Simultaneous Noether normalization and lifting fibers}
From now on, let $\bfs\Lambda:=(\Lambda_{ij})_{1\leq i \leq n, 1\leq
j \leq n}$ denote a set of $n^2$ indeterminates over $\mathbb{Q}$.
For $1\leq s \leq r$, we write $\bfs\Lambda^s:=(\Lambda_{ij})_{1\leq
i \leq n, 1\leq j \leq n-s+1}$. Further, for
$\bfs\lambda:=(\lambda_{ij})_{1 \leq i \leq n, 1 \leq j \leq n} \in
\Z^{n^2}$, we denote $\bfs\lambda^s:=(\lambda_{ij})_{1 \leq i \leq
n-s+1, 1 \leq j \leq n}$. Let ${\sf R} \in
\overline{\mathbb{Q}}[\bfs\Lambda]\setminus \{0\}$ be the polynomial
defined by
\begin{equation}\label{R_definition}
{\sf R}:= \prod_{s=1}^{r-1}{\sf C}_s{\sf R}_s.
\end{equation}
Since $\deg {\sf C}_s \leq \deg {\sf M}_s + \deg {\sf M}_{s+1}$,
taking into account \eqref{eq: deg_M_s} and the estimate for the
degree of ${\sf R}_s$ of Lemma \ref{lemma:
pure_dimesion_r_subvariety}, we easily deduce that
\begin{equation}\label{deg_R_estimate}
\deg {\sf R} \leq D:=(2n-r+4)r(\delta^3+2\delta^2).
\end{equation}

Let $\bfs\lambda \in \Z^{n^2}\setminus \{0\}$ be such that ${\sf
R}(\bfs\lambda)\neq 0$ and define ${\sf N}_{\bfs\lambda}\in \Z[Z_1,
\dots, Z_{n-1}]\setminus \{0\}$ as
\begin{equation}\label{N_lambda_definition}
{\sf N}_{\bfs\lambda}:={\sf M}_r(\bfs\lambda^r, Z_1, \dots,
Z_{n-r})\prod_{s=1}^{r-1}{\sf M}_s(\bfs\lambda^s, Z_1, \dots,
Z_{n-s}){\sf L}_{\bfs\lambda^s}(Z_1, \dots, Z_{n-s-1}).
\end{equation}
Observe that
\begin{align*}
\deg {\sf N}_{\bfs\lambda}\le\sum_{s=1}^r\deg {\sf M}_s+
\sum_{s=1}^r\deg {\sf L}_{\bfs\lambda^s}\le 2(\delta^3+\delta^2)
\sum_{s=1}^{r-1}(n-s+2)+2(n-r+2)\delta^2\le D.
\end{align*}
Let $\bfs p:=(p_1, \dots, p_{n-1})\in \Z^{n-1}$ be such that ${\sf
N}_{\bfs\lambda}(\bfs p)\neq 0$ and denote $\bfs p^s:=(p_1, \dots,
p_{n-s})$ for $1 \leq s \leq r$. With hypotheses as above we easily
obtain the following result.
\begin{teo} \label{th:_simultaneous_noether_normalization}
Let $\bfs\lambda\in \Z^{n^2}\setminus \{0\}$ and $\bfs p\in
\Z^{n-1}$ be such that $\det(\bfs\lambda){\sf R}(\bfs\lambda)\neq 0$
and ${\sf N}_{\bfs\lambda}(\bfs p)\neq0$. Let
$\mathfrak{N}:=\det(\bfs\lambda){\sf N}_{\bfs\lambda}(\bfs p)$ and
$Y_i:=\bfs\lambda_i\cdot\bfs X$ for $1\leq i \leq n$. If $p$ is  a
prime number such that $p\nmid \mathfrak{N}$, then $Y_{1,p}\klk
Y_{n,p}$ define a new set of variables for $\cfp[\bfs X]$ and
conditions $(1)$--$(4)$ of Theorem \ref{th: lifting_fibers_modulo_p}
are satisfied for $1 \leq s \leq r-1$ with $\bfs p:=\bfs p^s$ and
$\bfs p^*:=\bfs p^{s+1}$. In particular, $F_{1,p}, \dots, F_{r,p}$
define a reduced regular sequence in $\cfp[\bfs X]$.
\end{teo}

In the sequel, a prime $p$ as in Theorem
\ref{th:_simultaneous_noether_normalization} will be called
``lucky'' and a reduction modulo such a prime $p$ is called
``good''.

We end this section by discussing Kronecker representations for a
good modular reduction. Given $\bfs\lambda:=(\lambda_{ij})_{1\leq i,
j \leq n}\in \Z^{n^2}$ and $\bfs p:=(p_1,\dots,p_{n-1})\in \Z^{n-1}$
satisfying the hypotheses of Theorem
\ref{th:_simultaneous_noether_normalization}, define
$Y_i:=\bfs\lambda_i\cdot \bfs X$ for $1\leq i \leq n$, and let
$R_s:=\Q[Y_1, \dots, Y_{n-s}]$ and $B_s:=\Q[\mathcal{V}_s]$ for
$1\leq s \leq r$. Since $A_s(\bfs \lambda^{s+1})\rho_s(\bfs
\lambda^s, \bfs p^s)\neq 0$ for $1\leq s\leq r$, by Theorem
\ref{th:_lifting_point_rank_and_finite_morphism} the following
conditions are satisfied:
\begin{itemize}
  \item $Y_1 \klk Y_{n-s}$ are in Noether position with respecto to
  $\mathcal{I}_s$;
  \item $\bfs p^s$ is a lifting point of the finite morphism $\pi_s: \mathcal{V}_s\rightarrow \A^{n-s}$ defined by $Y_1 \klk
  Y_{n-s}$;
  \item $B_s$ is a free $R_s$--module of rank equal to $\delta_s$.
\end{itemize}
Let $\mathcal{I}_s:=(F_1 \klk F_s)$ and
$\mathcal{J}_s:=\mathcal{I}_s +(Y_1-p_1, \dots, Y_{n-s}-p_{n-s})$
for $1\leq s \leq r$ and $\mathcal{K}_s:=\mathcal{I}_s +(Y_1-p_1,
\dots, Y_{n-s-1}-p_{n-s-1})$ for $1\leq s \leq r-1$. According to
Lemma \ref{lemma: positive_dimension_lifting_fiber}, $\mathcal{J}_s$
and $\mathcal{K}_s$ are the vanishing ideals of the lifting fiber
$\mathcal{V}_{\bfs p^s}$ and the lifting curve $\mathcal{W}_{\bfs
p^s}$ respectively. Further, identifying $\mathcal{I}_s$ with its
image in $\Q[Y_{n-s+1} \klk Y_n]$ and $\mathcal{K}_s$ with its image
in  $\mathbb{Q}[Y_{n-s}\klk Y_n]$ as in Lemma \ref{lemma:
lifting_curve_isomorphism}, the following conditions hold for $1\leq
s \leq r$:
\begin{itemize}
  \item $\mathbb{Q}[Y_{n-s+1} \klk Y_n]/\mathcal{J}_s$ is a $\Q$--vector space of dimension
  $\delta_s$;
  \item $Y_{n-s}\klk Y_n$ are in Noether position with respect to
  $\mathcal{K}_s$;
  \item $\mathbb{Q}[Y_{n-s}\klk Y_n]/\mathcal{K}_s$ is a free
$\mathbb{Q}[Y_{n-s}]$--module of rank equal to
$\mbox{rank}_{R_s}\mathbb{Q}[\mathcal{V}_s]$.
\end{itemize}

We can obtain Kronecker representations of $\mathcal{I}_s$,
$\mathcal{J}_s$, and $\mathcal{K}_s$ as in Section \ref{subsec:
kronecker_representations_from_specializations_of_the_Chow_form},
namely let $T$ be a new indeterminate and define $Q^s, W^s_{n-s+2}
\klk W^s_n\in R_s[T]$ by
\begin{equation}\label{eq: def: kronecker_representation_of_I_s}
Q^s:=\frac{P_s(\bfs\lambda^s, Y_1, \dots,
Y_{n-s},T)}{A_s(\bfs\lambda^{s+1})},\quad
W^s_j:=-\sum_{k=1}^{n}\frac{\lambda_{jk}}{A_s(\bfs\lambda^{s+1})}\frac{\partial P_s}{\partial
\Lambda_{n-s+1,k}}(\bfs\lambda^s, Y_1, \dots, Y_{n-s},T)
\end{equation}
for $n-s+2\leq j \leq n$, where $P_s\in\Z[\bfs\Lambda^s,Z_1\klk
Z_{n-s+1}]$ is a primitive Chow form of $\mathcal{V}_s$.
Propositions \ref{prop: kronecker_representation_of_I}, \ref{prop:
kronecker_representation_of_J} and \ref{prop:
kronecker_representation_of_K} then read as follows.
\begin{prop} \label{prop: kronecker_representation_of_I_s_J_s_and_K_s}
The following assertions hold:
\begin{itemize}
  \item the polynomials $Q^s, W^s_{n-s+2} \klk W^s_n$ form the Kronecker representation of $\mathcal{I}_s$
  with primitive element $Y_{n-s+1}$;
  \item the polynomials $Q^s(\bfs p^s,T), W^s_{n-s+2}(\bfs p^s, T) \klk\! W^s_n(\bfs p^s,T)$ form the Kronecker representation of
$\mathcal{J}_s$ with primitive element $Y_{n-s+1}$;
  \item the polynomials
 $Q^s(\bfs p^{s+1},\!Y_{n-s},T),W^s_{n-s+2}(\bfs
p^{s+1},\!Y_{n-s},T) \klk W^s_n(\bfs p^{s+1},\!Y_{n-s},T)$ form the
Kronecker representation of  $\mathcal{K}_s$  with primitive element
$Y_{n-s+1}$.
\end{itemize}
\end{prop}
Now let $p$ be a prime number as in Theorem
\ref{th:_simultaneous_noether_normalization}. Let
$\mathcal{I}_{s,p}$, $\mathcal{J}_{s,p}$ and $\mathcal{K}_{s,p}$ be
the ideals of  $\cfp[\bfs X]$ defined by
$\mathcal{I}_{s,p}:=(F_{1,p}\klk F_{s,p})$ and
$\mathcal{J}_{s,p}:=\mathcal{I}_{s,p}+(Y_{1,p}-p_{1,p}, \dots,
Y_{n-s,p}-p_{n-s,p})$ for $1\leq s\leq r$, and
$\mathcal{K}_{s,p}:=\mathcal{I}_{s,p}+(Y_{1,p}-p_{1,p}, \dots,
Y_{n-s-1,p}-p_{n-s-1,p})$ for $1\leq s\leq r-1$. By Theorem
\ref{th:_simultaneous_noether_normalization} the following
conditions are satisfied for $1\leq s\leq r$:
\begin{itemize}
   \item $\mathcal{I}_{s,p}$ is a radical, equidimensional ideal of dimension
   $n-s$;
   \item the variables $Y_{1,p}, \dots, Y_{n,p}$ are in Noether position with respect to
   $\mathcal{I}_{s,p}$;
   \item the mapping $\pi_{s,p}:\mathcal{V}_{s,p}\rightarrow \A_{\scriptscriptstyle \cfp}^{n-s}$ defined
   by $Y_{1,p}, \dots, Y_{n-s,p}$ is a finite morphism and $\bfs p_p$ is a lifting point of
   $\pi_{s,p}$;
   \item $P_{s,p}$ is a Chow form of $\mathcal{V}_{s,p}$.
 \end{itemize}
It follows that $\mathcal{I}_{s,p}$, $\mathcal{J}_{s,p}$ and
$\mathcal{K}_{s,p}$ are the defining ideals of the variety
$\mathcal{V}_{s,p}$, the lifting fiber $\mathcal{V}_{\bfs p_p^s}$
and the lifting curve $\mathcal{W}_{\bfs p_p^{s+1}}$ respectively.
Since $p \nmid A_s(\bfs\lambda^{s+1})$, the polynomials $Q^s_p,
W^s_{1,p}, \dots, W^s_{n,p}\in \fp[T]$ are well--defined, and we
have the following result.
\begin{prop}\label{prop: kronecker_representations_mod_p}
The following assertions hold:
\begin{itemize}
  \item $Q_p^s, W^s_{n-s+2,p}, \dots, W^s_{n,p}$ form the Kronecker
  representation of $\mathcal{I}_{s,p}$ with  primitive element
  $Y_{n-s+1,p}$;
  \item $Q_p^s(\bfs p_p^s,T), W^s_{n-s+2,p}(\bfs p_p^s,T), \dots, W^s_{n,p}(\bfs p_p^s,T)$
  form the Kronecker representation of $\mathcal{J}_{s,p}$  with primitive element
  $Y_{n-s+1,p}$;
  \item $Q_p^s(\bfs p_p^{s+1},\!Y_{n-s,p},T), W^s_{n-s+2,p}(\bfs p_p^{s+1},Y_{n-s,p},T),
  \dots, W^s_{n,p}(\bfs p_p^{s+1},Y_{n-s,p},T)$
  form the Kronecker representation of $\mathcal{K}_{s,p}$  with primitive element $Y_{n-s+1,p}$.
\end{itemize}
\end{prop}
\begin{proof} From  \eqref{eq: def: kronecker_representation_of_I_s} we
deduce that
\begin{align*}
Q^s_p =&\frac{P_{s,p}\bigl(\bfs\lambda^s_p, Y_{1,p}\klk
Y_{n-s,p},T\bigr)} {A_{s,p}(\bfs\lambda^{s+1}_p)}, \\ W^s_{j,p}
=&-\sum_{k=1}^{n}\frac{\lambda_{jk,p}}{A_{s,p}(\bfs\lambda^{s+1}_p)}\frac{\partial
P_{s,p}}{\partial \Lambda_{n-s+1,k}}(\bfs\lambda_p^s, Y_{1,p} \klk
Y_{n-s,p},T)\quad(n-s+2 \leq j \leq n).
\end{align*}
As $P_{s,p}$ is a Chow form of $\mathcal{V}_{s,p}$, the proposition
follows taking into account the condition $p \nmid
A_s(\bfs\lambda^{s+1})\rho_s(\bfs\lambda^s, \bfs p^s)$ and arguing
as in Propositions \ref{prop: kronecker_representation_of_I},
\ref{prop: kronecker_representation_of_J} and \ref{prop:
kronecker_representation_of_K}.
\end{proof}
%
%
\section{Computation of a Kronecker representation}
\label{section: computation_of_a_kronecker_representation}
Let $F_1 \klk F_r\in \Z[\bfs X]$ be, as in Section \ref{section:
modular_simultaneous_noether_normalization}, polynomials defining a
reduced regular sequence. In this section we establish an upper
bound on the bit complexity of computing a Kronecker representation
of a zero--dimensional fiber $\pi_r^{-1}(\bfs p^r)$ of
$\mathcal{V}(F_1 \klk F_r)$. For this purpose, following the
approach of \cite{GiLeSa01}, we perform this computation modulo a
prime number $p$ and apply $p$--adic lifting to recover the integers
of a Kronecker representation of $\pi_r^{-1}(\bfs p^r)$ over $\Q$.
Assuming that a ``lucky'' prime $p$ is given, the complexity of
computing a Kronecker representation of a zero--dimensional fiber of
$\mathcal{V}(F_{1,p} \klk F_{r,p})$ was analyzed in \cite{CaMa06a}.
On the other hand, the complexity of the $p$--adic lifting step was
analyzed in \cite{GiLeSa01}. Accordingly, in this section we analyze
the cost of computing a ``lucky'' prime (Proposition \ref{prop:
computation_of_a_lucky_prime}), and then obtain an upper bound on
the bit complexity of computing a Kronecker representation of
$\pi_r^{-1}(\bfs p^r)$ over $\Q$ (Theorem \ref{th:
bit_complexity_kronecker_solver}).


%
%
\subsection{Computation of a Kronecker representation modulo $p$}
Let ${\sf S}:=\{0, \dots, \textsf{a}\}$ and ${\sf T}:=\{0, \dots,
\textsf{b}\}$, where  $\textsf{a}:=\lfloor 8D\rfloor$ and
$\textsf{b}:=\lfloor {9\sf D}\rfloor$. Assume that we have randomly
chosen $(\bfs\lambda, \bfs p)\in {\sf S}^{n^2} \times {\sf T}^{n-1}$
such that ${\sf R}(\bfs\lambda)\neq 0$ and ${\sf
N}_{\bfs\lambda}(\bfs p)\neq 0$. The following result asserts that
this can be done with a high probability of success.
\begin{lema}\label{lema: lambda and p_probability}
Let $(\bfs\lambda, \bfs p)$ be a point chosen uniformly at random in
${\sf S}^{n^2}\times {\sf T}^{n-1}$. Then the probability that ${\sf
R}(\bfs\lambda)\neq 0$ and  ${\sf N}_{\bfs\lambda}(\bfs p)\neq 0$ is
greater than $\frac{7}{9}$.
\end{lema}
\begin{proof} Since $\deg {\sf R}\leq
D$, by Lemma \ref{lemma: zippel_schwartz} we see that for a random
choice of $\bfs\lambda$ in ${\sf S}^{n^2}$, the probability that
${\sf R}(\bfs\lambda)\neq 0$ is greater than $\frac{7}{8}$.
Similarly, as $\deg ({\sf N}_{\bfs\lambda})\leq D$, for a point
$\bfs p$ chosen uniformly at random in ${\sf T}^{n-1}$, the
conditional probability that ${\sf N}_{\bfs\lambda}(\bfs p)\neq 0$,
given that ${\sf R}(\bfs\lambda)\neq 0$, is greater than
$\frac{8}{9}$. This finishes the proof of the lemma.
\end{proof}

For such a choice of $\bfs\lambda$ and $\bfs p$, let $\mathfrak{N}$
be the integer of Theorem
\ref{th:_simultaneous_noether_normalization}. According to Theorem
\ref{th: app: mathfrak_N_height}, we have
\begin{equation}\label{eq: log_frak_N_height}
h(\mathfrak{N})\leq\mathfrak{H} \quad  \textrm{and} \quad \log
\mathfrak{H} \in \mathcal{O}^{\sim}\bigr(\log (d^r nh)\bigl).
\end{equation}
Now we can estimate the complexity of computing a ``lucky'' prime
$p$ of ``low'' bit length.
\begin{prop} \label{prop: computation_of_a_lucky_prime}
There is a probabilistic algorithm which takes $\mathfrak{H}$ as
input and computes a prime $p$ with $12\mathfrak{H}+1\leq p \leq
24\mathfrak{H}$ such that $p \nmid \mathfrak{N}$. The algorithm uses
$\SO \bigr(\log^2 (d^r nh)\bigl)$ bit operations and returns a right
result with probability at least $\frac{3}{4}$.
\end{prop}
\begin{proof}
The proposition follows applying Lemma \ref{lemma: finding_primes}
with $B=m\mathfrak{H}$, $M=\mathfrak{N}$, $m=12$, and $k=5+
\log\log(12\mathfrak{H})$, and taking into account \eqref{eq:
log_frak_N_height}.
\end{proof}
%
%
Assume that we have computed a ``lucky'' prime $p$ as in Proposition
\ref{prop: computation_of_a_lucky_prime}. Further, assume that we
are given a straight--line program of length at most $L$ which
represents the polynomials $F_{1,p} \klk F_{r,p}$. Since the integer
$\mathfrak{H}$ of \eqref{eq: log_frak_N_height} can certainly be
chosen with $\mathfrak{H}\geq 5n^2d\delta^4$, we can assume that $p>
60n^2d\delta^4$. Thus we can apply the algorithm described in
\cite{CaMa06a} to compute a Kronecker representation of the lifting
fiber $V_{\bfs p_p^r}$.

The algorithm starts computing the Kronecker representation of the
fiber $V_{\bfs p_p^1}$ of the hypersurface $\{F_{1,p}=0\}$, with
$Y_{n,p}$ as primitive element. With the notations of Proposition
\ref{prop: kronecker_representations_mod_p}, such a Kronecker
representation only consists of the minimal polynomial
$Q^1\bigl(\bfs p^1,T\bigr)$ of $Y_{n,p}$ modulo $\mathcal{J}_{1,p}$.
Since, with notations as in Section \ref{section:
lifting_points_and_lifting_fibers}, we have
$\mathcal{J}_{1,p}=\bigl(F_{1,p}(\bfs p_p^1,Y_{n,p})\bigr)$, we see
that $\cfp[V_{\bfs p_p^1}]=\cfp[Y_{n,p}]/\bigl(F_{1,p}(\bfs
p_p^1,Y_{n,p})\bigr)$. It follows that $Q^1\bigl(\bfs p_p^1,T\bigr)$
equals the polynomial $F_{1,p}(\bfs p_p^1,T)$ divided by its leading
coefficient.

Then the algorithm proceeds in $r-1$ stages. For $s\in \{1\klk
r-1\}$, the $s$th stage takes as input a Kronecker representation
$Q^s(\bfs p_p^s,T), W^s_{n-s+2}(\bfs p_p^s,T) \klk W^s_n(\bfs
p_p^s,T)$ of $\mathcal{J}_{s,p}$ and outputs a Kronecker
representation $Q^{s+1}(\bfs p_p^{s+1},T), W^{s+1}_{n-s+1}(\bfs
p_p^{s+1},T) \klk W^{s+1}_n(\bfs p_p^{s+1},T)$ of
$\mathcal{J}_{s+1,p}$. This stage, whose cost is analyzed below,
consists in two main tasks, which are called the lifting step and
the intersection step.
%
%
\subsubsection{Lifting step}
In the lifting step we compute the Kronecker representation
\linebreak $Q^s(\bfs p_p^{s+1}, Y_{n-s,p},T), W^s(\bfs
p_p^{s+1},Y_{n-s,p},T) \klk W^s(\bfs p_p^{s+1},Y_{n-s,p},T)$ of
$\mathcal{K}_{s,p}$ with primitive element $Y_{n-s+1,p}$, from the
univariate representation of $\mathcal{J}_{s,p}$  with $Y_{n-s+1,p}$
as primitive element. By Proposition \ref{prop:
kronecker_representations_mod_p}, such a Kronecker representation is
defined by the specializations of $Q_p^s, W^s_{n-s+2,p} \klk
W^s_{n,p}$ at $Y_{1,p}=p_{1,p} \klk Y_{n-s-1,p}=p_{n-s-1,p}$.  Let
$\widehat{R}_{s,p}:=\fp[\![Y_{1,p}-p_{1,p}\klk
Y_{n-s,p}-p_{n-s,p}]\!]$. By Remark \ref{total_degrees_remark} we
conclude that it suffices to compute the approximation of $Q_p^{s},
W^s_{n-s+2,p} \klk W^s_{n,p}$ to precision $(Y_{1,p}-p_{1,p} \klk
Y_{n-s,p}-p_{n-s,p})^{\delta_s+1}$ in $\widehat{R}_{s,p}[T]$.
As $F_{1,p}(\bfs p_p^{s+1},Y_{n-s,p}\klk Y_{n,p}) \klk F_{s,p}(\bfs
p_p^{s+1},Y_{n-s,p}\klk Y_{n,p})$ generate the radical ideal
$\mathcal{K}_{s,p}$ and form a regular sequence of
$\cfp[Y_{n-s,p}\klk Y_{n,p}]$  by Lemma \ref{lema:
reduced_regular_sequence}, applying the Global Newton algorithm of
\cite[II.4]{GiLeSa01} we have the following result.
\begin{prop}\label{prop: lifting_step}
There exists a deterministic  algorithm that takes as input:
\begin{itemize}
  \item a straight--line program of length $L$ which represents the polynomials $F_{1,p} \klk
  \!F_{s,p}$;
  \item the dense representation of the polynomials in $\fp[T]$
  which form the univariate representation of $\mathcal{J}_{s,p}$ with primitive element
  $Y_{n-s+1,p}$;
\end{itemize}
and outputs the dense representation of the polynomials in
$\fp[Y_{n-s,p},T]$ which form the Kronecker representation of
$\mathcal{K}_{s,p}$ with  primitive element $Y_{n-s+1,p}$. The
algorithm uses $\SO\bigr((nL+n^5)\delta_s^2\log p\bigl)$ bit
operations.
\end{prop}
%
%
\subsubsection{Intersection step}
The input of the intersection step is  the output of the algorithm
underlying Proposition \ref{prop: lifting_step}, namely the
Kronecker representation of $\mathcal{K}_{s,p}$  with primitive
element $Y_{n-s+1,p}$. Let $Q^s(\bfs p_p^{s+1}\!,\! Y_{n-s,p}, T),
V^s_{n-s+2}(\bfs p_p^{s+1}\!,\! Y_{n-s,p}, T), \dots,$ $V^s_n(\bfs
p_p^{s+1}, Y_{n-s,p}, T)$ be the corresponding univariate
representation. The output is the univariate representation
$Q^{s+1}(\bfs p_p^{s+1}, T), V^s_{n-s+1}(\bfs p_p^{s+1}, T),
\dots,V^{s+1}_n(\bfs p_p^{s+1}, T)$  of $\mathcal{J}_{s+1,p}$ with
primitive element $Y_{n-s,p}$. 
Consider $F_{s+1,p}$ as an element of $\fp[Y_{1,p} \klk Y_{n,p}]$
and define $h\in \fp(Y_{n-s,p})[T]$ by
\begin{align*}
h(T):=F_{s+1,p}\bigl(\bfs p_p^{s+1}, Y_{n-s,p}, T,V^{s}_{n-s+2}(\bfs
p_p^{s+1}, T) \klk V^{s}_n(\bfs p_p^{s+1}, T)\bigr)\qquad \\\mod
\bigl(Q^{s}(\bfs p_p^{s+1}, Y_{n-s,p}, T)\bigr).
\end{align*}
The following result provides an expression for $Q^{s+1}(\bfs
p_p^{s+1},T)$ from which we shall be able to compute it efficiently.
\begin{prop} We have
\[
Q^{s+1}(\bfs p_p^{s+1},Y_{n-s,p})=\epsilon\,
\mathrm{Res}_{T}\bigl(h(T),Q^{s}(\bfs p_p^{s+1}, Y_{n-s,p},T)\bigr),
\]
for some $\epsilon \in \fp\setminus \{0\}$.
\end{prop}
\begin{proof}
Let $M_h$ be the matrix of the homothety of multiplication by $h$ in
\linebreak $\cfp(Y_{n-s,p})[T]/\bigl(Q^{s}(\bfs p_p^{s+1},
Y_{n-s,p}, T)\bigr)$ with respect to the basis $\{1, T\klk
T^{\delta_s-1}\}$. We have (see, e.g.,  \cite[Proposition
5.4]{ElMo07}):
\[
\det (M_{h})=\mathrm{Res}_{T}\bigl(h(T),Q^{s}(\bfs p_p^{s+1},
Y_{n-s,p}, T)\bigr).
\]
Consider the isomorphism of $\cfp(Y_{n-s,p})$--algebras
\[
\Phi: \cfp(Y_{n-s,p})[Y_{n-s+1,p}\klk
Y_{n,p}]/\overline{\mathcal{K}}_{s,p}^e \rightarrow
\cfp(Y_{n-s,p})[T]/\bigl(Q^{s}(\bfs p_p^{s+1}, Y_{n-s,p}, T)\bigr),
\]
which maps $Y_{n-s+1,p} \mod \overline{\mathcal{K}}_{s,p}^e$ to $T
\mod \bigl(Q^{s}(\bfs p_p^{s+1}, Y_{n-s,p}, T)\bigr)$. Let $S$ be a
new indeterminate and $\chi \in \cfp[Y_{n-s,p}][S]$ the
characteristic polynomial of the homothety by $F_{s+1,p}(\bfs
p_p^{s+1}, Y_{n-s,p} \klk Y_{n,p})$ modulo
$\overline{\mathcal{K}}_{s,p}^e$. Let $\chi_0\in \fp[Y_{n-s,p}]$ be
the constant term of $\chi$. Since $\Phi$ maps $F_{s+1,p}(\bfs
p_p^{s+1}, Y_{n-s,p}, \dots, Y_{n,p})\mod
\overline{\mathcal{K}}_{s,p}^e$ to $h \mod \bigl(Q^{s}(\bfs
p_p^{s+1}, Y_{n-s,p}, T)\bigr)$, $\chi$ coincides with the
characteristic polynomial of the homothety of multiplication by $h$
modulo $\bigl(Q^{s}(\bfs p_p^{s+1}, Y_{n-s,p}, T)\bigr)$. Thus
$\chi_0=(-1)^{\delta_s}\det (M_{h})$. On the other hand, as the
hypersurface $\{F_{s+1,p}(\bfs p_p^{s+1}, Y_{n-s,p}, \dots,
Y_{n,p})=0\}$ intersects the lifting curve $\mathcal{W}_{\bfs
p_p^{s+1}}$ in the finite fiber $V_{\bfs p^{s+1}_p}$, the polynomial
$F_{s+1,p}\bigl(\bfs p_p^{s+1}, Y_{n-s,p}, \dots, Y_{n,p}\bigr)$ is
not a zero divisor in $\cfp[Y_{n-s,p}\klk
Y_{n,p}]/\overline{\mathcal{K}}_{s,p}$. Since
$\overline{\mathcal{J}}_{s+1,p}=\overline{\mathcal{K}}_{s,p}+\bigl(F_{s+1,p}(\bfs
p_p^{s+1}, Y_{n-s,p}, \dots, Y_{n,p})\bigr)$, by \cite[Proposition
2.7]{DuLe08} we deduce that $\chi_0(T)$ coincides, up to multiples
in $\fp\setminus \{0\}$, with the characteristic polynomial of
$Y_{n-s,p}$ in $\cfp[Y_{n-s,p}\klk
Y_{n,p}]/\overline{\mathcal{J}}_{s+1,p}$. Since $Y_{n-s,p}$ induces
a primitive element for $\overline{\mathcal{J}}_{s+1,p}$, we
conclude that $\chi_0(T)=\epsilon Q^{s+1}(\bfs p_p^{s+1}, T)$ for
some $\epsilon \in \cfp\setminus \{0\}$. This finishes the proof of
the Proposition.
\end{proof}

Now we discuss the computation of the polynomials
$V^{s+1}_{n-s+1}(\bfs p_p^{s+1}\!,\! T), \dots,\! V^{s+1}_{n}(\bfs
p_p^{s+1}\!,\!T)$. Let $Q^{s+1}(\bfs p_p^{s+1},T)=q_1\cdots
q_{\ell}$ be the irreducible factorization of $Q^{s+1}(\bfs
p_p^{s+1},T)$ in $\fp[T]$. We describe below how to compute
$V^{s+1}_j(\bfs p_p^{s+1}, T) \mod q_k$ for $n-s+1 \leq j \leq n$
and $1\leq k \leq \ell$. Then the $V^{s+1}_j(\bfs p_p^{s+1}, T)$ can
be recovered by means of the Chinese remainder theorem. For $1 \leq
k\leq \ell$, let $a$ be the residue class of $T$ in $\fp[T]/(q_k)$.
Set $\mathbb{L}=\fp[T]/(q_k)$. Thus $\mathbb{L}:=\fp[a]$ is a finite
extension of $\fp$ which contains the root $a$ of $Q^{s+1}(\bfs
p_p^{s+1},T)$. Let $\overline{\mathbb{L}}$ be the algebraic closure
of $\mathbb{L}$. We have a field isomorphism
$\overline{\mathbb{L}}=\cfp$. By Remark
\ref{fiber_of_the_fiber_remark} we know that
$\rho_s\bigl(\bfs\lambda_p^s,(\bfs p_p^{s+1},a)\bigr)\neq 0$. Thus
$(\bfs p_p^{s+1},a)$ is a lifting point of $\pi_{s,p}$ and
$Y_{n-s+1,p}$ induces a primitive element of the lifting fiber
$\pi_{s,p}^{-1}(\bfs p_p^{s+1},a)$. Moreover, $\mathcal{K}_{s,p}+
(Y_{n-s}-a)$ is a radical ideal of $\cfp[\bfs X]$ by Lemma
\ref{lemma: positive_dimension_lifting_fiber}, and therefore it is
the vanishing ideal of $\pi_{s,p}^{-1}(\bfs p_p^{s+1},a)$. Let $q_a,
w_{a,n-s+2} \klk w_{a,n}$ be the Kronecker representation of
$\mathcal{K}_{s,p}+ (Y_{n-s}-a)$ with primitive element
$Y_{n-s+1,p}$. Let $Q_p^{s}, W^{s}_{n-s+2,p} \klk W^{s}_{n,p}$ be
the Kronecker representation of $\mathcal{I}_{s,p}$ with primitive
element $Y_{n-s+1,p}$. By Proposition \ref{prop:
kronecker_representation_of_J} the specializations of $Q_p^{s},
W^{s}_{n-s+2,p} \klk W^{s}_{n,p}$ at $Y_{1,p}=p_{1,p} \klk
Y_{n-s-1,p}=p_{n-s-1,p}, Y_{n-s,p}=a$ coincide with $q_a,
w_{a,n-s+2} \klk w_{a,n}$. Since the input polynomials  $Q^{s}(\bfs
p_p^{s+1}, Y_{n-s,p}, T)$, $W^{s}_{n-s+2}(\bfs p_p^{s+1}, Y_{n-s,p},
T) \klk W^{s}_n(\bfs p_p^{s+1}, Y_{n-s,p}, T)$ coincide with the
specializations of $Q_p^{s},W^{s}_{n-s+2,p} \klk W^{s}_{n,p}$ at
$Y_{1,p}=p_{1,p} \klk Y_{n-s-1,p}=p_{n-s-1,p}$, we see that $q_a,
w_{a,n-s+2} \klk w_{a,n}$ can be obtained by substituting $a$ for
$Y_{n-s,p}$ in $Q^{s}(\bfs p_p^{s+1}, Y_{n-s,p}, T)$,
$W^{s}_{n-s+2}(\bfs p_p^{s+1}, Y_{n-s,p}, T), \dots, W^{s}_n(\bfs
p_p^{s+1}, Y_{n-s,p}, T)$. Then we can compute the corresponding
univariate representation $q_a, v_{a,n-s+2} \klk v_{a,n}$ by means
of the identities $v_{a,j}= (q'_a)^{-1}w_{a,j} \mod q_a$ for
$n-s+2\leq j \leq n$. Let $g(Y_{n-s+1,p}):=F_{s+1,p}\bigl(\bfs
p_p^{s+1}, a, Y_{n-s+1,p}, v_{a,n-s+2}(Y_{n-s+1,p}), \dots,
v_{a,n}(Y_{n-s+1,p})\bigr)$. Now $V^{s+1}_{n-s+1}(\bfs p_p^{s+1}, a)
,$ $\ldots,V^{s+1}_n(\bfs p_p^{s+1}, a)$ can be computed using the
following identities (see, e.g., \cite{DuLe08}):
\begin{align*}
\nonumber Y_{n-s+1,p}-V^{s+1}_{n-s+1}(\bfs p_p^{s+1}, a)&
= \gcd \bigl(g(Y_{n-s+1,p}), q_a(Y_{n-s+1,p})\bigr),\\
\nonumber V^{s+1}_j(\bfs
p_p^{s+1},a)&=v_{a,j}\bigl(V^{s+1}_{n-s+2}(\bfs p_p^{s+1},a)\bigr)
\quad  (n-s+2 \leq j \leq n).
\end{align*}
More precisely, these identities allows us to compute
$V^{s+1}_j(\bfs p_p^{s+1},T) \mod Q_k$ for $n-s+1\leq j \leq n$.
Having done this for $1\leq k \leq \ell$, we can recover
$V^{s+1}_{n-s+1}(\bfs p_p^{s+1},T), \dots,$ $V^{s+1}_n(\bfs
p_p^{s+1},T)$ by the Chinese remainder theorem.

As it is shown in \cite[Section 4]{CaMa06a}, the previous
computations can be rendered into an efficient procedure from which
we obtain the following result (see \cite[Proposition
4.7]{CaMa06a}).
\begin{prop}\label{prop: intersection_step}
There exists a probabilistic algorithm that takes as input
\begin{itemize}
  \item a straight--line program of size at most $L$ which represents the polynomial
  $F_{s+1,p}$;
  \item the dense representation of the polynomials in $\fp[Y_{n-s,p},T]$
  which form the Kronecker representation of $\mathcal{K}_{s,p}$ with primitive element
  $Y_{n-s+1,p}$;
\end{itemize}
and outputs the dense representation of the polynomials in $\fp[T]$
which form the univariate representation of $\mathcal{J}_{s+1,p}$
with primitive element $Y_{n-s,p}$. It uses an expected number of
$\SO\bigr((L+n)\delta_s(d\delta_s+ \log p)\log p\bigl)$
bit operations and returns the right result with probability at
least $1-1/60n$.
\end{prop}
%
%
Taking into account the complexity and probability estimates of
Propositions \ref{prop: lifting_step}  and \ref{prop:
intersection_step} for $1\leq s \leq r-1$, we easily deduce the
following result.
\begin{teo}\label{th: algorithm_Kronecker representation modulo p}
There exists a probabilistic algorithm that takes as input
 \begin{itemize}
    \item a ``lucky'' prime $p$ as in Proposition \ref{prop:
    computation_of_a_lucky_prime};
   \item the points $\bfs \lambda_p \in \fp^{n^2}$ and $\bfs p_p\in \fp^{n-1}$, which are the images of $\bfs \lambda$ and $\bfs
   p$ modulo $p$;
   \item a straight--line program of length at most $L$ which represents the polynomials $F_{1,p}, \dots,
   F_{r,p}$;
 \end{itemize}
and outputs the Kronecker representation of $\mathcal{J}_{r,p}$ with
primitive element $Y_{n-r+1,p}$. It uses an expected number of
$\SO\bigl(r(nL+n^5)\delta(d\delta + \log p)\log p\bigr)$
bit operations and returns the right result with probability at
least $1-1/12$.
\end{teo}
%
%
\subsection{Lifting the integers}\label{subsec: lifting_the_integers}
Let $s$ with $1\leq s\leq r$ and let $p$ be a ``lucky'' prime  as in
Proposition \ref{prop: computation_of_a_lucky_prime}. We have seen
that the Kronecker representation $Q^{s}(\bfs p_p^s,T),$
$W^{s}_{n-s+2}(\bfs p_p^s, T), \dots, w^{s}_n(\bfs p_p^s, T)\in
\fp[T]$ of Proposition \ref{prop: kronecker_representations_mod_p}
is obtained by reducing modulo $p$ the integers of the Kronecker
representation $Q^{s}(\bfs p^s,T), W^{s}_{n-s+2}(\bfs p^s, T),
\dots, w^{s}_n(\bfs p^{(s)},T)$ of Proposition \ref{prop:
kronecker_representation_of_I_s_J_s_and_K_s}. Further, by Lemma
\ref{lemma: jacobian_nonzero} the Jacobian determinant of the
polynomials $F_{1,p}(\bfs p_p^s, Y_{n-s+1,p}\klk Y_{n,p}), \dots,
F_{s,p}(\bfs p_p^s, Y_{n-s+1,p}\klk Y_{n,p})$ with respect to the
variables $Y_{n-s+1,p}\klk Y_{n,p}$ is invertible in
$\fp[Y_{n-s+1,p}\klk Y_{n,p}]/\overline{\mathcal{J}}_{s,p}$. With
these conditions, the following result holds (see \cite[Theorem
2]{GiLeSa01}).
\begin{prop}\label{prop: lifting_the_integers}
Assume that we are given:
\begin{itemize}
\item an upper bound $\eta_s$ for the heights of
$Q^{s}(\bfs p^s,T), W^{s}_{n-s+2}(\bfs p^s, T), \dots,W^{s}_n(\bfs
p^s,T)$;
\item a lucky prime number $p$ as in Proposition \ref{prop: computation_of_a_lucky_prime};
 \item  the polynomials $Q^{s}(\bfs p_p^s,T), W^{S}_{n-s+2}(\bfs p_p^s, T), \dots,
 W^{s}_n(\bfs p_p^{(s)}, T)\in \fp[T]$.
 \end{itemize}
Then $Q^{s}(\bfs p^s,T), W^{s}_{n-s+2}(\bfs p^s, T), \dots,
W^{s}_n(\bfs p^s,T)$ can be computed using
$\SO\bigl((nL+n^4)\delta_s\eta_s\bigr)$
bit operations.
\end{prop}
%
%
\subsection{Computation of a Kronecker representation over the rationals}
Combining the algorithm underlying Theorem \ref{th:
algorithm_Kronecker representation modulo p} with the $p$--adic
lifting procedure of Proposition \ref{prop: lifting_the_integers} we
obtain a probabilistic algorithm for computing a Kronecker
representation of a zero--dimensional fiber of the variety defined
by $F_1, \dots, F_r$.

More precisely, assume that $F_1\klk  F_r$ are given by a
straight--line program $\beta$ of length  at most $L$ with integer
parameters. We first choose at random a point $(\bfs\lambda, \bfs
p)\in {\sf S}^{n^2}\times {\sf T}^{n-1}$ such that ${\sf
R}(\bfs\lambda)\neq 0$ and ${\sf N}_{\bfs\lambda}\neq 0$.  Then we
compute a ``lucky'' prime $p$ as in Proposition \ref{prop:
computation_of_a_lucky_prime}. By reducing the parameters of $\beta$
modulo $p$ we obtain a straight--line program $\beta_p$ of length at
most $L$ which represents the polynomials $F_{1,p}, \dots, F_{r,p}$.
Then, by means of the algorithm underlying Theorem \ref{th:
algorithm_Kronecker representation modulo p}, we compute the
Kronecker representation $Q^{r}_p, W^{r}_{1,p} \klk  W^{r}_{n,p}$ of
the lifting fiber $V_{\bfs p_p^r}$ with primitive element
$Y_{n-r+1,p}$. Finally, applying the algorithm  underlying
Proposition \ref{prop: lifting_the_integers} we lift  these
polynomials to the Kronecker representation $Q^{r}, W^{r}_1 \klk
W^{r}_n$ of the lifting fiber $V_{\bfs p^r}$ with primitive element
$Y_{n-r+1}$. We have the following result.
\begin{teo}\label{th: bit_complexity_kronecker_solver}
There exists a probabilistic algorithm that takes as input a
straight--line program $\beta$ of length at most $L$ which
represents the polynomials $F_1 \klk F_r$, and outputs a Kronecker
representation of a zero--dimensional fiber of $\mathcal{V}(F_1,
\dots, F_r)$ with probability at least $\frac{77}{144}$. If $h$ is
an upper bound for the bit length of the coefficients of $F_1\klk
F_r$ and the parameters in $\beta$, then the bit complexity of the
algorithm is in
$$\mathcal{O}^{\sim}\big(r(nL+n^5)\delta(d\delta+nd^rh)\big).$$
\end{teo}
\begin{proof}
Let $\mathcal{C}_p$ denote the bit complexity of computing a
``lucky'' prime $p$ and $\eta$ an upper bound for heights of the
integers in the output. Combining the complexity estimates in
Theorem \ref{th: algorithm_Kronecker representation modulo p} and
Proposition \ref{prop: lifting_the_integers}, the bit complexity of
the algorithm above is in
\[
\SO\Bigl(r(nL+n^5)\delta\bigr((d\delta+\log p)\log p+ \eta\bigl)
\Bigr) + \mathcal{C}_{p}.
\]
By Proposition \ref{eta_s_height_estimate} we can take $\eta \in
\SO\bigl(nd^{r-1}(h+rd)\bigr)$. Then, taking into account the
estimate for $\mathcal{C}_p$  in Proposition \ref{prop:
computation_of_a_lucky_prime}, we obtain the complexity estimate of
the theorem.

Finally, taking into account Lemma \ref{lema: lambda and
p_probability} and the estimates for the probability of success of
Proposition \ref{prop: computation_of_a_lucky_prime} and Theorem
\ref{th: algorithm_Kronecker representation modulo p}, the theorem
follows.
\end{proof}
%
%
\appendix
%
%
\section{Height estimates}
\label{section: appendix} In this appendix we obtain estimates for
the height of the integer $\mathfrak{N}$ of Theorem
\ref{th:_simultaneous_noether_normalization} and the integers
occurring in the output of the algorithm underlying  Theorem
\ref{th: bit_complexity_kronecker_solver}, namely the polynomials in
Proposition \ref{prop: kronecker_representation_of_I_s_J_s_and_K_s}
which form the Kronecker representation of $\mathcal{J}_r$. For this
purpose, we shall rely on the arithmetic Nullstellens\"{a}tze of
\cite{KrPaSo01}. We start recalling the notions of height of
polynomials and varieties and basic facts about these, and then
proceed to obtain the estimates.
%
%
\subsection{Height of polynomials and varieties}
We define the \emph{height} of a nonzero integer $a$ as $h(a):=\log
|r|$, where $\log$ stands for the logarithm to the base $2$.
Further, we define $h(0):=0$. It follows that the height of $a$
bounds from above the bit length of $a$. The height $h(F)$ of a
polynomial $F\in \Z[\bfs X]$ is defined as the maximum of the
heights of its coefficients. More generally, if $F\in
\mathbb{Q}[\bfs X]\setminus \{0\}$ and $a\in \mathbb{N}$ is a
minimal common denominator of all the coefficients of $F$, then we
define $h(F):= \max\{h(aF), h(a)\}$.

Let $V\subset  \mathbb{A}^n(\overline{\mathbb{Q}})$ be an
equidimensional $\mathbb{Q}$--variety of dimension $n-s$, with
$1\leq s \leq n$, and let $h(V)$ be the Faltings height of its
projective closure $\overline{V}\subset
\mathbb{P}^n(\overline{\mathbb{Q}})$ (see \cite{Faltings91}). We
have the following identity:
\begin{equation}\label{eq: app: faltings_height}
h(V)=m(F_{\scriptscriptstyle V}; S_{n+1}^{n-s+1}) + \sum_{p}\log
|F_{\scriptscriptstyle V}|_p +
(n-s+1)\Biggl(\sum_{i=1}^n\frac{1}{2i}\Biggr)\deg V,
\end{equation}
where $F_{\scriptscriptstyle V}$ is any Chow form of $V$,
$m(F_{\scriptscriptstyle V}; S_{n+1}^{n-s+1})$ is the
$S_{n+1}^{n-s+1}$--\emph{Mahler measure} of $F_{\scriptscriptstyle
V}$ and $|F_{\scriptscriptstyle V}|_p$ is the  $p$--adic absolute
value over $\mathbb{Q}$ for all rational primes $p$ (see, e.g.,
\cite[Section 1.2.4]{KrPaSo01}). Since $F_{\scriptscriptstyle V}$ is
uniquely determined up to nonzero multiples in $\mathbb{Q}$, we may
assume that $F_{\scriptscriptstyle V}$ is a primitive polynomial in
$\Z[\bfs\Lambda^h_1, \dots, \bfs\Lambda^h_{n-s+1}]$, in which case
$\log |F_{\scriptscriptstyle V}|_p=0$ for every prime $p$ and the
sum $\sum_{p}\log |F_{\scriptscriptstyle V}|_p$ in \eqref{eq: app:
faltings_height} disappears. On the other hand, by \cite[Lemma
1.1]{KrPaSo01} we have
\begin{equation}\label{eq: app: mahler_measure_versus_height}
|m(F_{\scriptscriptstyle V})-h(F_{\scriptscriptstyle V})|\leq
(n-s+1)\log(n+2)\deg V,
\end{equation}
where $m(F_{\scriptscriptstyle V})$ denotes the \emph{Mahler
measure} of $F_{\scriptscriptstyle V}$. The {Mahler measure} and the
$S_{n+1}^{n-s+1}$--{Mahler measure} of $F_{\scriptscriptstyle V}$
are related by
\begin{equation}\label{eq: app: Mahler_versus_S_n_n-s+1_mahler measure}
    0\leq m(F_{\scriptscriptstyle V})-m(F_{\scriptscriptstyle V};
    S_{n+1}^{n-s+1})\leq (n-s+1)\deg(V)\sum_{i=1}^n\frac{1}{2i}
\end{equation}
(see, e.g., \cite[(1.2)]{KrPaSo01}). Combining \eqref{eq: app:
faltings_height}, \eqref{eq: app: mahler_measure_versus_height} and
\eqref{eq: app: Mahler_versus_S_n_n-s+1_mahler measure} gives
\begin{equation*}\label{eq: app: height_chow_form_versus_height_variety}
    h(F_{\scriptscriptstyle V})\leq h(V) + (n-s+1)\log(n+2)\deg V.
\end{equation*}
Further, the canonical height $\widehat{h}(V)$ of $V$ is defined by
$\widehat{h}(V):=\widehat{h}(\overline{V})$, where $\widehat{h}(V)$
is the \emph{canonical height} of
$\overline{V}\subset\mathbb{P}^n(\overline{\mathbb{Q}})$ defined as
in \cite{DaKrSo13}. The Faltings and the canonical height of $V$ are
related by the inequality
$$|\widehat{h}(V)-h(V)|\leq \frac{7}{2}\log(n+1)\deg V$$
(see, e.g., \cite[Proposition 2.39 (5)]{DaKrSo13}). As a
consequence, we have
\begin{equation}\label{eq: app: height_chow_form_versus_canonical_height}
    h(F_{\scriptscriptstyle V})\leq \widehat{h}(V) + \frac{9}{2}(n-s+1)\log(n+2)\deg V.
\end{equation}
%
%
\subsection{Estimates for Chow forms, discriminants and Kronecker representations}
From now on, we return to the setting of Sections \ref{section:
modular_simultaneous_noether_normalization} and \ref{section:
computation_of_a_kronecker_representation}, namely we consider
polynomials $F_1\klk F_r\in \Z[\bfs X]$ which form a regular
sequence, denote by $\mathcal{V}_s$ the affine equidimensional
subvariety of $\mathbb{A}^{n}$ defined by $F_1\klk F_s$ and by
$\delta_s$ its degree for $1 \leq s \leq r$. Let $d_j:=\deg(F_j)$
and $h_j:=h(F_j)$ for $1 \leq j \leq r$, and denote
$$\delta:=\max_{1\leq s \leq r} \delta_s,\quad d:=\max_{1\leq j \leq
r}d_j, \quad h:=\max_{1\leq j \leq r}h_j.$$
Let $\widehat{h}_s:=\widehat{h}(\mathcal{V}_s)$ for $1\leq s \leq r$
and $\widehat{h}:=\max_{1\leq s \leq r}\widehat{h}_s$. By
\cite[Corollary 2.62]{DaKrSo13}, taking into account \cite[Lemma
2.30 (1)]{DaKrSo13}, we have
\begin{equation}\label{eq: app: widehat_v_s_height}
\widehat{h}(\mathcal{V}_s)\leq \sum_{\ell=1}^s
h_{\ell}\Biggl(\prod_{\scriptscriptstyle j=1, \scriptscriptstyle j
\neq \ell}^{s}d_j\Biggr) +
s\Biggl(\prod_{j=1}^{s}d_j\Biggr)\log(n+2)\quad (1\leq s \leq r).
\end{equation}

%
Let $\mu$ and $\varepsilon$ be fixed real numbers with $0< \mu,
\varepsilon < 1$. Let $\textsf{a}:=\lfloor {\sf D/(1-\mu)}\rfloor$
and $\textsf{b}:=\lfloor {\sf D/(1-\varepsilon)}\rfloor$, where $D$
is defined in \eqref{deg_R_estimate}. Recall that $D$ is an upper
bound for the degree of the polynomials ${\sf R}$ and ${\sf
N}_{\bfs\lambda}$ of \eqref{R_definition} and
\eqref{N_lambda_definition}. Since $D\in \mathcal{O}(rnd^{3r})$ and
$h({\sf a}), h({\sf b}) \in \mathcal{O} (\log D)$, we have the
following remark.
\begin{remark}  \label{rem: app: max h_a_h_b_estimate}
$h({\sf a}), h({\sf b})\in \SO (r\log d +\log n)$.
\end{remark}

Set ${\sf S}:=\{0, \dots, \textsf{a}\}$ and ${\sf T}:=\{0, \dots,
\textsf{b}\}$. Further, let $\bfs\lambda:=(\lambda_{ij})_{1\leq i
\leq n, 1\leq j\leq n} \in {\sf S}^{n^2}$ and $\bfs p:=(p_1, \dots,
p_{n-1})\in {\sf T}^{n-1}$ be such that ${\sf R}(\bfs\lambda)\neq 0$
and ${\sf N}_{\bfs\lambda}(\bfs p)\neq 0$. By Lemma \ref{lemma:
zippel_schwartz}, for a random choice of $\bfs\lambda$ and $\bfs p$
such a condition holds with probability at least $\mu\varepsilon$.

Write $\bfs\lambda^s:=(\lambda_{ij})_{1\leq i \leq n-s+1, 1\leq
j\leq n}$ and $\bfs p^s:=(p_1\klk p_{n-s})$ for $1\leq s \leq r$.
Denote $h(\bfs\lambda^s):=\max_{1\leq i \leq n-s+1, 1\leq j\leq n}
h(\lambda_{ij})$ and $h(\bfs p^s):=\max_{1\leq i \leq n-s} h(p_i)$.
Finally, let $\bfs\lambda_i:=(\lambda_{i1}, \dots, \lambda_{in})$
and $Y_i=\bfs\lambda_i \cdot \bfs X$ for $1 \leq i \leq n$. In the
sequel, assuming that $n\geq 2$ and $d\geq 2$, we aim to estimate
the height of the integer
\begin{equation}\label{eq: app: frak_N_def}
\mathfrak{N}:=\det(\bfs\lambda){\sf N}_{\bfs\lambda}(\bfs
p)=\det(\bfs\lambda){\sf M}_r\bigl(\bfs\lambda^r,\bfs
p^r\bigr)\prod_{s=1}^{r-1}{\sf M}_s\bigl(\bfs\lambda^s,\bfs
p^s\bigr)L_{\bfs\lambda^s}\bigl(\bfs p^{s+1}\bigr).
\end{equation}

We start with an estimate for the degree and height of a primitive
Chow form of $\mathcal{V}_s$ and related polynomials.
\begin{lema}
For $1\leq s \leq r$, we have
\begin{gather}
    h(P_s)\in \SO \bigl(nd^{s-1}(h+d)\bigr), \label{eq: app: height_P_s}\\
    \deg P_s(\bfs \Lambda^s, \bfs \Lambda^s \bfs X )\in \SO (nd^s), \quad
    h\bigl(P_s(\bfs \Lambda^s, \bfs \Lambda^s \bfs X )\bigr)\in \SO \bigl(nd^{s-1}(h+d)\bigr). \label{eq: app: deg_and_height_P_s_Lambda}
\end{gather}
\end{lema}
\begin{proof}
\eqref{eq: app: height_chow_form_versus_canonical_height} and
\eqref{eq: app: widehat_v_s_height}, combined with the B\'ezout
inequality \eqref{bezout_inequality}, yields \eqref{eq: app:
height_P_s}. The degree estimate in \eqref{eq: app:
deg_and_height_P_s_Lambda} is clear. Next, observe that $P_s$ is an
element of $\Z[\bfs\Lambda^s\!,\!Z_1\klk\! Z_{n-s+1}]$  of total
degree $(n-s+1)\delta_s$ and $\Lambda_{ij}$ $(1\leq i \leq n-s+1,
1\leq j \leq n)$, $\bfs\Lambda_i\cdot\bfs X$ $(1\leq i \leq n-s+1)$
are elements of $\Z[\bfs\Lambda^s,\bfs X]$ having total degrees at
most $2$ and heights equal to $0$. Therefore, from \cite[Lemma
2.37(3)]{DaKrSo13} we deduce that
\[
h \bigl(P_s(\bfs\Lambda^s, \bfs\Lambda^s \bfs X)\bigr)\!\leq\!
h(P_s) + (n-s+1)\delta_s\Bigl(\log\bigl((n-s+1)(n+1)+1\bigl)+
2\log\bigl((n-s+2)n+1\bigr)\!\Bigr).
\]
This, together with \eqref{eq: app: height_P_s}, readily implies the
height estimate in \eqref{eq: app: deg_and_height_P_s_Lambda}.
\end{proof}

Next we estimate the degree and height of the discriminant $\rho_s$
and the polynomial $\rho_{\bfs\lambda^s}$ of Section \ref{subsec:
lifting_fibers_not_meeting_a_discriminant}.
For this purpose, we use the following result.
\begin{lema}
Let $U_1, \dots, U_{k+1}$ be indeterminates over $\Z$ and $F,G\in
\Z[U_1, \dots, U_{k+1}]$ nonzero polynomials with
$l:=\deg_{U_{k+1}}F$ and $m:=\deg_{U_{k+1}}G$. Then
  \[
  h\bigl(\mathrm{Res}_{U_{k+1}}(F,G)\bigr)\leq mh(F)+lh(G)
  + \log(k+1)\bigl((m-1)\deg F+ l\deg G\bigr) + \log\bigl((l+m)!\bigr).
  \]
\end{lema}
\begin{proof}
Write $F=\sum_{i=0}^lF_iU_{k+1}^i$ and $G=\sum_{j=0}^mG_jU_{k+1}^j$,
where $F_i, G_j \in \Z[U_1, \dots, U_k]$. The determinant
$\mathrm{Res}_{U_{k+1}}(F,G)$ is a sum of $(l+m)!$ terms, each of
which is a product of the form $\pm F_{i_1}\cdots
F_{i_m}G_{j_1}\cdots G_{j_l}$. By \cite[Lemma 2.37(2)]{DaKrSo13},
each term has height at most $mh(F)+lh(G) + \log(k+1)\bigl((m-1)\deg
F+l\deg G\bigr)$. Then \cite[Lemma 2.37(1)]{DaKrSo13} completes the
proof of the lemma.
\end{proof}

Now we are able to estimate the degree and height of $\rho_s$ and
$\rho_{\bfs\lambda^s}$.
\begin{lema}\label{lema:_bound_height_resultant}
For $1\leq s \leq r$, we have
\begin{gather*}
 \deg \rho_s \in \mathcal{O}(nd^{2s}),
 \quad h (\rho_s)\in \SO \bigl(nd^{2s-1}(h+d)\bigr),  \\
 \deg \rho_{\bfs\lambda^s} \in \mathcal{O}(nd^{2s}),
 \quad h(\rho_{\bfs\lambda^s}) \in \SO \bigl(nd^{2s-1}(h+d)\bigr).
\end{gather*}
\end{lema}
\begin{proof}
Since $\rho_{\bfs\lambda^s}:=\rho_s(\bfs \lambda^s, \bfs
\lambda^{s+1} \bfs X)$, we have $\deg \rho_{\bfs\lambda^s} \leq \deg
\rho_s \leq (n-s+2)\delta_s^2$, which proves the degree estimates.
Next, as $\rho_s:=\mathrm{Res}_{Z_{n-s+1}}\left(P_s, \frac{\partial
P_s}{\partial Z_{n-s+1}}\right)$, Lemma
\ref{lema:_bound_height_resultant} implies
$$
h(\rho_s)\leq \delta_s\big(2h(P_s)+\log\delta_s\big)+
2\delta_s^2\log\bigl((n-s+1)(n+1)\bigr) +
\log\bigl((2\delta_s)!\bigr).$$
This and \eqref{eq: app: height_P_s} prove the estimate for
$h(\rho_s)$. Further, since $h(\bfs\lambda^s)\leq h({\sf a})$  for
all $s$, from \cite[Lemma 2.37 (3)]{DaKrSo13} we deduce that
\[
h(\rho_{\bfs\lambda^s})  \leq   h(\rho_s) + \deg\rho_s\Bigl(h({\sf
a}) + \log\bigl((n-s+1)(n+1)\bigr) + \log(n+1)\Bigr).
\]
Combining this,  Remark \ref{rem: app: max h_a_h_b_estimate} and the
estimate for $h(\rho_s)$ yields the one for $h
(\rho_{\bfs\lambda^s})$.
\end{proof}

We end this section with an estimate of the height of the Kronecker
representations of the fibers of each recursive step of our main
algorithm.
\begin{prop}\label{eta_s_height_estimate}
Let $\eta_s$ be the maximum of the heights of the polynomials
$Q^s(\bfs p^s,T)$, $W^s_{n-s+2}(\bfs p^s,T) \klk W^s_n(\bfs p^s,T)$
of Proposition  \ref{prop:
kronecker_representation_of_I_s_J_s_and_K_s}. Then $\eta_s\in \SO
\bigl(nd^{s-1}(h+rd)\bigr)$.
\end{prop}
\begin{proof}
 Note that
\begin{align}
Q^s(\bfs p^s,T)&=\frac{P_s(\bfs\lambda^s,\bfs p^s,
T)}{A_s(\bfs\lambda_1, \dots,\bfs\lambda_{n-s})},
\label{kronecker_representation_height_estimate_section_equation_1}\\
W^s_j(\bfs p^s,T)&=
-\sum_{k=1}^{n}\frac{\lambda_{jk}}{A_s(\bfs\lambda_1, \dots,
\bfs\lambda_{n-s})}\frac{\partial P_s(\bfs\lambda^s,\bfs
p^s,T)}{\partial \Lambda_{n-s+1,k}} \quad (n-s+2\leq j \leq n).
\label{kronecker_representation_height_estimate_section_equation_2}
\end{align}
Since $h(\bfs\lambda^s)\leq h({\sf a})$ and $h(\bfs p^s)\leq h({\sf
b})$, by \cite[Lemma 2.37 (3)]{DaKrSo13} we deduce that
\begin{align*}
h\bigl(P_s(\bfs\lambda^s\!,\bfs p^s\!, T)\bigr)\!&\leq h(P_s)\!+\!
(n-s+1)\delta_s\Bigl(\max\{h({\sf a}),h({\sf b})\}\!+
\!\log\bigl((n-s+\!1)(n+\!1)\!+\!1\bigr)\!+\!1\!\Bigr)\\
&\leq h(P_s) + (n-s+1)\delta_s\bigl(\max\{h({\sf a}),h({\sf b})\} +
\log(4n^2)\bigr).
\end{align*}
Further, as $h\bigl(\frac{\partial P_s}{\partial
\Lambda_{n-s+1,k}}\bigr)\leq h(P_s) + \log\delta_s$, a similar
argument shows that
$$
h\left(\frac{\partial P_s\left(\bfs\lambda^s,\bfs
p^s,T\right)}{\partial \Lambda_{n-s+1,k}} \right)\leq  h(P_s)
+\log\delta_s+ (n-s+1)\delta_s\bigl(\max\{h({\sf a}),h({\sf b})\} +
\log(4n^2)\bigr).
$$
Therefore, by \cite[Lemma 2.37(1)]{DaKrSo13} we obtain
\begin{align}\label{kronecker_representation_height_estimate}
h\left( \sum_{k=1}^{n} \lambda_{jk} \frac{\partial
P_s\left(\bfs\lambda^s,\bfs p^s,T\right)}{\partial
\Lambda_{n-s+1,k}} \right)
& \leq  h(P_s) +\log\delta_s + h({\sf a}) + \log n\\
\nonumber & + (n-s+1)\delta_s\bigl(\max\{h({\sf a}),h({\sf b})\} +
\log(4n^2)\bigr)
\end{align}
for $n-s+2\leq j \leq n$. Similarly we  deduce that
$$
h\bigl(A_s(\bfs\lambda_1, \dots,\bfs\lambda_{n-s})\bigr)\leq
h(P_s)+(n-s)\delta_s\Bigl(h({\sf a}) +
\log\bigl((n-s+1)n+1\bigr)\Bigr).
$$

By
\eqref{kronecker_representation_height_estimate_section_equation_1},
\eqref{kronecker_representation_height_estimate_section_equation_2}
and the previous estimates we see that $\eta_s$ is bounded above by
the right--hand side of
(\ref{kronecker_representation_height_estimate}). The proposition
then follows by \eqref{eq: app: height_P_s} and Remark \ref{rem:
app: max h_a_h_b_estimate}.
\end{proof}
%
%
\subsection{Estimates for unmixedness
and generic smoothness}
In this section we estimate the height of integers $\alpha_s$ and
$\gamma_s$ as in \eqref{alpha_s_definition} and
\eqref{gamma_s_definition}, whose nonvanishing modulo $p$ implies
that the corresponding modular reduction is unmixed and generically
smooth, and yields new variables in Noether position (Theorem
\ref{th: conditions_mod_p_summary}).

We start with $\alpha_s$. Taking into account that
$\widehat{h}\bigl(\mathbb{A}^{(n-s+2)n}\bigr)=0$ and
$\deg\bigl(\mathbb{A}^{(n-s+2)n}\bigr)=1$, from \cite[Theorem
2]{DaKrSo13} it follows that there exists $\alpha_s\in
\Z\setminus\{0\}$ as in \eqref{alpha_s_definition} with
\begin{equation*}
h(\alpha_s)\leq 3h\bigl(P_s(\bfs\Lambda, \bfs\Lambda \bfs
X)\bigr)\prod_{j=1}^sd_j+ 2\deg\bigl(P_s(\bfs\Lambda^s,
\bfs\Lambda^s \bfs X)\bigr)\prod_{j=1}^sd_j\Biggl( h\sum_{\ell=1}^s
\frac{1}{d_{\ell}}+c(n)\Biggr),
\end{equation*}
where $c(n)\in \mathcal{O}^{\sim}(n)$. Combining this with
\eqref{eq: app: deg_and_height_P_s_Lambda} we deduce the following
result.
\begin{lema}\label{lema: app: alpha_s_height_estimate}
We have $h(\alpha_s)   \in  \SO \bigl(nd^{2s-1}(h+nd)\bigr)$.
\end{lema}

Next we consider $\gamma_s$. Let $J_s$ be the Jacobian determinant
of $Y_1, \dots, Y_{n-s},F_1, \dots, F_s$ with respect to the
variables $X_1, \dots, X_n$.
\begin{lema}\label{lema: app: jacobian_degree_and_height_estimate}
The following assertions hold:
\begin{itemize}
  \item $\deg J_s \leq s(d-1)$;
  \item $h(J_s)  \leq s(\log d + h)
  + (n-s)h(\textsf{a}) + s\,d\log(n+1)+ \log (n!)$.
\end{itemize}
\end{lema}
\begin{proof}
The assertion on the degree of $J_s$ is clear. To prove the second
assertion, we observe that $J_s$ is a sum of $n!$ terms of the form
$\pm\partial F_1/\partial X_{j_1}\cdots \partial F_s/\partial
X_{j_s}\lambda_{1,l_1}\cdots\lambda_{n-s,l_{n-s}}$. Since
$h(\lambda_{ij})\leq h(\textsf{a})$ and $h(\partial F_i/\partial
X_j)\leq h(F_i) +\log(d_i)$, by \cite[Lemma 2.37(2)]{DaKrSo13} we
deduce that each term has height at most $ s(h+\log d) +
(n-s)h(\textsf{a}) + \log(n+1)\bigl((s-1)(d-1)\bigr). $ The estimate
for the height of $J_s$ follows by \cite[Lemma 2.37(1)]{DaKrSo13}.
\end{proof}

Let $d_j\!:=\!1$ and $h_j\!:=\!h(Y_{j-s}-p_{j-s})$ for $s+1\leq j
\leq n$, $d_{n+1}:=\deg J_s$ and $h_{n+1}\!:=h(J_s)$. By
\cite[Theorem 1]{DaKrSo13}, there exist $\gamma_s\in \Z\setminus
\{0\}$ and $G_1, \dots, G_{n+1}\in \Z[\bfs X]$ as in
\eqref{gamma_s_definition} with
\begin{align*}
h(\gamma_s)  &\leq \sum_{\ell=1}^{n+1}\Biggl(\prod_{j\neq
\ell}d_j\Biggr)h_{\ell} + (4n+8)\log(n+3)\prod_{j=1}^{n+1}d_j\\
&\leq \deg
J_s\Biggl(\prod_{j=1}^sd_j\Biggr)\Biggl(\sum_{\ell=1}^s\frac{h_{\ell}}{d_{\ell}}
+ \sum_{\ell=1}^{n-s}h(Y_{\ell}-p_\ell) +(4n+8)\log(n+3) \Biggr) +
h(J_s)\prod_{j=1}^sd_j.
\end{align*}
Since $h(Y_{\ell})\leq h({\sf a})$ and $h(p_{\ell})\leq h({\sf b})$
for all $\ell$, we obtain
$$
h(\gamma_s)\leq \deg J_s\,d^{s-1}sh + \deg
J_s\,d^s\bigl((n-s)\max\{h({\sf a}),h({\sf b})\} +
(4n+8)\log(n+3)\bigr) + h(J_s)d^s.
$$
Combining this with Remark \ref{rem: app: max h_a_h_b_estimate} and
Lemma \ref{lema: app: jacobian_degree_and_height_estimate}, we
deduce the following result.
\begin{lema}\label{eq: app: gamma_s_height} We have
$h(\gamma_s) \in \SO \bigl(d^s(h+rnd)\bigr)$.
\end{lema}
%
%
\subsection{Estimates for smooth fibers}
In this section we estimate the height of the integers considered in
Section \ref{subsec: lifting_fibers_not_meeting_a_discriminant},
namely ${\sf M}_s(\bfs\lambda^s, \bfs p^s)$ and ${\sf
L}_{\bfs\lambda^s}(\bfs p^{s+1})$, where ${\sf M}_s$ is the
polynomial of \eqref{def:polynomial_M_s} and ${\sf
L}_{\bfs\lambda^s}$ is the polynomial of \eqref{def:L_lambda_s}.
Combining these estimates we shall be able to estimate the height of
the integer $\mathfrak{N}$ of \eqref{eq: app: frak_N_def}, which
comprises all the unlucky primes $p$.

We start estimating the height of ${\sf M}_s(\bfs\lambda^s, \bfs
p^s)$.
\begin{lema}\label{lema: app: height_M_s_lambda_s_p_s} For $1\leq s \leq r$,
we have $ h\bigl({\sf M}_s(\bfs\lambda^s, \bfs p^s)\bigr)\in \SO
\bigl(nd^{2s-1}(h+nd)\bigr). $
\end{lema}
\begin{proof}
By \cite[Lemma 2.37 (3)]{DaKrSo13}, we have
\begin{equation}\label{eq: app: M_s_lambda_s_p_s_height_estimate}
h\bigl({\sf M}_s(\bfs\lambda^s,\bfs p^s)\bigr)\leq h({\sf M}_s) +
\deg({\sf M}_s)\Bigl(\max \{h(\bfs\lambda^s), h(\bfs p^s)\} +
\log\bigl((n-s+1)(n+1)+1\bigr)\Bigr).
\end{equation}
Recall that ${\sf M}_s:=\alpha_s\gamma_sA_s\rho_s$ and, by
definition, $\deg A_s \leq (n-s)\delta_s$ and $h(A_s)\leq h(P_s)$.
As a consequence, from \cite[Lemma 2.37 (2)]{DaKrSo13} we deduce
that
$$
h({\sf M}_s)  \leq h(\alpha_s)+ h(\gamma_s) + h(P_s) + h(\rho_s) +
(n-s)\delta_s\log\bigl((n-s+1)(n+1)+1\bigr).
$$
Combining this with \eqref{eq: app: height_P_s} and Lemmas
\ref{lema:_bound_height_resultant}, \ref{lema: app:
alpha_s_height_estimate} and \ref{eq: app: gamma_s_height} we obtain
  $$
  h({\sf M}_s)\in \SO \bigl(nd^{2s-1}(h+nd)\bigr).
  $$
On the other hand, since $h(\bfs\lambda^s)\leq h({\sf a})$ and $h(
\bfs p^s)\leq h({\sf b})$ for all $s$,  by Remark \ref{rem: app: max
h_a_h_b_estimate} we have $\max \{h(\bfs\lambda^s), h(\bfs p^s)\}
\in \SO (r\log d +\log n)$. Further, $\deg {\sf M}_s\in
\mathcal{O}(nd^{2s})$ by \eqref{eq: deg_M_s}.  Combining all these
estimates with \eqref{eq: app: M_s_lambda_s_p_s_height_estimate},
the lemma follows.
\end{proof}

Next we estimate ${\sf L}_{\bfs\lambda^s}(\bfs p^{s+1})$. As this
integer is expressed in terms of the integers $\mu_{\bfs\lambda^s}$
of \eqref{empty_case_equation} and $\beta_{\bfs\lambda^s}$ of
\eqref{nullstellensatz_implicit_equation} and the polynomial
$B_{\bfs\lambda^s}\in \mathbb{Z}[Z_1,\dots,Z_{n-s-1}]\setminus
\{0\}$ of \eqref{nonempty_case_implicit_equation}, we start with an
estimate for $\mu_{\bfs\lambda^s}$ and $B_{\bfs\lambda^s}$.
\begin{prop}
Let $1\leq s \leq r-1$ and assume that
$\mathcal{W}_{\bfs\lambda^s}=\emptyset$. Then there exists
$\mu_{\bfs\lambda^s}\in \Z\setminus \{0\}$ as in
\eqref{empty_case_equation} with
\begin{equation}\label{eq: app: height_mu_s_lambda_s}
h(\mu_{\bfs\lambda^s})\in \SO \bigl(n^2d^{3s}(h+d)\bigr).
\end{equation}
On the other hand, if $\mathcal{W}_{\bfs\lambda^s}\neq \emptyset$,
then there exists $B_{\bfs\lambda^s}\in
\mathbb{Z}[Z_1,\dots,Z_{n-s-1}]\setminus \{0\}$ as in
(\ref{nonempty_case_implicit_equation}) with
 \begin{equation}\label{eq: app: deg_and_height_B_s_lambda_s}
    \deg B_{\bfs\lambda^s} \in \mathcal{O}(nd^{3s+1}), \quad
    h(B_{\bfs\lambda^s}) \in \SO \bigl(nd^{3s}(h+rnd)\bigr)
 \end{equation}
\end{prop}
\begin{proof} 
Assume that $\mathcal{W}_{\bfs\lambda^s}\!:=\!\mathcal{V}_{s+1}\cap
\{\rho_s(\bfs\lambda^s\!,\!\bfs\lambda^{s+1} \bfs X)=0\}=\emptyset$
and let $\rho_{\bfs\lambda^s}:=\rho_s( \bfs\lambda^s,\!
\bfs\lambda^{s+1} \bfs X)$. By \cite[Theorem 1]{DaKrSo13} there
exists $\mu_{\bfs\lambda^s}\in \Z\setminus \{0\}$ as in
\eqref{empty_case_equation} with
\begin{align*}
h(\mu_{\bfs\lambda^s})\leq&
h(\rho_{\bfs\lambda^s})\prod_{j=1}^{s+1}d_j +
\deg(\rho_{\bfs\lambda^s})\prod_{j=1}^{s+1}d_j\Biggl(
\sum_{\ell=1}^{s+1}\frac{h_{\ell}}{d_{\ell}}+(4n+8)\log(n+3)\Biggr)\\
\le &d^{s+1}\bigl(h(\rho_{\bfs\lambda^s}) +
\deg(\rho_{\bfs\lambda^s})(4n+8)\log(n+3)\bigr) +
(s+1)\deg(\rho_{\bfs\lambda^s})d^sh
\end{align*}
Combining this with Lemma \ref{lema:_bound_height_resultant} proves
the first assertion of the lemma.

On the other hand, assume that $\mathcal{W}_{\bfs\lambda^s}\neq
\emptyset$. By hypothesis ${\sf R}_s( \bfs\lambda^s)\neq 0$, and
hence Lemma \ref{lemma: pure_dimesion_r_subvariety} proves that
$\mathcal{W}_{\bfs\lambda^s}$ is equidimensional of dimension
$n-s-2$. By \cite[Corollary 3.23]{DaKrSo13} there exists a
polynomial $B_{\bfs\lambda^s}\in
\mathbb{Z}[Z_1,\dots,Z_{n-s-1}]\setminus \{0\}$ as in
(\ref{nonempty_case_implicit_equation}) with
\begin{align}
\deg (B_{\bfs\lambda^s})& \leq \deg\mathcal{W}_{\bfs\lambda^s}, \label{B_s_lambda_s_deg_estimate}\\
 h(B_{\bfs\lambda^s})&  \leq \widehat{h}(\mathcal{W}_{\bfs\lambda^s})+\deg\mathcal{W}_{\bfs\lambda^s}
 \left(\sum_{\ell=1}^{n-s-1}h(Y_{\ell}) + (n-s)\log (2n + 8)\right). \label{B_s_lambda_s_height_estimate}
\end{align}

Next we obtain estimates for $\deg \mathcal{W}_{\bfs\lambda^s}$ and
$h(\mathcal{W}_{\bfs\lambda^s})$ in terms of the degrees and heights
of $\mathcal{V}_s$ and $\mathcal{V}_{s+1}$. For this purpose, let
$\overline{\mathcal{V}}_{s+1}$ and
$\overline{\mathcal{W}}_{\bfs\lambda^s}$ denote the projective
closures of $\mathcal{V}_{s+1}$ and $\mathcal{W}_{\bfs\lambda^s}$
respectively, via the canonical inclusion $\mathbb{A}^n
\hookrightarrow \mathbb{P}^n$. Let $\rho_{\bfs\lambda^s}^h$ be the
homogenization of $\rho_{\bfs\lambda^s}$. Lemma \ref{lemma:
pure_dimesion_r_subvariety} implies that
$\{\rho_{\bfs\lambda^s}^h=0\}$ of $\mathbb{P}^n$ cuts
$\overline{\mathcal{V}}_{s+1}$ properly. By \cite[Corollary
2.62]{DaKrSo13} we conclude that
\[
\widehat{h}\bigl(\overline{\mathcal{V}}_{\!s+1}\cap
\{\rho_{\bfs\lambda^s}^h\!=\!0\}\bigr)\leq \deg\rho_{\bfs
\lambda^s}\,\widehat{h}\bigl(\overline{\mathcal{V}}_{\!s+1}\bigr)
+\deg\overline{\mathcal{V}}_{\!s+1}\,h(\rho_{\bfs\lambda^s}^h) +
\deg\overline{\mathcal{V}}_{\!s+1}\deg\rho_{\bfs\lambda^s}^h\,\log(n+2).
\]
Since $\overline{\mathcal{V}}_{s+1}\cap
\{\rho_{\bfs\lambda^s}^h=0\}$ is equidimensional of dimension
$n-s-2$ and contains every component of
$\overline{\mathcal{W}}_{\bfs\lambda^s}$, we see that
$\widehat{h}\bigl(\overline{\mathcal{W}}_{\bfs\lambda^s}\bigr)\leq
\widehat{h}\bigl(\overline{\mathcal{V}}_{s+1}\cap
\{\rho_{\bfs\lambda^s}^h=0\}\bigr)$. Recalling that
$\widehat{h}(\mathcal{V}_{s+1})=\widehat{h}(\overline{\mathcal{V}}_{s+1})$
and $\deg \mathcal{V}_{s+1}=\deg\overline{\mathcal{V}}_{s+1}$, and
taking into account that
$\deg\rho_{\bfs\lambda^s}^h=\deg\rho_{\bfs\lambda^s}$ and
$h(\rho_{\bfs\lambda^s}^h)=h(\rho_{\bfs\lambda^s})$, we obtain
 \begin{gather*}
 \deg \mathcal{W}_{\bfs\lambda^s}  \leq  \deg \mathcal{V}_{s+1}\deg\rho_{\bfs\lambda^s},\\
 \widehat{h}(\mathcal{W}_{\bfs\lambda^s}) \leq
 \deg\rho_{\bfs\lambda^s}\,\widehat{h}(\mathcal{V}_{s+1})+\deg\mathcal{V}_{s+1}\,
 h(\rho_{\bfs\lambda^s}) + \deg\mathcal{V}_{s+1}\,\deg\rho_{\bfs\lambda^s}\,\log(n+2).
 \end{gather*}
By \eqref{eq: app: widehat_v_s_height} we have
$\widehat{h}(\mathcal{V}_{s+1})\in \SO \bigl(d^s(h+d)\bigr)$.
Therefore, by Lemma \ref{lema:_bound_height_resultant} we conclude
that
 $$
\deg \mathcal{W}_{\bfs\lambda^s} \in \mathcal{O}(nd^{3s+1}), \quad
\widehat{h}(\mathcal{W}_{\bfs\lambda^s}) \in \SO
\bigl(nd^{3s}(h+d)\bigr).
 $$
Combining these estimates with \eqref{B_s_lambda_s_deg_estimate} and
\eqref{B_s_lambda_s_height_estimate}, and taking into account that
$h(Y_{\ell})\in \SO (r\log d+\log n)$ for all $\ell$, the second
assertion of the lemma easily follows.
\end{proof}

Now we estimate the height of $\beta_{\bfs\lambda^s}$.
\begin{lema}\label{lema: app: b_s_lambda_s_height}
Let $1\leq s \leq r-1$ and assume that $\mathcal{W}_{\bfs
\lambda^s}\neq \emptyset$. Then there exists
$\beta_{\bfs\lambda^s}\in \Z\setminus\{0\}$ as in
(\ref{nullstellensatz_implicit_equation}) with
$h(\beta_{\bfs\lambda^s})\in \SO \bigl(n^3d^{8s+1}(h+rd)\bigr)$.
\end{lema}
\begin{proof}
Let $d_j=\deg f_j$ and $h_j:=h(f_j)$ for $1\leq j \leq s+1$, and
$d_{s+2}:=\deg \rho_{\bfs\lambda^s}$ and
$h_{s+2}:=h(\rho_{\bfs\lambda^s})$. Further, define $d_0:=\deg
B_{\bfs\lambda^s}(Y_1 \klk Y_{n-s-1})$ and $h_0:=
h(B_{\bfs\lambda^s}(Y_1 \klk Y_{n-s-1}))$. Finally, denote
$D:=\prod_{j=1}^{s+2}d_j$ and $H:=\max_{1\leq j \leq s+2} h_j$. By
\cite[Theorem 2]{DaKrSo13}, taking into account that
$\deg\mathbb{A}^n=1$ and $\widehat{h}(\mathbb{A}^n)=0$, it follows
that there exists $\beta_{\bfs\lambda^s}\in \Z\setminus\{0\}$ as in
\eqref{nullstellensatz_implicit_equation} with
$$
h(\beta_{\bfs\lambda^s})\leq
2d_0D\left(\frac{3h_0}{2d_0}+\sum_{\ell=1}^{s+2}\frac{H}{d_{\ell}}+e(n)\right),
$$
where $e(n)\in \SO (n)$. Now, by Lemma
\ref{lema:_bound_height_resultant} we have $h_{s+2}\in \SO
\bigl(nd^{2s-1}(h+d)\bigr)$. Since $H=\max\{h,h_{s+2}\}$, we deduce
that $H\in \SO \bigl(nd^{2s-1}(h+d)\bigr)$. On the other hand, $d_0
\leq \deg B_{\bfs\lambda^s}\in \SO (nd^{3s+1})$ by \eqref{eq: app:
deg_and_height_B_s_lambda_s} and $D\leq d^{s+1}d_{s+2}\in \SO
(nd^{3s+1})$. This implies
\begin{equation}\label{eq: app: b_s_lambda_s_height}
d_0D\left(\sum_{\ell=1}^{s+2}\frac{H}{d_{\ell}}+e(n) \right) \in \SO
\bigl(n^3d^{8s+1}(h+d)\bigr).
\end{equation}
Next, since $h(\bfs\lambda^s)\leq h({\sf a})$ for all $s$, by
\cite[Lemma 2.37 (3)]{DaKrSo13} we have
$$
h_0\leq h(B_{\bfs\lambda^s})+ \deg B_{\bfs\lambda^s}\bigl(h({\sf a})
+ \log (n-s) + \log(n+1)\bigr).
$$
Combining this with \eqref{eq: app: deg_and_height_B_s_lambda_s} and
Remark \ref{rem: app: max h_a_h_b_estimate} we deduce that $h_0\in
\SO \bigl(nd^{3s}(h+rnd)\bigr)$. Hence $Dh_0\in \SO
\bigl(n^2d^{6s+1}(h+rnd)\bigr)$ which, together with \eqref{eq: app:
b_s_lambda_s_height}, proves the lemma.
\end{proof}

Now we are finally able to estimate the height of ${\sf
L}_{\bfs\lambda^s}(\bfs p^{s+1})$.
\begin{coro}\label{coro: app: L_s_lambda_s_p_s+1_height}
For $1\leq s \leq r-1$, it holds that $h\bigl({\sf
L}_{\bfs\lambda^s}(\bfs p^{s+1})\bigr) \in \SO
\bigl(n^3d^{8s+1}(h+rd)\bigr)$.
\end{coro}
\begin{proof}
Observe that
$ h\bigl({\sf L}_{\bfs\lambda^s}(\bfs p^{s+1})\bigr)=
h(\mu_{\bfs\lambda^s})$ for $\mathcal{W}_{\bfs\lambda^s}=\emptyset,
$
and $h\bigl({\sf L}_{\bfs\lambda^s}(\bfs p^{s+1})\bigr)=
h(\beta_{\bfs\lambda^s}) + h\bigl(B_{\bfs\lambda^s}(\bfs
p^{s+1})\bigr)$ for $\mathcal{W}_{\bfs\lambda^s}\neq \emptyset. $
Since $h(\bfs p^{s+1})\leq h({\sf b})$, by \cite[Lemma 2.37
(3)]{DaKrSo13} we have
$$
h\bigl(B_{\bfs\lambda^s}(\bfs p^{s+1})\bigr)  \leq
h(B_{\bfs\lambda^s}) + \deg B_{\bfs\lambda^s}\,\bigl(h({\sf b}) +
\log(n-s)\bigr).
$$
This inequality, Remark \ref{rem: app: max h_a_h_b_estimate} and
\eqref{eq: app: deg_and_height_B_s_lambda_s} imply
$h\bigl(B_{\bfs\lambda^s}(\bfs p^{s+1})\bigr)\in \SO
\bigl(nd^{3s}(h+rnd)\bigr)$. Comparing this with \eqref{eq: app:
height_mu_s_lambda_s} and Lemma \ref{lema: app: b_s_lambda_s_height}
yields the estimate of the lemma.
\end{proof}

As a consequence of Lemma \ref{lema: app: height_M_s_lambda_s_p_s}
and Corollary \ref{coro: app: L_s_lambda_s_p_s+1_height} we are able
to estimate the height of the multiple $\mathfrak{N}$ of all the
unlucky primes.
\begin{teo}\label{th: app: mathfrak_N_height}
The integer $\mathfrak{N}$ of \eqref{eq: app: frak_N_def} satisfies
$h(\mathfrak{N})\in \SO \bigl(n^3d^{8r-7}(h+rd)\bigr)$.
\end{teo}
\begin{proof}
Note that $h(\det \bfs \lambda)\leq \log(n!) +nh({\sf a})\in \SO
(rn)$. This, together Lemma with \ref{lema: app:
height_M_s_lambda_s_p_s} and Corollary \ref{coro: app:
L_s_lambda_s_p_s+1_height}, readily implies the theorem.
\end{proof}

\providecommand{\bysame}{\leavevmode\hbox to3em{\hrulefill}\thinspace}
\providecommand{\MR}{\relax\ifhmode\unskip\space\fi MR }
\providecommand{\MRhref}[2]{%
  \href{http://www.ams.org/mathscinet-getitem?mr=#1}{#2}
}
\providecommand{\href}[2]{#2}

\end{document}